\documentclass[11pt]{article}

\usepackage[hidelinks]{hyperref}
\hypersetup{
  colorlinks   = true, 
  urlcolor     = blue, 
  linkcolor    = blue, 
  citecolor   = red 
}
\usepackage{amsmath,amsthm,amssymb}
\usepackage{graphicx}
\usepackage{bbm}
\newcommand{\sy}{\boldsymbol{\Psi}}
\newcommand{\py}{\boldsymbol{\Phi}}
\newcommand{\T}{\mathbb{T}^N}

\usepackage{xcolor}
\usepackage[margin=1in]{geometry} 
\usepackage{enumerate,mathtools,mathrsfs,bm,graphicx}
\usepackage[numbers, super]{natbib}
\usepackage[hidelinks]{hyperref}
\newcommand{\N}{\mathbb{N}}									
\newcommand{\Z}{\mathbb{Z}}

\newcommand{\R}{\mathbb{R}}

\newcommand{\vertiii}[1]{{\left\vert\kern-0.25ex\left\vert\kern-0.25ex\left\vert #1 
    \right\vert\kern-0.25ex\right\vert\kern-0.25ex\right\vert}}
    
\newcommand{\overbar}[1]{\mkern 1.5mu\overline{\mkern-1.5mu#1\mkern-1.5mu}\mkern 1.5mu}
\newcommand{\inner}[2]{\langle #1, #2 \rangle}
\DeclarePairedDelimiter\abs{\lvert}{\rvert}					
\DeclarePairedDelimiter\norm{\lVert}{\rVert}				

\newtheorem{theorem}{Theorem}[section]
\newtheorem{corollary}{Corollary}[theorem]

\newtheorem{lemma}[theorem]{Lemma}
\newtheorem{proposition}[theorem]{Proposition}
\newtheorem*{remark}{Remark}
\newtheorem{definition}[theorem]{Definition}

\newtheorem{assumption}[theorem]{Assumption}

\begin{document}
	\title{On the Navier-Stokes Equations with Stochastic Lie Transport}
	\author{Daniel Goodair \qquad \qquad Dan Crisan}
	\date{\today} 
	\maketitle
\setcitestyle{numbers}	
	\begin{abstract}
	    We prove the existence and uniqueness of maximal solutions to the 3D SALT (Stochastic Advection by Lie Transport, [\cite{Holm}]) Navier-Stokes Equation in velocity and vorticity form, on the torus and the bounded domain respectively. The current work partners the paper [\cite{Goodair abs}] as an application of the abstract framework presented there, justifying the results first announced in [\cite{Goodair conference}]. In particular this represents the first well-posedness result for a fluid equation perturbed by a general transport type noise on a bounded domain.
	\end{abstract}

\tableofcontents
\thispagestyle{empty}
\newpage

\setcounter{page}{1}
\addcontentsline{toc}{section}{Introduction}

\section*{Introduction}

The theoretical analysis of Stochastic Navier-Stokes Equations dates back to the work of Bensoussan and Temam [\cite{Benny}] in 1973, where the problem of existence of solutions is addressed in the presence of a random forcing term. The well-posedness question for additive and multiplicative noise has since seen significant developments, for example through the works [\cite{HoltzZiane}, \cite{Nav1}, \cite{Nav2}, \cite{Nav3}, \cite{Nav4}, \cite{Nav5}] and references therein. The interest in this problem has seen developments into analytical properties of these solutions, particularly along the lines of ergodicity, which can be seen in [\cite{Nav6}, \cite{Nav7}, \cite{Nav8}, \cite{Nav9}, \cite{Nav10}]. In the present work our concern is the Navier-Stokes Equations with Stochastic Lie Transport, derived through the principle of Stochastic Advection by Lie Transport (SALT) introduced in [\cite{Holm}]. We consider the equation
\begin{equation} \label{number2equation}
    u_t - u_0 + \int_0^t\mathcal{L}_{u_s}u_s\,ds - \nu\int_0^t \Delta u_s\, ds + \int_0^t Bu_s \circ d\mathcal{W}_s + \int_0^t \nabla \rho_sds = 0
\end{equation}
where $u$ represents the fluid velocity, $\rho$ the pressure, $\mathcal{W}$ is a Cylindrical Brownian Motion, $\mathcal{L}$ represents the nonlinear term and $B$ is a first order differential operator (the SALT Operator) formally addressed in Subsection \ref{subsection the salt operator}. Intrinsic to this stochastic methodology is that $B$ is defined relative to a collection of functions $(\xi_i)$ which physically represent spatial correlations. These $(\xi_i)$ can be determined at coarse-grain resolutions from finely resolved numerical simulations, and mathematically are derived as eigenvectors of a velocity-velocity correlation matrix (see [\cite{Cotter1}, \cite{Cotter2}, \cite{Cotter3}]). We pose the equation (\ref{number2equation}) in $N$ dimensions for $N=2,3$ and impose the divergence free constraint on $u$. We shall consider the problem both over the torus $\T$ and a smooth bounded domain $\mathscr{O} \subset \R^N$. In the case of the torus we supplement the equation with the zero-average condition (as is classical), whilst for the bounded domain we impose the boundary condition \begin{equation}\label{boundary condition}u \cdot n = 0, \qquad w=0\end{equation} where $n$ represents the outwards unit normal at the boundary, and $w$ the fluid vorticity. These are the so called Lions boundary conditions, considered in [\cite{Lions}] and shown to be a particular case of the Navier boundary conditions in [\cite{Kelli}] (note that this is done in $2D$ and our analysis will take place in $3D$, though with some clarification it continues to hold in $2D$). The significance of such a boundary condition is well documented in that work by Kelliher, and can be seen in other works such as [\cite{Kelli2}] where the boundary layer is explicitly addressed. The precise mathematical interpretation of these conditions, and the operators of (\ref{number2equation}), are explicated in Subsection \ref{functional framework subsection}. A complete derivation of this equation can be found in [\cite{street}].\\

This work continues the theoretical development of fluid models perturbed by a transport type noise, which has seen significant developments since the seminal works [\cite{Holm}, \cite{Memin}]. This paper partners that of [\cite{Goodair abs}] where we showed the existence and uniqueness of maximal solutions to Stochastic Partial Differential Equations satisfying an abstract framework, built to cope with a general transport type noise as we see in (\ref{number2equation}). The significance of such equations in modelling, numerical schemes and data assimilation is reviewed there, along with the theoretical developments of these equations. We only draw particular attention here to the Navier-Stokes Equations, and results on a bounded domain. The Navier-Stokes Equations have been studied with transport type noise, for example in the works [\cite{Franco1}, \cite{Arnaud1}, \cite{Romeo1}], though typically solutions are analytically weak and where strong solutions are considered major concessions in the noise are made. In these cases a cancellation property is evident in the noise term, so that in energy methods the differential operator is not felt. These difficulties have been addressed on the torus, in the likes of the papers [\cite{Crisan1}, \cite{Crisan2}] and those further addressed in [\cite{Goodair abs}], but extending a control of this noise term to a bounded domain remains open. Indeed the situation becomes more complex in the presence of viscosity, as energy methods require non-standard Sobolev inner products to conduct the required integration by parts in the bounded domain, rendering control on the noise terms completely out of familiarity. The problem of analytically strong solutions to fluid equations perturbed by a transport type noise in the bounded domain has been considered in [\cite{Brez}], though the authors assume that the gradient dependency is of a small enough size to be directly controlled and that the noise terms are traceless under Leray Projection; such assumptions are designed to circumvent the technical difficulties of a first order noise operator on a bounded domain. Our results of Section \ref{section vorticity form} pertaining to a bounded domain thus represent a substantial improvement on the literature.\\

The goal of this paper is to apply the abstract framework established in [\cite{Goodair abs}], together providing a rigorous justification of the results first announced in [\cite{Goodair conference}] and extending them to the vorticity form of equation (\ref{number2equation}) on a bounded domain. In Section \ref{section preliminaries} we establish the stochastic and functional framework necessary to understand (\ref{number2equation}), along with fundamental properties of the operators involved. In Section \ref{section velocity} we make precise how equation (\ref{number2equation}) fits into the framework of [\cite{Goodair abs}], as a problem posed on the torus $\T$. In Section \ref{section vorticity form} we consider the vorticity form of equation (\ref{number2equation}) as a problem posed on a bounded domain of $\R^3$. We again justify the assumptions in [\cite{Goodair abs}] to prove existence and uniqueness of maximal solutions to this equation. Additional details for the proofs are given in the Appendices Section \ref{section appendices}, along with the results of the partnering paper [\cite{Goodair abs}].

\section{Preliminaries} \label{section preliminaries}

\subsection{Elementary Notation}

In the following $\mathcal{O}$ represents a subset of $\R^N$, and will be used throughout the paper to represent either the $N-$dimensional torus $\T$ or a smooth bounded domain $\mathscr{O} \subset \R^N$ for $N=2,3$. In the paper's duration we consider Banach Spaces as measure spaces equipped with the Borel $\sigma$-algebra, and will on occasion use $\lambda$ to represent the Lebesgue Measure.

\begin{definition} \label{definition of spaces}
Let $(\mathcal{X},\mu)$ denote a general measure space, $(\mathcal{Y},\norm{\cdot}_{\mathcal{Y}})$ and $(\mathcal{Z},\norm{\cdot}_{\mathcal{Z}})$ be Banach Spaces, and $(\mathcal{U},\inner{\cdot}{\cdot}_{\mathcal{U}})$, $(\mathcal{H},\inner{\cdot}{\cdot}_{\mathcal{H}})$ be general Hilbert spaces. $\mathcal{O}$ is equipped with Euclidean norm.
\begin{itemize}
    \item $L^p(\mathcal{X};\mathcal{Y})$ is the usual class of measurable $p-$integrable functions from $\mathcal{X}$ into $\mathcal{Y}$, $1 \leq p < \infty$, which is a Banach space with norm $$\norm{\phi}_{L^p(\mathcal{X};\mathcal{Y})}^p := \int_{\mathcal{X}}\norm{\phi(x)}^p_{\mathcal{Y}}\mu(dx).$$ The space $L^2(\mathcal{X}; \mathcal{Y})$ is a Hilbert Space when $\mathcal{Y}$ itself is Hilbert, with the standard inner product $$\inner{\phi}{\psi}_{L^2(\mathcal{X}; \mathcal{Y})} = \int_{\mathcal{X}}\inner{\phi(x)}{\psi(x)}_\mathcal{Y} \mu(dx).$$ In the case $\mathcal{X} = \mathcal{O}$ and $\mathcal{Y} = \R^N$ note that $$\norm{\phi}_{L^2(\mathcal{O};\R^N)}^2 = \sum_{l=1}^N\norm{\phi^l}^2_{L^2(\mathcal{O};\R)}$$ for the component mappings $\phi^l: \mathcal{O} \rightarrow \R$. 
    
\item $L^{\infty}(\mathcal{X};\mathcal{Y})$ is the usual class of measurable functions from $\mathcal{X}$ into $\mathcal{Y}$ which are essentially bounded, which is a Banach Space when equipped with the norm $$ \norm{\phi}_{L^{\infty}(\mathcal{X};\mathcal{Y})} := \inf\{C \geq 0: \norm{\phi(x)}_Y \leq C \textnormal{ for $\mu$-$a.e.$ }  x \in \mathcal{X}\}.$$
    
    \item $L^{\infty}(\mathcal{O};\R^N)$ is the usual class of measurable functions from $\mathcal{O}$ into $\R^N$ such that $\phi^l \in L^{\infty}(\mathcal{O};\R)$ for $l=1,\dots,N$, which is a Banach Space when equipped with the norm $$ \norm{\phi}_{L^{\infty}}:= \sup_{l \leq N}\norm{\phi^l}_{L^{\infty}(\mathcal{O};\R)}.$$
    
      \item $C(\mathcal{X};\mathcal{Y})$ is the space of continuous functions from $\mathcal{X}$ into $\mathcal{Y}$.
      
    \item $C^m(\mathcal{O};\R)$ is the space of $m \in \N$ times continuously differentiable functions from $\mathcal{O}$ to $\R$, that is $\phi \in C^m(\mathcal{O};\R)$ if and only if for every $N$ dimensional multi index $\alpha = \alpha_1, \dots, \alpha_N$ with $\abs{\alpha}\leq m$, $D^\alpha \phi \in C(\mathcal{O};\R)$ where $D^\alpha$ is the corresponding classical derivative operator $\partial_{x_1}^{\alpha_1} \dots \partial_{x_N}^{\alpha_N}$.
    
    \item $C^\infty(\mathcal{O};\R)$ is the intersection over all $m \in \N$ of the spaces $C^m(\mathcal{O};\R)$.
    
    \item $C^m_0(\mathscr{O};\R)$ for $m \in \N$ or $m = \infty$ is the subspace of $C^m(\mathscr{O};\R)$ of functions which have compact support.
    
    \item $C^m(\mathcal{O};\R^N), C^m_0(\mathscr{O};\R^N)$ for $m \in \N$ or $m = \infty$ is the space of functions from $\mathcal{O}$ to $\R^N$ whose $N$ component mappings each belong to $C^m(\mathcal{O};\R), C^m_0(\mathscr{O};\R)$.
    
        \item $W^{m,p}(\mathcal{O}; \R)$ for $1 \leq p < \infty$ is the sub-class of $L^p(\mathcal{O}, \R)$ which has all weak derivatives up to order $m \in \N$ also of class $L^p(\mathcal{O}, \R)$. This is a Banach space with norm $$\norm{\phi}^p_{W^{m,p}(\mathcal{O}, \R)} := \sum_{\abs{\alpha} \leq m}\norm{D^\alpha \phi}_{L^p(\mathcal{O}; \R)}^p$$ where $D^\alpha$ is the corresponding weak derivative operator. In the case $p=2$ the space $W^{m,2}(U, \R)$ is Hilbert with inner product $$\inner{\phi}{\psi}_{W^{m,2}(\mathcal{O}; \R)} := \sum_{\abs{\alpha} \leq m} \inner{D^\alpha \phi}{D^\alpha \psi}_{L^2(\mathcal{O}; \R)}.$$
    
    \item $W^{m,\infty}(\mathcal{O};\R)$ for $m \in \N$ is the sub-class of $L^\infty(\mathcal{O}, \R)$ which has all weak derivatives up to order $m \in \N$ also of class $L^\infty(\mathcal{O}, \R)$. This is a Banach space with norm $$\norm{\phi}_{W^{m,\infty}(\mathcal{O}, \R)} := \sup_{\abs{\alpha} \leq m}\norm{D^{\alpha}\phi}_{L^{\infty}(\mathcal{O};\R^N)}.$$
    
        \item $W^{m,p}(\mathcal{O}; \R^N)$ for $1 \leq p < \infty$ is the sub-class of $L^p(\mathcal{O}, \R^N)$ which has all weak derivatives up to order $m \in \N$ also of class $L^p(\mathcal{O}, \R^N)$. This is a Banach space with norm $$\norm{\phi}^p_{W^{m,p}(\mathcal{O}, \R^N)} := \sum_{l=1}^N\norm{\phi^l}_{W^{m,p}(\mathcal{O}; \R)}^p.$$ In the case $p=2$ the space $W^{m,2}(\mathcal{O}, \R^N)$ is Hilbertian with inner product $$\inner{\phi}{\psi}_{W^{m,2}(\mathcal{O}; \R^N)} := \sum_{l=1}^N \inner{\phi^l}{\psi^l}_{W^{m,2}(\mathcal{O}; \R)}.$$
    
          \item $W^{m,\infty}(\mathcal{O}; \R^N)$ for is the sub-class of $L^\infty(\mathcal{O}, \R^N)$ which has all weak derivatives up to order $m \in \N$ also of class $L^\infty(\mathcal{O}, \R^N)$. This is a Banach space with norm $$\norm{\phi}_{W^{m,\infty}(\mathcal{O}, \R^N)} := \sup_{l \leq N}\norm{\phi^l}_{W^{m,\infty}(\mathcal{O}; \R)}.$$
          
    \item  $\dot{L}^2(\T;\R^N)$ is the subset of $L^2(\T;\R^N)$ of functions $\phi$ such that $$\int_{\T}\phi \ d\lambda = 0.$$    
    
    \item $\dot{W}^{m,2}(\T;\R^N)$ is simply the intersection $W^{m,2}(\T;\R^N) \cap \dot{L}^2(\T;\R^N)$.
          
    \item $W^{m,p}_0(\mathscr{O};\R), W^{m,p}_0(\mathscr{O};\R^N)$ for $m \in N$ and $1 \leq p \leq \infty$ is the closure of $C^\infty_0(\mathscr{O};\R),C^\infty_0(\mathscr{O};\R^N)$ in $W^{m,p}(\mathscr{O};\R), W^{m,p}(\mathscr{O};\R^N)$.
    
    \item $\mathscr{L}(\mathcal{Y};\mathcal{Z})$ is the space of bounded linear operators from $\mathcal{Y}$ to $\mathcal{Z}$. This is a Banach Space when equipped with the norm $$\norm{F}_{\mathscr{L}(\mathcal{Y};\mathcal{Z})} = \sup_{\norm{y}_{\mathcal{Y}}=1}\norm{Fy}_{\mathcal{Z}}$$ and is simply the dual space $\mathcal{Y}^*$ when $\mathcal{Z}=\R$, with operator norm $\norm{\cdot}_{\mathcal{Y}^*}.$
    
     \item $\mathscr{L}^2(\mathcal{U};\mathcal{H})$ is the space of Hilbert-Schmidt operators from $\mathcal{U}$ to $\mathcal{H}$, defined as the elements $F \in \mathscr{L}(\mathcal{U};\mathcal{H})$ such that for some basis $(e_i)$ of $\mathcal{U}$, $$\sum_{i=1}^\infty \norm{Fe_i}_{\mathcal{H}}^2 < \infty.$$ This is a Hilbert space with inner product $$\inner{F}{G}_{\mathscr{L}^2(\mathcal{U};\mathcal{H})} = \sum_{i=1}^\infty \inner{Fe_i}{Ge_i}_{\mathcal{H}}$$ which is independent of the choice of basis.

\end{itemize}

\end{definition}
We will consider a partial ordering on the $N-$dimensional multi-indices by $\alpha \leq \beta$ if and only if for all $l =1, \cdots \N$ we have that $\alpha_l \leq \beta_l$. We extend this to notation $<$ by 
$\alpha < \beta$ if and only if $\alpha \leq \beta$ and for some $l = 1, \dots, N$, $\alpha_l < \beta_l$.

\subsection{Stochastic Framework}

We work with a fixed filtered probability space $(\Omega,\mathcal{F},(\mathcal{F}_t), \mathbb{P})$ satisfying the usual conditions of completeness and right continuity. We take $\mathcal{W}$ to be a cylindrical Brownian Motion over some Hilbert Space $\mathfrak{U}$ with orthonormal basis $(e_i)$. Recall ([\cite{goodair2022stochastic}], Subsection 1.4) that $\mathcal{W}$ admits the representation $\mathcal{W}_t = \sum_{i=1}^\infty e_iW^i_t$ as a limit in $L^2(\Omega;\mathfrak{U}')$ whereby the $(W^i)$ are a collection of i.i.d. standard real valued Brownian Motions and $\mathfrak{U}'$ is an enlargement of the Hilbert Space $\mathfrak{U}$ such that the embedding $J: \mathfrak{U} \rightarrow \mathfrak{U}'$ is Hilbert-Schmidt and $\mathcal{W}$ is a $JJ^*-$cylindrical Brownian Motion over $\mathfrak{U}'$. Given a process $F:[0,T] \times \Omega \rightarrow \mathscr{L}^2(\mathfrak{U};\mathscr{H})$ progressively measurable and such that $F \in L^2\left(\Omega \times [0,T];\mathscr{L}^2(\mathfrak{U};\mathscr{H})\right)$, for any $0 \leq t \leq T$ we define the stochastic integral $$\int_0^tF_sd\mathcal{W}_s:=\sum_{i=1}^\infty \int_0^tF_s(e_i)dW^i_s$$ where the infinite sum is taken in $L^2(\Omega;\mathscr{H})$. We can extend this notion to processes $F$ which are such that $F(\omega) \in L^2\left( [0,T];\mathscr{L}^2(\mathfrak{U};\mathscr{H})\right)$ for $\mathbb{P}-a.e.$ $\omega$ via the traditional localisation procedure. In this case the stochastic integral is a local martingale in $\mathscr{H}$. \footnote{A complete, direct construction of this integral, a treatment of its properties and the fundamentals of stochastic calculus in infinite dimensions can be found in [\cite{goodair2022stochastic}] Section 1.}

\subsection{Functional Framework} \label{functional framework subsection}

We now recap the classical functional framework for the study of the deterministic Navier-Stokes Equation, using freely the notation introduced in \ref{definition of spaces}. As promised we now formally define the operator $\mathcal{L}$ as well as the divergence-free and Lions boundary conditions. Firstly though we briefly comment on the pressure term $\nabla \rho$, which will not play any role in our analysis. $\rho$ does not come with an evolution equation and is simply chosen to ensure the incompressibility (divergence-free) condition; moreover we will ignore this term via a suitable projection (in Section 3 we even consider a different form of the equation) and treat the projected equation, with the understanding to append a pressure to it later. This procedure is well discussed in [\cite{Robinson}] Section 5 and [\cite{vorticity}], and an explicit form for the pressure for the SALT Euler Equation is given in [\cite{street}] Subsection 3.3.\\ 

The mapping $\mathcal{L}$ is defined for sufficiently regular functions $f,g:\mathcal{O} \rightarrow \R^N$ by $$\mathcal{L}_fg:= \sum_{j=1}^Nf^j\partial_jg.$$ Here and throughout the text we make no notational distinction between differential operators acting on a vector valued function or a scalar valued one; that is, we understand $\partial_jg$ by its component mappings \begin{equation} \label{distinction} (\partial_lg)^l := \partial_jg^l.\end{equation} We can immediately give some clarification as to 'sufficiently regular'.

\begin{lemma} \label{continuityofnonlinear}
    For any $m \in \N$, the mapping $\mathcal{L}: W^{m+1,2}(\mathcal{O};\R^N) \rightarrow W^{m,2}(\mathcal{O};\R^N)$ defined by $$f \mapsto \mathcal{L}_ff$$ is continuous.
\end{lemma}

\begin{proof}
See Appendix I, \ref{Appendix I}.
\end{proof}

\begin{lemma} \label{basic nonlinear bound}
    There exists a constant $c$ such that for any $f\in W^{2,2}(\mathcal{O};\R^N)$ and $g\in W^{1,2}(\mathcal{O};\R^N)$, we have the bounds
    \begin{align} \label{first begin align}
        \norm{\mathcal{L}_{f}{g}} + \norm{\mathcal{L}_{g}{f}}&\leq c\norm{g}_{W^{1,2}}\norm{f}_{W^{2,2}}.
        \end{align}
        If $f$ has the additional regularity $f \in W^{3,2}(\mathcal{O};\R^N)$ then
        \begin{equation} \label{second begin align}
            \norm{\mathcal{L}_{g}{f}}_{W^{1,2}} \leq c\norm{g}_{W^{1,2}}\norm{f}_{W^{3,2}}
        \end{equation}
        whilst for $g \in W^{2,2}(\mathcal{O};\R^N)$ we have that 
         \begin{equation} \label{third begin align}
            \norm{\mathcal{L}_{g}{f}}_{W^{1,2}} \leq c\norm{g}_{W^{2,2}}\norm{f}_{W^{2,2}}.
        \end{equation}
\end{lemma}

\begin{proof}
See Appendix I, \ref{Appendix I}.
\end{proof}

As for the divergence-free condition we mean a function $f$ such that the property $$\textnormal{div}f := \sum_{j=1}^N \partial_jf^j = 0$$ holds. We require this property and the boundary condition to hold for our solution $u$ at all times for which it is a solution, though there is some ambiguity in how we understand these conditions for a solution $u$ which need not be defined pointwise everywhere on $\bar{\mathcal{O}}$. Moreover we shall understand these conditions in their traditional weak sense, that is for weak derivatives $\partial_j$ so $\sum_{j=1}^N \partial_jf^j = 0$ holds as an identity in $L^2(\mathcal{O};\R)$ whilst the boundary condition $w=0$ is understood as each component mapping $w^j$ having zero trace (recall e.g. [\cite{evans}] that $f^j \in W^{1,2}(\mathscr{O};\R) \cap C(\bar{\mathscr{O}};\R)$ has zero trace if and only if $f^j(x) = 0$ for all $x \in \partial \mathscr{O}$). The boundary condition $u \cdot n=0$ is a little more technical and not particularly relevant for our analysis, so we defer to [\cite{Robinson}] Lemma 2.12 for a precise formulation in the weak sense. We impose these conditions by incorporating them into the function spaces where our solution takes value, and give separate definitions for the cases of the Torus and the bounded domain.

\begin{definition}
We define $C^{\infty}_{0,\sigma}(\mathscr{O};\R^N)$ as the subset of $C^{\infty}_0(\mathscr{O};\R^N)$ of functions which are divergence-free. $L^2_\sigma(\mathcal{O};\R^N)$ is defined as the completion of $C^{\infty}_{0,\sigma}(\mathscr{O};\R^N)$ in $L^2(\mathscr{O};\R^N)$, whilst we introduce $W^{1,2}_\sigma (\mathscr{O};\R^N)$ as the intersection of $W^{1,2}_0(\mathscr{O};\R^N)$ with $L^2_\sigma(\mathscr{O};\R^N)$ and $W^{2,2}_{\sigma} (\mathscr{O};\R^N)$ as the intersection of $W^{2,2}(\mathscr{O};\R^N)$ with $W^{1,2}_\sigma(\mathscr{O};\R^N)$.
\end{definition}

For the definitions on the Torus we make explicit reference to the Fourier decomposition. We do this just to define the spaces, after which we make an analogy with the corresponding spaces on the bounded domain and need not rely on any techniques exploiting the Fourier decomposition. To this end recall that any function $f \in L^2(\T;\R^N)$ admits the representation \begin{equation} \label{fourier rep}f(x) = \sum_{k \in \mathbb{Z}^N}f_ke^{ik\cdot x}\end{equation} whereby each $f_k \in \mathbb{C}^N$ is such that $f_k = \overbar{f_{-k}}$ and the infinite sum is defined as a limit in $L^2(\T;\R^N)$, see e.g. [\cite{Robinson}] Subsection 1.5 for details.

\begin{definition}
We define $L^2_{\sigma}(\T;\R^N)$ as the subset of $\dot{L}^2(\T;\R^N)$ of functions $f$ whereby for all $k \in \mathbbm{Z}^N$, $k \cdot f_k = 0$ with $f_k$ as in (\ref{fourier rep}). For general $m \in \N$ we introduce $W^{m,2}_{\sigma}(\T;\R^N)$ as the intersection of $W^{m,2}(\T;\R^N)$ respectively with $L^2_{\sigma}(\T;\R^N)$.
\end{definition}

\begin{remark}
$L^2_{\sigma}(\T;\R^N)$ is closed in the $L^2(\T;\R^N)$ norm. The same is trivially true of $L^2_{\sigma}(\mathscr{O};\R^N)$. Furthermore as the $W^{1,2}(\mathcal{O};\R^N)$ norm induces a coarser topology, the space $W^{1,2}_{\sigma}(\mathcal{O};\R^N)$ is closed in the $W^{1,2}(\mathcal{O};\R^N)$ norm and similarly for $W^{2,2}_{\sigma}(\mathcal{O};\R^N)$. Thus $L^2_{\sigma}(\mathcal{O};\R^N)$, $W^{1,2}_{\sigma}(\mathcal{O};\R^N)$ and $W^{2,2}_{\sigma}(\mathcal{O};\R^N)$ are Hilbert Spaces in their respective inner products.
\end{remark}

\begin{lemma}
$W^{1,2}_{\sigma}(\T;\R^N)$ is precisely the subspace of $W^{1,2}(\T;\R^N)$ consisting of zero-average divergence free functions.
\end{lemma}

\begin{proof}
    Note that every $f \in W^{1,2}(\T;\R^N)$ admits the representation (\ref{fourier rep}) but for convergence holding in $W^{1,2}(\T;\R^N)$. Moreover $$f^j(x) = \sum_{k \in \mathbb{Z}^N}f^j_ke^{ik\cdot x}$$ with convergence in $W^{1,2}(\T;\R)$ so $$\partial_j f^j(x) = \sum_{k \in \mathbb{Z}^N}\partial_j\left(f^j_ke^{ik\cdot x}\right) = \sum_{k \in \mathbb{Z}^N}if^j_kk^je^{ik\cdot x}$$ where the infinite sum is in $L^2(\T;\R)$; with the property $f_k^j = \overbar{f^j_{-k}}$, a direct computation shows that $$if_k^j k^j = \overbar{if^j_{-k}(-k^j)}$$ verifying once more that this is indeed real valued. Furthermore the divergence of $f$ is given by $$\sum_{k \in \mathbb{Z}^N}i(f_k\cdot k)e^{ik\cdot x}$$ so we see that this is zero if and only if $f_k \cdot k = 0$ for all $k \in \Z^N$. The result follows.
\end{proof}

\begin{lemma} \label{W12sigmacharacter}
$W^{1,2}_{\sigma}(\mathscr{O};\R^N)$ is precisely the subspace of $W^{1,2}_0(\mathscr{O};\R^N)$ consisting of divergence free functions. Moreover, $W^{1,2}_{\sigma}(\mathscr{O};\R^N)$ is the completion of $C^{\infty}_{0,\sigma}(\mathscr{O};\R^N)$ in $W^{1,2}(\mathscr{O};\R^N)$. 
\end{lemma}

\begin{proof}
We claim at first that any given $f \in W^{1,2}_{\sigma}(\mathscr{O};\R^N)$ is divergence-free. In [\cite{Robinson}] Lemma 2.11 it is shown that $f$ satisfies the property $$\inner{f}{\nabla \phi}=0$$ for every $\phi \in C^{\infty}_0(\mathscr{O};\R)$, or equivalently that $$\sum_{j=1}^N \inner{f^j}{\partial_j \phi}_{L^2(\mathscr{O};\R)} = 0$$ for such $\phi$. We can use the compact support of $\phi$ and $W^{1,2}(\mathscr{O};\R^N)$ regularity of $f$ to carry out an integration by parts, taking the sum through the inner product to see that $$\left\langle \sum_{j=1}^N \partial_jf^j , \phi \right\rangle_{L^2(\mathscr{O};\R)}= 0.$$ Observing that $\phi \in C^{\infty}_0(\mathscr{O};\R)$ was arbitrary and using the density of this space in $L^2(\mathscr{O};\R)$, we deduce that the divergence of $f$ is zero as an element of $L^2(\mathscr{O};\R)$ and so $f$ is divergence-free. Working in the reverse direction now, we suppose that $g \in W^{1,2}_0(\mathscr{O};\R^N)$ is divergence-free with the intent to show that $g \in W^{1,2}_{\sigma}(\mathscr{O};\R^N)$, which is to say $g \in L^2_{\sigma}(\mathscr{O};\R^N)$. For this we simply refer to Lemma 2.14 again of [\cite{Robinson}], noting that $g$ is of zero trace so satisfies the boundary condition and the divergence-free property implies all 'weak divergence-free' notions of [\cite{Robinson}].\\
The final statement to prove is the characterisation of $W^{1,2}_{\sigma}(\mathscr{O};\R^N)$ as the completion of $C^{\infty}_{0,\sigma}(\mathscr{O};\R^N)$ in $W^{1,2}(\mathscr{O};\R^N)$. We defer this to [\cite{temam NS}] Theorem 1.6, using the representation of $W^{1,2}_{\sigma}(\mathscr{O};\R^N)$ just proved.
\end{proof}

\begin{remark}
The space $W^{1,2}_{\sigma}(\mathcal{O};\R^N)$ thus incorporates the divergence-free and zero-average/zero-trace condition. 
\end{remark}

With these spaces in place we can define the Leray Projector, which is the projection alluded to when discussing the pressure term.

\begin{definition}
The Leray Projector $\mathcal{P}$ is defined as the orthogonal projection in $L^2(\mathcal{O};\R^N)$ onto $L^2_{\sigma}(\mathcal{O};\R^N)$.
\end{definition}

We immediately give two important properties of this projection.

\begin{proposition} \label{continuityofP}
    For any $m \in \N$, $\mathcal{P}$ is continuous as a mapping $\mathcal{P}: W^{m,2}(\mathcal{O};\R^N) \rightarrow W^{m,2}(\mathcal{O};\R^N)$.
\end{proposition}

\begin{proof}
    See [\cite{temam NS}] Remark 1.6.
\end{proof}

In fact, the complement space of $L^2_{\sigma}(\mathcal{O};\R^N)$ can be characterised. This is the so called Helmholtz-Weyl decomposition.

\begin{proposition} \label{HWDecomp}
    Define the space $$L^{2, \perp}_{\sigma}(\mathcal{O};\R^N):= \{\psi \in L^{2}(\mathcal{O};\R^N): \psi = \nabla g \textnormal{ for some } g \in W^{1,2}(\mathcal{O};\R)   \}.$$ Then indeed $L^{2, \perp}_{\sigma}(\mathcal{O};\R^N)$ is orthogonal to $L^{2}_{\sigma}(\mathcal{O};\R^N)$ in $L^{2}(\mathcal{O};\R^N)$, i.e. for any $\phi \in L^{2}_{\sigma}(\mathcal{O};\R^N)$ and $\psi \in L^{2, \perp}_{\sigma}(\mathcal{O};\R^N)$ we have that $$\inner{\phi}{\psi}=0.$$ Moreover, every $f \in L^{2}(\mathscr{O};\R^N)$ has the unique decomposition \begin{equation}\label{first unique}f = \phi + \psi\end{equation} for some $\phi \in L^{2}_{\sigma}(\mathcal{O};\R^N)$, $\psi \in L^{2, \perp}_{\sigma}(\mathcal{O};\R^N)$ and every $f \in L^{2}(\T;\R^N)$ has the unique decomposition \begin{equation} \label{second unique}
        f = \phi + \psi + c
    \end{equation}
    where $\phi \in L^{2}_{\sigma}(\T;\R^N)$, $\psi \in L^{2, \perp}_{\sigma}(\T;\R^N)$ and $c$ is a constant function: that is, there exists $k \in \R^N$ such that each component mapping $c^j$ is identically equal to $k^j$, $j=1, \dots, N$.
\end{proposition}

\begin{proof}
    See [\cite{temam NS}] Theorems 1.4, 1.5 and [\cite{Robinson}] Theorem 2.6.
\end{proof}

\begin{corollary} \label{HWDecomp2}
Every $f \in L^{2}(\mathscr{O};\R^N)$ admits the representation \begin{equation}\label{decomp}f = \mathcal{P}f + \nabla g\end{equation} for some $g\in W^{1,2} (\mathscr{O};\R)$.
\end{corollary}
\begin{corollary}\label{HWDecomp torus}
Every $f \in L^{2}(\T;\R^N)$ admits the representation \begin{equation}\label{decomp torus}f = \mathcal{P}f + \nabla g + c\end{equation} for some $g\in W^{1,2} (\T;\R)$ and constant function $c$.

\end{corollary}

\begin{proof}
   It is an immediate property of the orthogonal projection that $\mathcal{P}f$ is the unique element $\phi \in L^{2}_{\sigma}(\mathcal{O};\R^N)$ of (\ref{first unique}),(\ref{second unique}).
\end{proof}

This representation allows us to establish important relations between the Leray Projector and the operators involved in (\ref{number2equation}), and gives rise to a new operator that will be fundamental in defining the spaces used in the abstract formulations of [\cite{goodair2022stochastic}, \cite{Goodair abs}, \cite{Goodair conference}].

\begin{definition}
The Stokes Operator $A: W^{2,2}(\mathcal{O};\R^N) \rightarrow L^2_{\sigma}(\mathcal{O};\R^N)$ is defined by $A:= -\mathcal{P}\Delta$.
\end{definition}

\begin{remark}
Once more we understand the Laplacian as an operator on vector valued functions in the sense (\ref{distinction}).
\end{remark}

\begin{remark}
From Proposition \ref{continuityofP} we have immediately that for $m \in \{0\} \cup \N$, $A: W^{m+2,2}(\mathcal{O};\R^N) \rightarrow W^{m,2}(\mathcal{O};\R^N)$ is continuous. 
\end{remark}

\begin{lemma} \label{APequalsA}
    $A\mathcal{P}$ is equal to $A$ on $W^{2,2}(\mathcal{O};\R^N)$.
\end{lemma}

\begin{proof}
     We call upon the representations (\ref{decomp}) and (\ref{decomp torus}), so briefly distinguish between the two domains. In the case of $\T$ with representation (\ref{decomp torus}), Proposition \ref{continuityofP} ensures that for $f \in W^{2,2}(\T;\R^N)$, $\mathcal{P}f \in W^{2,2}(\T;\R^N)$ and hence so too is $\nabla g$. We then have $$Af = A\mathcal{P}f + A\nabla g + Ac = A\mathcal{P}f + A\nabla g$$ where the Laplacian of the constant is trivially zero, and this same expression is achieved directly for $f \in W^{2,2}(\mathscr{O};\R^N)$ from (\ref{decomp}) so we work now with general $f \in W^{2,2}(\mathcal{O};\R^N)$. The result would thus be true if $A\nabla g = 0$. This is not difficult to see however, as $$A\nabla g = -\mathcal{P}\Delta \nabla g = -\mathcal{P}\nabla \Delta g = 0$$ owing to the fact that $\Delta g \in W^{1,2}(\mathcal{O};\R^N)$ from $\nabla g \in W^{2,2}(\mathcal{O};\R^N)$ so $\nabla \Delta g \in L^{2, \perp}_{\sigma}(\mathcal{O};\R^N)$ and $\mathcal{P}$ projects to an orthogonal space. 
\end{proof}

As we look to exploit properties of the Stokes Operator, we will rely heavily on the following Proposition. 

\begin{proposition} \label{basis of stokes}
    There exists a collection of functions $(a_k)$, $a_k \in W^{1,2}_{\sigma}(\mathcal{O};\R^N) \cap C^{\infty}(\overbar{\mathcal{O}};\R^N)$ such that the $(a_k)$ are eigenfunctions of $A$, are an orthonormal basis in $L^2_{\sigma}(\mathcal{O};\R^N)$ and an orthogonal basis in $W^{1,2}_{\sigma}(\mathcal{O};\R^N)$. The eigenvalues $(\lambda_k)$ are strictly positive and approach infinity as $k \rightarrow \infty$.
\end{proposition}

\begin{proof}
    See [\cite{Robinson}] Theorem 2.24.
\end{proof}

\begin{remark}
Every $f \in L^2_{\sigma}(\mathcal{O};\R^N)$ admits the representation \begin{equation}\label{L2sigrep}
    f = \sum_{k=1}^\infty f_ka_k
\end{equation}
where $f_k = \inner{f}{a_k}$, as a limit in $L^2(\mathcal{O};\R^N)$.
\end{remark}

\begin{lemma}
    For $f \in L^2(\mathcal{O};\R^N)$, $$\mathcal{P}f = \sum_{k=1}^\infty \inner{f}{a_k}a_k.$$
\end{lemma}

\begin{proof}
    This is immediate from the fact that the $(a_k)$ form an orthogonal basis of $L^2_{\sigma}(\mathcal{O};\R^N)$, $\mathcal{P}f \in L^2_{\sigma}(\mathcal{O};\R^N)$ and $\mathcal{P}$ is an orthogonal projection: $$\mathcal{P}f = \sum_{k=1}^\infty \inner{\mathcal{P}f}{a_k}a_k = \sum_{k=1}^\infty \inner{f}{a_k}a_k.$$
\end{proof}

\begin{definition}
For $m \in \N$ we introduce the spaces $D(A^{m/2})$ as the subspaces of $L^2_{\sigma}(\mathcal{O};\R^N)$ consisting of functions $f$ such that $$\sum_{k=1}^\infty \lambda_k^m f_k^2 < \infty $$ for $f_k$ as in (\ref{L2sigrep}). Then $A^{m/2}:D(A^{m/2}) \rightarrow L^2_{\sigma}(\mathcal{O};\R^N)$ is defined by $$A^{m/2}:f \mapsto \sum_{k=1}^\infty \lambda_k^{m/2}f_ka_k.$$
\end{definition}

\begin{proposition} \label{A1}
    $D(A^{m/2}) \subset W^{m,2}(\mathcal{O};\R^N) \cap W^{1,2}_{\sigma}(\mathcal{O};\R^N)$ and the bilinear form $$\inner{f}{g}_m:= \inner{A^{m/2}f}{A^{m/2}g}$$ is an inner product on $D(A^{m/2})$. For $m$ even the induced norm is equivalent to the $W^{m,2}(\mathcal{O};\R^N)$ norm, and for $m$ odd we have the relation $$\norm{\cdot}_{W^{m,2}} \leq  c\norm{\cdot}_m$$ for some constant $c$.
\end{proposition}

\begin{proposition} \label{A2}
    $D(A) = W^{2,2}_{\sigma}(\mathcal{O};\R^N)$ and $D(A^{1/2}) = W^{1,2}_{\sigma}(\mathcal{O};\R^N)$ with the additional property that $\norm{\cdot}_1$ is equivalent to $\norm{\cdot}_{W^{1,2}}$ on this space.
\end{proposition}

\begin{proof}[Proof of \ref{A1}, \ref{A2}:]
    See [\cite{Constantin}] Proposition 4.12, [\cite{Robinson}] Exercises 2.12, 2.13 and the discussion in Subsection 2.3.
\end{proof}

\begin{proposition}  \label{prop for moving stokes}
For any $p,q \in \N$ with $p \leq q$, $p + q = 2m$ and $f \in D(A^{m/2})$, $g \in D(A^{q/2})$ we have that $$\inner{f}{g}_m = \inner{A^{p/2}f}{A^{q/2}g}.$$
\end{proposition}

\begin{proof}
The proof is direct:
   \begin{align*}
    \inner{f}{g}_m &= \inner{A^{m/2}f}{A^{m/2}g}\\
    & = \left \langle \sum_{k=1}^\infty \lambda_k^{m/2}f_ka_k, \sum_{j=1}^\infty \lambda_j^{m/2} g_j a_j  \right \rangle\\
    &= \sum_{k=1}^\infty\sum_{j=1}^\infty  \lambda_k^{m/2}f_k\lambda_j^{m/2} g_j\inner{a_k}{a_j}\\
    &= \sum_{k=1}^\infty \lambda_k^m f_kg_k\\
       &= \sum_{k=1}^\infty\sum_{j=1}^\infty  \lambda_k^{p/2}f_k\lambda_j^{q/2} g_j\inner{a_k}{a_j}\\
    &= \left \langle \sum_{k=1}^\infty \lambda_k^{p/2}f_ka_k, \sum_{j=1}^\infty \lambda_j^{q/2} g_j a_j  \right \rangle\\
    &= \inner{A^{p/2}f}{A^{q/2}g}.
\end{align*}
\end{proof}

\begin{lemma}
    The collection of functions $(a_k)$ form an orthogonal basis of $W^{1,2}_{\sigma}(\mathcal{O};\R^N)$ equipped with the $\inner{\cdot}{\cdot}_1$ inner product. 
\end{lemma}

\begin{proof}
    The completeness follows from that in $W^{1,2}_{\sigma}(\mathcal{O};\R^N)$ for the equivalent $\norm{\cdot}_{W^{1,2}}$. As for orthogonality, observe that $$\inner{a_j}{a_k}_1 = \inner{A^{1/2}a_j}{A^{1/2}a_k} = \lambda_j^{1/2}\lambda_k^{1/2}\inner{a_j}{a_k} = 0$$ from the orthogonality in $L^2(\mathcal{O};\R^N)$. 
\end{proof}

In addition to using these spaces defined by powers of the Stokes Operator, we also use the basis provided in Proposition \ref{basis of stokes} to consider finite dimensional approximations of these spaces.

\begin{definition}
We define $\mathcal{P}_n$ as the orthogonal projection onto $\textnormal{span}\{a_1, \dots, a_n\}$ in $L^2(\mathcal{O};\R^N)$. That is $\mathcal{P}_n$ is given by $$\mathcal{P}_n:f \mapsto \sum_{k=1}^n\inner{f}{a_k}a_k$$ for $f \in L^2(\mathcal{O};\R^N)$.
\end{definition}

\begin{lemma} \label{projectionsboundedlemma}
    The restriction of $\mathcal{P}_n$ to $D(A^{m/2})$ is self-adjoint for the $\inner{\cdot}{\cdot}_m$ inner product, and there exists a constant $c$ independent of $n$ such that for all $f\in D(A^{m/2})$, \begin{equation}\label{P_nbounds}
        \norm{\mathcal{P}_nf}_{W^{m,2}} \leq c\norm{f}_{W^{m,2}}.
    \end{equation}
\end{lemma}

\begin{proof}
    For $f,g \in D(A^{m/2})$ (and thus admitting the representation (\ref{L2sigrep})), we have that
    \begin{align*}
        \inner{\mathcal{P}_nf}{g}_m &= \left\langle \sum_{j=1}^n\inner{f}{a_j}a_j, \sum_{k=1}^\infty \inner{g}{a_k}a_k \right\rangle_m\\
        &= \left\langle \sum_{j=1}^n\inner{f}{a_j}\lambda_j^{m/2}a_j, \sum_{k=1}^\infty \inner{g}{a_k}\lambda_k^{m/2}a_k \right\rangle\\
        &= \left\langle \sum_{j=1}^n\inner{f}{a_j}\lambda_j^{m/2}a_j, \sum_{k=1}^n \inner{g}{a_k}\lambda_k^{m/2}a_k \right\rangle\\
        &= \left\langle \sum_{j=1}^\infty\inner{f}{a_j}\lambda_j^{m/2}a_j, \sum_{k=1}^n \inner{g}{a_k}\lambda_k^{m/2}a_k \right\rangle\\
        &= \inner{f}{\mathcal{P}_n g}_m
    \end{align*}
    as required for the first statement. For the second see [\cite{Robinson}] Lemma 4.1. 
\end{proof}

\begin{lemma} \label{one over n projection lemma}
    There exists a constant $c$ such that for all $f \in W^{1,2}_{\sigma}(\mathcal{O};\R^N)$, $g \in W^{2,2}_{\sigma}(\mathcal{O};\R^N)$ we have that 
    \begin{align*}
        \norm{(\mathcal{I}-\mathcal{P}_n)f}^2 &\leq \frac{1}{\lambda_n}\norm{f}_1^2\\
        \norm{(\mathcal{I}-\mathcal{P}_n)g}_1^2 &\leq \frac{1}{\lambda_n}\norm{g}_2^2
    \end{align*}
    where $\mathcal{I}$ represents the identity operator in the relevant spaces.
\end{lemma}

\begin{proof}
    For the first result, note that
    \begin{align*}
        \norm{(\mathcal{I}-\mathcal{P}_n)f}^2 &= \sum_{k=n+1}^\infty \inner{f}{a_k}^2\\
        &= \sum_{k=n+1}^\infty \frac{\lambda_k}{\lambda_k} \inner{f}{a_k}^2\\
        &\leq \frac{1}{\lambda_n}\sum_{k=n+1}^\infty \lambda_k\inner{f}{a_k}^2\\
        &\leq \frac{1}{\lambda_n}\sum_{k=1}^\infty \lambda_k\inner{f}{a_k}^2\\
        &= \frac{1}{\lambda_n}\norm{f}_1^2
    \end{align*}
    and the second follows identically.
\end{proof}

\begin{remark}
Evidently from the proof, the stronger property
\begin{align*}
        \norm{(\mathcal{I}-\mathcal{P}_n)f}^2 &\leq \frac{1}{\lambda_n}\norm{(\mathcal{I}-\mathcal{P}_n)f}_1^2\\
        \norm{(\mathcal{I}-\mathcal{P}_n)g}_1^2 &\leq \frac{1}{\lambda_n}\norm{(\mathcal{I}-\mathcal{P}_n)g}_2^2
    \end{align*}
    is true.
\end{remark}

To conclude this subsection we discuss briefly bounds related to the nonlinear term, which will be used in our analysis.

\begin{lemma} \label{biggglemma}
    For every $\phi \in W^{1,2}_{\sigma}(\mathcal{O};\R^N)$ and $f, g \in W^{1,2}(\mathcal{O};\R^N)$, we have that \begin{equation}\label{wloglhs}\inner{\mathcal{L}_{\phi}f}{g}= -\inner{f}{\mathcal{L}_{\phi}g}.\end{equation} 
\end{lemma}

\begin{proof}
    See Appendix I, \ref{Appendix I}.
\end{proof}

\begin{corollary} \label{cancellationproperty}
    For every $\phi \in W^{1,2}_{\sigma}(\mathcal{O};\R^N)$ and $f \in W^{1,2}(\mathcal{O};\R^N)$, we have that \begin{equation}\nonumber \inner{\mathcal{L}_{\phi}f}{f}= 0.\end{equation} 
\end{corollary}

\begin{proposition} \label{inequalityforcauchynonlinear}
    There exists a constant $c$ such that for any $f \in W^{1,2}_{\sigma}(\T;\R^N)$ and $g \in W^{2,2}_{\sigma}(\T;\R^N)$, \begin{equation} \label{nonlinear eq1} \norm{\mathcal{L}_fg} \leq c\norm{f}_{1}\norm{g}_{1}^{1/2}\norm{g}_{2}^{1/2}.\end{equation} Meanwhile for $f \in W^{1,2}_{\sigma}(\mathscr{O};\R^N)$ and $g \in W^{2,2}_{\sigma}(\mathscr{O};\R^N)$ we have that \begin{equation} \label{nonlinear eq2} \norm{\mathcal{L}_fg} \leq c\norm{f}_{1}\left(\norm{g}_{1}^{1/2}\norm{g}_{2}^{1/2} + \norm{g}_1\right).\end{equation}
\end{proposition}

\begin{proof}
See Appendix I, \ref{Appendix I}.
\end{proof}

\subsection{The SALT Operator} \label{subsection the salt operator}

Having established the relevant function spaces and some fundamental properties of the operators involved in the deterministic Navier-Stokes Equation, we now address the operator $B$ appearing in the Stratonovich integral of (\ref{number2equation}). As in [\cite{goodair2022stochastic}] Subsection 2.2, the operator $B$ is defined by its action on the basis vectors $(e_i)$ of $\mathfrak{U}$. We shall show in Subsection \ref{subsection ito} that $B$ does indeed satisfy Assumption 2.2.2 of [\cite{goodair2022stochastic}] for the spaces to $V,H,U,X$ to be defined. With the notation of [\cite{goodair2022stochastic}], each $B_i$ is defined relative to the correlations $\xi_i$ for sufficiently regular $f$ by the mapping $$B_i:f \mapsto \mathcal{L}_{\xi_i}f + \mathcal{T}_{\xi_i}f$$ where $\mathcal{L}$ is as before, and $\mathcal{T}$ is a new operator that we introduce defined by $$\mathcal{T}_{g}f := \sum_{j=1}^N f^j\nabla g^j.$$ We shall assume throughout that each $\xi_i$ belongs to the space $W^{1,2}_{\sigma}(\mathcal{O};\R^N)$, and for the meantime that there is some fixed $m \in \N$ with $\xi_i \in W^{m+2,\infty}(\mathcal{O};\R^N)$.

\begin{lemma} \label{biggestcorcor}
     There exists a constant $c$ such that for each $i$ and for all $f \in W^{k,2}(\mathcal{O};\R^N)$ with $k=0, \dots, m+1$, we have the bound \begin{equation}\label{T_ibound}\norm{\mathcal{T}_{\xi_i}f}_{W^{k,2}}^2 \leq  c \norm{\xi_i}^2_{W^{k+1,\infty}}\norm{f}^2_{W^{k,2}}.\end{equation} Therefore $\mathcal{T}_{\xi_i}$ is a bounded linear operator on $L^2(\mathcal{O};\R^N)$ so has an everywhere defined adjoint $\mathcal{T}_{\xi_i}^*$ in this space. 
\end{lemma}

\begin{proof}
See Appendix II, \ref{appendix ii}.
\end{proof}

\begin{lemma} \label{boundsonL_xi}
    There exists a constant $c$ such that for each $i$ and for all $f \in W^{k+1,2}(\mathcal{O};\R^N)$ with $k=0, \dots, m+1$, we have the bound \begin{equation} \label{L_ibound} \norm{\mathcal{L}_{\xi_i}f}_{W^{k,2}}^2 \leq c\norm{\xi_i}^2_{W^{k,\infty}}\norm{f}^2_{W^{k+1,2}}.\end{equation}
    Therefore $\mathcal{L}_{\xi_i}$ is a densely defined operator in $L^2(\mathcal{O};\R^N)$ with domain of definition $W^{1,2}(\mathcal{O};\R^N)$, and has adjoint $\mathcal{L}_{\xi_i}^*$ in this space given by $-\mathcal{L}_{\xi_i}$ with same dense domain of definition.
\end{lemma}

\begin{proof}
    See Appendix II, \ref{appendix ii} for the proof of (\ref{L_ibound}). The fact that $-\mathcal{L}_{\xi_i}$ is the adjoint for $\mathcal{L}_{\xi_i}$ on $W^{1,2}(\mathcal{O};\R^N)$ follows immediately from Lemma \ref{biggglemma}.
\end{proof}

\begin{corollary} \label{corollary for B_i adjoint}
    There exists a constant $c$ such that for each $i$ and for all $f \in W^{k+1,2}(\mathcal{O};\R^N)$ with $k=0, \dots, m+1$, we have the bound \begin{equation} \label{boundsonB_i} \norm{B_if}_{W^{k,2}}^2 \leq c\norm{\xi_i}^2_{W^{k+1,\infty}}\norm{f}^2_{W^{k+1,2}}.\end{equation}
    Moreover $B_{i}$ is a densely defined operator in $L^2(\mathcal{O};\R^N)$ with domain of definition $W^{1,2}(\mathcal{O};\R^N)$, and has adjoint $B_i^*$ in this space given by $-\mathcal{L}_{\xi_i} + \mathcal{T}_{\xi_i}^*$ with same dense domain of definition.
\end{corollary}
Our techniques all centre around energy estimates, where the key idea as to how we preserve these estimates in the case of a transport type noise owes to the following proposition.

\begin{proposition} \label{prop for conservation B_i}
    There exists a constant $c$ such that for each $i$ and for all $f \in W^{k+2,2}(\mathcal{O};\R^N)$ with $k=0, \dots, m$, we have the bounds
    \begin{align}
    \inner{B_i^2f}{f}_{W^{k,2}} +  \norm{B_if}_{W^{k,2}}^2 &\leq c\norm{\xi_i}_{W^{k+2,\infty}}^2\norm{f}_{W^{k,2}}^2 \label{combinedterminenergyinequality},\\
    \inner{B_if}{f}_{W^{k,2}}^2 &\leq c\norm{\xi_i}^2_{W^{k+1,\infty}}\norm{f}^4_{W^{k,2}}. \label{finalboundinderivativeproof}
\end{align}
    
\end{proposition}

\begin{proof}
        See Appendix II, \ref{appendix ii}. 
\end{proof}

\begin{remark}
Proposition \ref{prop for conservation B_i} has been understood to hold in the case $\mathcal{O} = \T$, for example in the works [\cite{Crisan1}, \cite{Crisan2}], though perhaps not for the bounded domain.
\end{remark}

We will also use the corresponding result to Lemma \ref{APequalsA}
for the case of the operator $B_i$. This holds true only in the presence of the additional $\mathcal{T}_{\xi_i}$ term in the operator, highlighting the significance of considering a noise which is not purely transport.

\begin{lemma} \label{PBiequalsPBiP}
    We have that $$B_i:L^{2, \perp}_{\sigma}(\mathcal{O};\R^N) \cap W^{1,2}(\mathcal{O};\R^N) \rightarrow L^{2, \perp}_{\sigma}(\mathcal{O};\R^N)$$ and moreover that $\mathcal{P}B_i = \mathcal{P}B_i\mathcal{P}$ on $W^{1,2}(\mathcal{O};\R^N)$. 
\end{lemma}

\begin{proof}
      For $\nabla g \in L^{2, \perp}_{\sigma}(\mathcal{O};\R^N) \cap W^{1,2}(\mathcal{O};\R^N)$,
    \begin{align*}
    B_i(\nabla g) &= \mathcal{L}_{\xi_i}(\nabla g) + \mathcal{T}_{\xi_i}(\nabla g)\\
    &= \sum_{j=1}^N\xi_i^j \partial_j(\nabla g) + \sum_{j=1}^N \partial_jg \nabla \xi_i^j\\
    &= \sum_{j=1}^N\xi_i^j (\nabla \partial_jg) + \sum_{j=1}^N (\nabla \xi_i^j) \partial_jg\\
    &= \nabla \sum_{j=1}^N\xi_i^j\partial_jg\\
    &\in L^{2, \perp}_{\sigma}(\mathcal{O};\R^N)
\end{align*}
    using in the last line the assumption that $\nabla g \in W^{1,2}(\mathcal{O};\R^N)$ and that $\xi_i \in W^{1,\infty}(\mathcal{O};\R^N)$. Just as seen in Lemma \ref{APequalsA} we now make a distinction between the settings of $\T$ and $\mathscr{O}$. For the bounded domain $\mathscr{O}$, we take any $f \in W^{1,2}(\mathscr{O};\R^N)$ and use the representation (\ref{decomp}) to see that $$\mathcal{P}B_if = \mathcal{P}B_i\left(\mathcal{P}f + \nabla g\right) = \mathcal{P}B_i\mathcal{P}f + \mathcal{P}(B_i\nabla g) =  \mathcal{P}B_i\mathcal{P}f$$ as required, using again the $W^{1,2}(\mathscr{O};\R^N)$ regularity of both components of the decomposition (\ref{decomp}). In the case of the Torus we must address the constant term in the decomposition (\ref{decomp torus}), appreciating that $$B_ic = \mathcal{T}_{\xi_i}c = \sum_{j=1}^Nc^j\nabla \xi_i^j = \nabla \sum_{j=1}^Nc^j\xi_i^j \in L^{2, \perp}_{\sigma}(\T;\R^N)$$ so the result follows in the same manner.
\end{proof}

This relationship between $B_i$ and the Leray Projector sets up an analysis of the $B_i$ terms in the $\inner{\cdot}{\cdot}_m$ inner product spaces, to which we are further interested in the commutation relationship between $B_i$ and the Laplacian $\Delta$. This is addressed here.

\begin{proposition} \label{boundoncommutator}
    There exists a constant $c$ such that for every $f\in W^{3,2}(\mathcal{O};\R^N)$, $$\norm{[\Delta,B_i]f}^2 \leq c\norm{\xi_i}_{W^{3,\infty}}^2\norm{f}_{W^{2,2}}^2$$ where $[\Delta,B_i]$ is the commutator $$[\Delta,B_i]:= \Delta B_i - B_i\Delta.$$
\end{proposition}

\begin{proof}
    See Appendix II, \ref{appendix ii}.
\end{proof}

\newpage
\section{Analysis of the Velocity Equation} \label{section velocity}

In this section we restrict ourselves to the Torus $\T$, leaving a treatment of the bounded domain to Section \ref{section vorticity form}. We also now fix our assumptions on the $(\xi_i)$, assuming that each $\xi_i \in W^{1,2}_{\sigma}(\T;\R^N) \cap W^{3,\infty}(\T;\R^N)$ and they collectively satisfy \begin{equation} \label{bounds on xi_i}
    \sum_{i=1}^\infty\norm{\xi_i}_{W^{3,\infty}}^2 < \infty.
\end{equation}

\subsection{The It\^{o} Form} \label{subsection ito}
To facilitate our analysis we work not with the equation (\ref{number2equation}), but instead one written in terms of an It\^{o} integral and projected by the Leray Projector as discussed at the start of Subsection \ref{functional framework subsection}. As a first step then we consider the new equation \begin{equation} \label{projected strato}
    u_t - u_0 + \int_0^t\mathcal{P}\mathcal{L}_{u_s}u_s\,ds + \nu\int_0^t A u_s\, ds + \int_0^t \mathcal{P}Bu_s \circ d\mathcal{W}_s = 0
\end{equation}
obtained at a heuristic level by projecting all terms of (\ref{number2equation}). Having not defined solutions of (\ref{number2equation}) we cannot be too formal here, but the idea is that we required solutions in $L^2_{\sigma}(\mathcal{O};\R^N)$ with initial condition also in this space so they are invariant under $\mathcal{P}$, and $\mathcal{P}$ is a bounded linear operator so can be taken through the integrals (see [\cite{goodair2022stochastic}] Corollary 1.6.9.1 for this result in stochastic integration). We then look to convert (\ref{projected strato}) into It\^{o} form via the abstract procedure used in [\cite{goodair2022stochastic}] Subsections 2.2 and 2.3, the result of which is stated in Appendix III, \ref{appendix iii}.\\

Towards this goal we define the quartet of spaces \begin{align*}
    V&:= W^{3,2}_{\sigma}(\T;\R^N), \qquad H:= W^{2,2}_{\sigma}(\T;\R^N),\\
    U&:= W^{1,2}_{\sigma}(\T;\R^N), \qquad X:= L^2_{\sigma}(\T;\R^N).
\end{align*}
We equip $L^2_{\sigma}(\T;\R^N)$ with the usual $\inner{\cdot}{\cdot}$ inner product, but then equip $W^{1,2}_{\sigma}(\T;\R^N)$ and $W^{2,2}_{\sigma}(\T;\R^N)$ with the $\inner{\cdot}{\cdot}_1$ and $\inner{\cdot}{\cdot}_2$ inner products respectively, recalling propositions \ref{A1} and \ref{A2}. In fact we also have that $D(A^{3/2}) = W^{3,2}_{\sigma}(\T;\R^N)$ and that the $\inner{\cdot}{\cdot}_3$ inner product is equivalent to the usual $\inner{\cdot}{\cdot}_{W^{3,2}}$ one on this space (see [\cite{Robinson}] Theorem 2.27), so we endow $W^{3,2}_{\sigma}(\T;\R^N)$ with $\inner{\cdot}{\cdot}_3$.\\

Our SPDE (\ref{projected strato}) takes the form of (\ref{stratoSPDE}) for the operators $$\mathcal{Q}:= -\left(\mathcal{P}\mathcal{L} + \nu A \right) $$ and $$\mathcal{G}:=-\mathcal{P}B.$$ We now check the Assumptions \ref{Qassumpt}, \ref{Gassumpt}. Starting with \ref{Qassumpt}, we first of all have that $\nu A$ is continuous from $W^{3,2}_{\sigma}(\T;\R^N)$ into $W^{1,2}_{\sigma}(\T;\R^N)$ as the Laplacian is continuous from $W^{3,2}(\T;\R^N)$ into $W^{1,2}(\T;\R^N)$, Proposition \ref{continuityofP} and the equivalence of norms. Therefore it is measurable and as a linear operator too satisfies the boundedness. As for $\mathcal{P}\mathcal{L}$, measurability is satisfied in the same way (recall Lemma \ref{continuityofnonlinear}) and for the boundedness we have that
    $$\norm{\mathcal{P}\mathcal{L}_{f}f}_{1} \leq c\norm{\mathcal{P}\mathcal{L}_ff}_{W^{1,2}} \leq c\norm{\mathcal{L}_ff}_{W^{1,2}} \leq c\norm{f}_{W^{1,2}}\norm{f}_{W^{3,2}} \leq c\norm{f}_{1}\norm{f}_{3}$$
for any $f\in W^{3,2}_{\sigma}(\T;\R^N)$ where $c$ is a generic constant, critically applying (\ref{second begin align}). This verifies Assumption \ref{Qassumpt} so we move on to \ref{Gassumpt}, which is immediate from (\ref{boundsonB_i}) and the linearity of $\mathcal{P}B$ to show continuity in all relevant spaces. We note now though that the Leray Projector does not preserve the space $W^{1,2}_0(\mathscr{O};\R^N)$, and so we cannot say that $\mathcal{P}B_i:W^{2,2}_{\sigma}(\mathscr{O};\R^N) \rightarrow W^{1,2}_{\sigma}(\mathscr{O};\R^N)$. The issues arising from this operator not satisfying the zero-trace property are fundamentally why we only treat the Torus in this context.\\

With Theorem \ref{theorem for ito strat conversion} in mind, we move instead to an analysis of the It\^{o} Form
\begin{equation} \label{projected Ito}
    u_t = u_0 - \int_0^t\mathcal{P}\mathcal{L}_{u_s}u_s\ ds - \nu\int_0^t A u_s\, ds + \frac{1}{2}\int_0^t\sum_{i=1}^\infty \mathcal{P}B_i^2u_s ds - \int_0^t \mathcal{P}Bu_s d\mathcal{W}_s 
\end{equation}
where we rewrite $(-\mathcal{P}B_i)^2$ as $\mathcal{P}B_i^2$ firstly from the linearity of $\mathcal{P}B_i$ to deal with the minus and secondly using Lemma \ref{PBiequalsPBiP}. It is worth appreciating here that we chose to project the equation and then convert it into It\^{o} Form, but we may equally have chosen to convert the unprojected Stratonovich Form and then project the resulting It\^{o} Equation. Without addressing the conversion of the unprojected equation in complete detail, we would directly arrive at (\ref{projected Ito}) taking this approach as our correction term would simply be $\mathcal{P}\sum_{i=1}^\infty B_i^2$ which is just $\sum_{i=1}^\infty \mathcal{P}B_i^2$ from the linearity and continuity. Therefore Lemma \ref{PBiequalsPBiP} is necessary in ensuring a consistency between these approaches.

\subsection{Existence, Uniqueness and Maximality} \label{subsection velocity existence}

We now state and prove the existence, uniqueness and maximality results for our Navier-Stokes Equation, the idea being to use the abstract results of [\cite{Goodair abs}] stated in Appendices IV and V, \ref{appendix iv} and \ref{appendix v}. We first give the result pertaining directly to the Stratonovich form.

\begin{theorem} \label{existence for NS Strat}
    For any given $\mathcal{F}_0-$measurable $u_0: \Omega \rightarrow W^{2,2}_{\sigma}(\T;\R^N)$ there exists a pair $(u,\tau)$ such that: $\tau$ is a $\mathbbm{P}-a.s.$ positive stopping time and $u$ is a process whereby for $\mathbbm{P}-a.e.$ $\omega$, $u_{\cdot}(\omega) \in C\left([0,T];W^{2,2}_{\sigma}(\T;\R^N)\right)$ and $u_{\cdot}(\omega)\mathbbm{1}_{\cdot \leq \tau(\omega)} \in L^2\left([0,T];W^{3,2}_{\sigma}(\T;\R^N)\right)$ for all $T>0$ with $u_{\cdot}\mathbbm{1}_{\cdot \leq \tau}$ progressively measurable in $W^{3,2}_{\sigma}(\T;\R^N)$, and moreover satisfying the identity
\begin{equation} \nonumber
  u_t = u_0 - \int_0^{t\wedge \tau}\mathcal{P}\mathcal{L}_{u_s}u_s\,ds -\nu\int_0^{t\wedge \tau} A u_s\, ds - \int_0^{t \wedge \tau} \mathcal{P}Bu_s \circ d\mathcal{W}_s 
\end{equation}
$\mathbbm{P}-a.s.$ in $L^2_{\sigma}(\T;\R^N)$ for all $t \geq 0$.
\end{theorem}

The idea behind Theorem \ref{existence for NS Strat} is of course to analyse the It\^{o} Form and apply Theorem \ref{theorem for ito strat conversion} as justified in Subsection \ref{subsection ito}. Once we convert to the It\^{o} Form though starting with an initial condition in $W^{2,2}_{\sigma}(\T;\R^N)$ is not optimal in the sense that, at least according to the deterministic theory, we should be able to construct a solution (satisfying the identity in $L^2_{\sigma}(\T;\R^N)$ as is natural)
for only a $W^{1,2}_{\sigma}(\T;\R^N)$ initial condition. To this end we give the following definitions and the main result of this subsection. Definition \ref{definitionofirregularsolutionNS} is stated for an arbitrary $\mathcal{F}_0-$measurable $u_0:\Omega \rightarrow W^{1,2}_{\sigma}(\T;\R^N)$.

\begin{definition} \label{definitionofirregularsolutionNS}
A pair $(u,\tau)$ where $\tau$ is a $\mathbb{P}-a.s.$ positive stopping time and $u$ is a process such that for $\mathbb{P}-a.e.$ $\omega$, $u_{\cdot}(\omega) \in C\left([0,T];W^{1,2}_{\sigma}(\T;\R^N)\right)$ and $u_{\cdot}(\omega)\mathbf{1}_{\cdot \leq \tau(\omega)} \in L^2\left([0,T];W^{2,2}_{\sigma}(\T;\R^N)\right)$ for all $T>0$ with $u_{\cdot}\mathbf{1}_{\cdot \leq \tau}$ progressively measurable in $W^{2,2}_{\sigma}(\T;\R^N)$, is said to be a local strong solution of the equation (\ref{projected Ito}) if the identity
\begin{equation} \label{identityindefinitionoflocalsolutionHNS}
     u_t = u_0 - \int_0^{t\wedge \tau}\mathcal{P}\mathcal{L}_{u_s}u_s\ ds - \nu\int_0^{t\wedge \tau} A u_s\, ds + \frac{1}{2}\int_0^{t\wedge \tau}\sum_{i=1}^\infty \mathcal{P}B_i^2u_s ds - \int_0^{t\wedge \tau} \mathcal{P}Bu_s d\mathcal{W}_s
\end{equation}
holds $\mathbb{P}-a.s.$ in $L^2_{\sigma}(\T;\R^N)$ for all $t \geq 0$.
\end{definition}

\begin{remark}
If $(u,\tau)$ is a $V$-valued local strong solution of the equation (\ref{thespde}), then $u_\cdot = u_{
\cdot \wedge \tau}$. A justification that the integrals are well defined is given in the abstract case in [\cite{Goodair abs}]. 
\end{remark}

\begin{definition} \label{H valued maximal definition NS}
A pair $(u,\Theta)$ such that there exists a sequence of stopping times $(\theta_j)$ which are $\mathbb{P}-a.s.$ monotone increasing and convergent to $\Theta$, whereby $(u_{\cdot \wedge \theta_j},\theta_j)$ is a local strong solution of the equation (\ref{projected Ito}) for each $j$, is said to be a maximal strong solution of the equation (\ref{thespde}) if for any other pair $(v,\Gamma)$ with this property then $\Theta \leq \Gamma$ $\mathbb{P}-a.s.$ implies $\Theta = \Gamma$ $\mathbb{P}-a.s.$.
\end{definition}

\begin{definition} \label{definition unique NS}
A maximal strong solution $(u,\Theta)$ of the equation (\ref{projected Ito}) is said to be unique if for any other such solution $(v,\Gamma)$, then $\Theta = \Gamma$ $\mathbb{P}-a.s.$ and for all $t \in [0,\Theta)$, \begin{equation} \nonumber\mathbb{P}\left(\left\{\omega \in \Omega: u_{t}(\omega) =  v_{t}(\omega)  \right\} \right) = 1. \end{equation}
\end{definition}

\begin{theorem} \label{theorem2NS}
For any given $\mathcal{F}_0-$ measurable $u_0:\Omega \rightarrow W^{1,2}_{\sigma}(\T;\R^N)$, there exists a unique maximal strong solution $(u,\Theta)$ of the equation (\ref{projected Ito}). Moreover at $\mathbb{P}-a.e.$ $\omega$ for which $\Theta(\omega)<\infty$, we have that \begin{equation}\label{blowupproperty2NS}\sup_{r \in [0,\Theta(\omega))}\norm{u_r(\omega)}_{1}^2 + \int_0^{\Theta(\omega)}\norm{u_r(\omega)}_2^2dr = \infty.\end{equation}
\end{theorem}

Definitions \ref{definitionofirregularsolutionNS}, \ref{H valued maximal definition NS}, \ref{definition unique NS} and Theorem \ref{theorem2NS} are precisely Definitions \ref{definitionofirregularsolution}, \ref{H valued maximal definition}, \ref{definition unique} and Theorem \ref{theorem2} for the equation (\ref{projected Ito}) with respect to the spaces $V,H,U,X$ as defined in Subsection \ref{subsection ito}. Indeed we would also prove Theorem \ref{existence for NS Strat} through Theorem \ref{theorem for ito strat conversion} by showing the existence of a local solution with the regularity of Definition \ref{definitionofregularsolution}. Therefore we prove both Theorem \ref{existence for NS Strat} and \ref{theorem2NS} by showing that the assumptions of Appendices IV and V, \ref{appendix iv} and \ref{appendix v}, are satisfied.\\ 

We work with the operators
\begin{align*}
    \mathcal{A}&:= -\left(\mathcal{P}\mathcal{L} + \nu A \right) + \frac{1}{2}\sum_{i=1}^\infty \mathcal{P}B_i^2\\
    \mathcal{G}&:= -\mathcal{P}B
\end{align*}
which were addressed to be measurable mappings into the required spaces in Subsection \ref{subsection ito}. We now proceed to justify the assumptions of Appendix IV. First note that the density of the spaces is immediately inherited from the density of the usual Sobolev Spaces and the equivalence of the norms. The bilinear form satisfying (\ref{bilinear formog}) is chosen to be $$\inner{f}{g}_{U \times V}:= \inner{A^{1/2}f}{A^{3/2}g}$$ which reduces to the $\inner{\cdot}{\cdot}_2$ inner product from Proposition \ref{prop for moving stokes}. In the following $c$ will represent a generic constant changing from line to line, $c(\varepsilon)$ will be a generic constant dependent on a fixed $\varepsilon$, $f$ and $g$ will be arbitrary elements of $W^{3,2}_{\sigma}(\T;\R^N)$ and $f_n \in \textnormal{span}\{a_1, \cdots, a_n\}$. 

\begin{proof}[Assumption \ref{assumption fin dim spaces}:]
    We use the system $(a_k)$ of eigenfunctions of the Stokes Oeprator given in Proposition \ref{basis of stokes}, satisfying (\ref{projectionsboundedonH}) and (\ref{mu2}) from Lemmas \ref{projectionsboundedlemma} and \ref{one over n projection lemma} respectively.
\end{proof}

\begin{proof}[Assumption \ref{new assumption 1}:]
    Once more (\ref{111}) follows from the discussion in Subsection \ref{subsection ito}. For (\ref{222}) we treat the different operators in $\mathcal{A}$ individually, starting from the nonlinear term:
    \begin{align*}
        \norm{\mathcal{P}\mathcal{L}_ff - \mathcal{P}\mathcal{L}_gg}_1 &= \norm{\mathcal{P}\mathcal{L}_f(f-g) + \mathcal{P}\mathcal{L}_{f-g}g}_1\\
        &\leq \norm{\mathcal{P}\mathcal{L}_f(f-g)}_1 + \norm{\mathcal{P}\mathcal{L}_{f-g}g}_1\\
        &\leq c\norm{\mathcal{L}_f(f-g)}_{W^{1,2}} + c\norm{\mathcal{L}_{f-g}g}_{W^{1,2}}\\
        &\leq c\norm{f}_{W^{1,2}}\norm{f-g}_{W^{3,2}} + c\norm{f-g}_{W^{1,2}}\norm{g}_{W^{3,2}}\\
        &\leq c\left(\norm{f}_{W^{1,2}} + \norm{g}_{W^{3,2}}\right)\norm{f-g}_{W^{3,2}}\\
        &\leq c\left(\norm{f}_{1} + \norm{g}_{3}\right)\norm{f-g}_{3}
    \end{align*}
having applied (\ref{second begin align}). From the linearity of $\nu A$ and $\frac{1}{2}\sum_{i=1}^\infty \mathcal{P}B_i^2$ then the corresponding result follows immediately from (\ref{111}) and this subsequently justifies (\ref{222}). Additionally (\ref{333}) follows immediately from the already justified Assumption \ref{Gassumpt}. 

\end{proof}

For the justification of Asusmption \ref{assumptions for uniform bounds2} we call upon some intermediary results. 

\begin{lemma} \label{uniformdelanonlinear}
    For any $\varepsilon > 0$, we have that 
    $$\left\vert\inner{\mathcal{P}_n\mathcal{P}\mathcal{L}_{f_n}f_n}{f_n}_2\right\vert \leq c(\varepsilon)\norm{f_n}^4_2 + \varepsilon \norm{f_n}_{3}^2. $$
\end{lemma}

\begin{proof}
   We use the established result that for $k=2$ then the Sobolev Space $W^{k,2}(\mathcal{O};\R)$ is an algebra (a result first presented in [\cite{Strichartz}]), to deduce that
   \begin{align*}
       \norm{\mathcal{L}_{f_n}f_n}_{W^{2,2}} &= \left\Vert\sum_{j=1}^Nf_n^j\partial_jf_n \right\Vert_{W^{2,2}}\\
       &\leq \sum_{j=1}^N\left\Vert f_n^j\partial_jf_n \right\Vert_{W^{2,2}}\\
       &\leq \sum_{j=1}^N\sum_{l=1}^N\left\Vert f_n^j\partial_jf_n^l \right\Vert_{W^{2,2}(\T;\R)}\\
      &\leq \sum_{j=1}^N\sum_{l=1}^N\left\Vert f_n^j\right\Vert_{W^{2,2}(\T;\R)}\left\Vert\partial_jf_n^l \right\Vert_{W^{2,2}(\T;\R)}\\
      &\leq \left(\sum_{j=1}^N\sum_{l=1}^N\left\Vert f_n^j\right\Vert_{W^{2,2}(\T;\R)}\right)\left(\sum_{j=1}^N\sum_{l=1}^N\left\Vert\partial_jf_n^l \right\Vert_{W^{2,2}(\T;\R)}\right)\\
      &\leq c\norm{f_n}_{W^{2,2}}\norm{f_n}_{W^{3,2}}\\
      &\leq c\norm{f_n}_{2}\norm{f_n}_{3}.
   \end{align*}
   From here we simply use that $\mathcal{P}_n$ is self-adjoint (Lemma \ref{projectionsboundedlemma}) and Young's Inequality to see that 
    \begin{align*}
        \left\vert\inner{\mathcal{P}_n\mathcal{P}\mathcal{L}_{f_n}f_n}{f_n}_2\right\vert &= \left\vert\inner{\mathcal{P}\mathcal{L}_{f_n}f_n}{f_n}_2\right\vert\\
        &\leq \norm{\mathcal{P}\mathcal{L}_{f_n}f_n}_2\norm{f_n}_2\\
        &\leq c\norm{\mathcal{L}_{f_n}f_n}_{W^{2,2}}\norm{f_n}_{2}\\
        &\leq c\norm{f_n}_{2}\norm{f_n}_{3}\norm{f_n}_{2}\\
        &\leq c(\varepsilon)\norm{f_n}^4_2 + \varepsilon \norm{f_n}_{W^{3,2}}^2.
    \end{align*}
    
\end{proof}

\begin{remark}
The algebra property when $k=2$ is fundamental to this result; we are prevented from applying Theorem \ref{theorem1} in the case $H:=W^{1,2}_{\sigma}(\T;\R^N)$ as we would have no algebra property to apply the same method.
\end{remark}

\begin{lemma} \label{preliminary bound for cauchy}
    For any $\varepsilon > 0$ we have the bound $$\inner{\mathcal{P}_n\mathcal{P}B_i^2f_n}{f_n}_1 + \inner{\mathcal{P}_n\mathcal{P}B_if_n}{\mathcal{P}_n\mathcal{P}B_if_n}_1 \leq  c(\varepsilon)\norm{\xi_i}_{W^{3,\infty}}^2\norm{f_n}_1^2 + \varepsilon \norm{\xi_i}_{W^{3,\infty}}^2\norm{f_n}_2^2.$$
\end{lemma}

\begin{proof}
     From Lemma \ref{projectionsboundedlemma} we can readily justify the inequality $$\inner{\mathcal{P}_n\mathcal{P}B_i^2f_n}{f_n}_1 + \inner{\mathcal{P}_n\mathcal{P}B_if_n}{\mathcal{P}_n\mathcal{P}B_if_n}_1 \leq  \inner{\mathcal{P}B_i^2f_n}{f_n}_1 + \inner{\mathcal{P}B_if_n}{\mathcal{P}B_if_n}_1$$ and moreover from Proposition \ref{prop for moving stokes} that this is just $$\inner{\mathcal{P}B_i^2f_n}{Af_n} + \inner{\mathcal{P}B_if_n}{A\mathcal{P}B_if_n}.$$
     Using Lemmas \ref{PBiequalsPBiP} and \ref{APequalsA}, we rewrite this as $$\inner{(\mathcal{P}B_i)^2f_n}{Af_n} + \inner{\mathcal{P}B_if_n}{AB_if_n} $$ and further as \begin{equation}\label{furtheras}\inner{\mathcal{P}B_if_n}{B_i^*Af_n} - \inner{\mathcal{P}B_if_n}{\Delta B_if_n}\end{equation} for the adjoint $B_i^*$ specified in Corollary \ref{corollary for B_i adjoint}. We look to commute the Laplacian and $B_i$, using Proposition \ref{boundoncommutator} and subsequently the cancellation of the derivative in $B_i$ when considered with the adjoint. Indeed,
     \begin{align*}
         - \inner{\mathcal{P}B_if_n}{\Delta B_if_n} &= - \inner{\mathcal{P}B_if_n}{([\Delta, B_i] + B_i\Delta )f_n}\\
         &= - \inner{\mathcal{P}B_if_n}{[\Delta, B_i]f_n} - \inner{\mathcal{P}B_if_n}{ B_i\Delta f_n}\\
         &= - \inner{\mathcal{P}B_if_n}{[\Delta, B_i]f_n} - \inner{\mathcal{P}B_if_n}{ \mathcal{P}B_i\Delta f_n}\\
         &= - \inner{\mathcal{P}B_if_n}{[\Delta, B_i]f_n} + \inner{\mathcal{P}B_if_n}{ \mathcal{P}B_iAf_n}\\
         &= - \inner{\mathcal{P}B_if_n}{[\Delta, B_i]f_n} + \inner{\mathcal{P}B_if_n}{B_iAf_n}
     \end{align*}
     using Lemma \ref{PBiequalsPBiP} again. Thus (\ref{furtheras}) becomes 
    $$\inner{\mathcal{P}B_if_n}{B_i^*Af_n} - \inner{\mathcal{P}B_if_n}{[\Delta, B_i]f_n} + \inner{\mathcal{P}B_if_n}{B_iAf_n}$$ or simply \begin{equation}\label{orsimply}\inner{\mathcal{P}B_if_n}{(\mathcal{T}_{\xi_i} + \mathcal{T}_{\xi_i}^*)Af_n - [\Delta, B_i]f_n}\end{equation}
     which we look to bound through Cauchy-Schwartz and the results of (\ref{boundsonB_i}), Lemma \ref{biggestcorcor} and Proposition \ref{boundoncommutator} to see that \begin{align*}
         (\ref{orsimply}) &\leq\norm{\mathcal{P}B_if_n}\left( \norm{(\mathcal{T}_{\xi_i} + \mathcal{T}_{\xi_i}^*)Af_n} +  \norm{[\Delta, B_i]f_n}\right)\\ &\leq c\norm{\xi_i}_{W^{1,\infty}}\norm{f_n}_{W^{1,2}}\left(\norm{\xi_i}_{W^{1,\infty}}\norm{Af_n} + \norm{\xi_i}_{W^{3,\infty}}\norm{f_n}_{W^{2,2}}\right)\\
         &\leq c\norm{\xi_i}_{W^{1,\infty}}\norm{f_n}_{1}\left(\norm{\xi_i}_{W^{1,\infty}}\norm{f_n}_{2} + \norm{\xi_i}_{W^{3,\infty}}\norm{f_n}_{2}\right)\\
         &\leq c\norm{\xi_i}_{W^{3,\infty}}^2\norm{f_n}_{1}\norm{f_n}_{2}\\
         &\leq c(\varepsilon)\norm{\xi_i}_{W^{3,\infty}}^2\norm{f_n}_{1}^2 + \varepsilon \norm{\xi_i}_{W^{3,\infty}}^2\norm{f_n}_{2}^2
     \end{align*}
     as required.

\end{proof}

\begin{lemma} \label{uniformdelanoise}
    For any $\varepsilon > 0$, we have that $$\inner{\mathcal{P}_n\mathcal{P}B_i^2f_n}{f_n}_2 + \inner{\mathcal{P}_n\mathcal{P}B_if_n}{\mathcal{P}_n\mathcal{P}B_if_n}_2 \leq c(\varepsilon)\norm{\xi_i}^2_{W^{3,\infty}}\norm{f_n}_2^2 + \varepsilon \norm{\xi_i}_{W^{3,\infty}}^2\norm{f_n}_{3}^2. $$
\end{lemma}

\begin{proof}
    As with Lemma \ref{preliminary bound for cauchy} we can immediately say that \begin{align}\nonumber \inner{\mathcal{P}_n\mathcal{P}B_i^2f_n}{f_n}_2 &+ \inner{\mathcal{P}_n\mathcal{P}B_if_n}{\mathcal{P}_n\mathcal{P}B_if_n}_2 \\ & \qquad \qquad \qquad \qquad \qquad  \leq  \inner{\mathcal{P}B_i^2f_n}{f_n}_2 + \inner{\mathcal{P}B_if_n}{\mathcal{P}B_if_n}_2 \label{boundybound}
    \end{align}
    which we again manipulate with Lemmas \ref{APequalsA} and \ref{PBiequalsPBiP} to give
    \begin{align*}
        (\ref{boundybound}) &= \inner{A\mathcal{P}B_i^2f_n}{Af_n} + \inner{A\mathcal{P}B_if_n}{A\mathcal{P}B_if_n}\\
        & = \inner{AB_i^2f_n}{Af_n} + \inner{AB_if_n}{AB_if_n}\\
        & = -\inner{\mathcal{P} \Delta B_i^2f_n}{Af_n} - \inner{AB_if_n}{\mathcal{P}\Delta B_if_n}\\
        &= - \inner{\mathcal{P} [\Delta, B_i]B_if_n + \mathcal{P}B_i\Delta B_i f_n}{Af_n} - \inner{AB_if_n}{\mathcal{P}[\Delta, B_i]f_n + \mathcal{P}B_i\Delta f_n}\\
        &= - \inner{\mathcal{P} [\Delta, B_i]B_if_n}{Af_n} + \inner{\mathcal{P}B_i A B_i f_n}{Af_n} - \inner{AB_if_n}{\mathcal{P}[\Delta, B_i]f_n} + \inner{AB_if_n}{ \mathcal{P}B_iA f_n}\\
        &= \inner{B_i A B_i f_n}{Af_n} + \inner{AB_if_n}{B_iA f_n} - \inner{\mathcal{P} [\Delta, B_i]B_if_n}{Af_n} - \inner{AB_if_n}{\mathcal{P}[\Delta, B_i]f_n}\\
        &= \inner{AB_if_n}{(B_i + B_i^*)A f_n} - \inner{\mathcal{P} [\Delta, B_i]B_if_n}{Af_n} - \inner{AB_if_n}{\mathcal{P}[\Delta, B_i]f_n}\\
        &= \inner{AB_if_n}{(\mathcal{T}_{\xi_i} + \mathcal{T}_{\xi_i}^*)A f_n} - \inner{\mathcal{P} [\Delta, B_i]B_if_n}{Af_n} - \inner{AB_if_n}{\mathcal{P}[\Delta, B_i]f_n}.
    \end{align*}
    We shall treat each term individually using Cauchy-Schwartz and Young's Inequality in conjunction with the relevant bounds as seen in Lemma \ref{preliminary bound for cauchy}:
    \begin{align*}
    \inner{AB_if_n}{(\mathcal{T}_{\xi_i}+ \mathcal{T}_{\xi_i}^*)A f_n} &\leq \norm{AB_if_n}\norm{(\mathcal{T}_{\xi_i}+ \mathcal{T}_{\xi_i}^*)A f_n}\\
    &\leq c\norm{\xi_i}_{W^{1,\infty}}\norm{B_if_n}_{W^{2,2}}\norm{Af_n}\\ &\leq c\norm{\xi_i}_{W^{1,\infty}}\norm{\xi_i}_{W^{3,\infty}}\norm{f_n}_{W^{3,2}}\norm{f_n}_2\\
    &\leq c(\varepsilon)\norm{\xi_i}^2_{W^{3,\infty}}\norm{f_n}_2^2 + \frac{\varepsilon}{3} \norm{\xi_i}_{W^{3,\infty}}^2\norm{f_n}_{3}^2
    \end{align*}
    as well as 
    \begin{align*}
       - \inner{\mathcal{P} [\Delta, B_i]B_if_n}{Af_n} &\leq \norm{\mathcal{P} [\Delta, B_i]B_if_n}\norm{Af_n}\\
       &\leq c\norm{[\Delta, B_i]B_if_n}\norm{f_n}_2\\
       &\leq c\norm{\xi_i}_{W^{3,\infty}}\norm{B_if_n}_{W^{2,2}}\norm{f_n}_2\\
       &\leq c\norm{\xi_i}_{W^{3,\infty}}^2\norm{f_n}_{W^{3,2}}\norm{f_n}_2\\
       &\leq c(\varepsilon)\norm{\xi_i}^2_{W^{3,\infty}}\norm{f_n}_2^2 + \frac{\varepsilon}{3} \norm{\xi_i}_{W^{3,\infty}}^2\norm{f_n}_{3}^2
    \end{align*}
    and finally 
    \begin{align*}
        - \inner{AB_if_n}{\mathcal{P}[\Delta, B_i]f_n} &\leq \norm{AB_if_n}\norm{\mathcal{P}[\Delta, B_i]f_n}\\
        &\leq c\norm{B_if_n}_{W^{2,2}}\norm{[\Delta, B_i]f_n}\\
        &\leq c\norm{\xi_i}_{W^{3,\infty}}^2\norm{f_n}_{W^{3,2}}\norm{f_n}_{W^{2,2}}\\
       &\leq c(\varepsilon)\norm{\xi_i}^2_{W^{3,\infty}}\norm{f_n}_2^2 + \frac{\varepsilon}{3} \norm{\xi_i}_{W^{3,\infty}}^2\norm{f_n}_{3}^2
    \end{align*}
    Summing these up completes the proof.

\end{proof}

\begin{proof}[Asumption \ref{assumptions for uniform bounds2}:]
Lemmas \ref{uniformdelanonlinear} and \ref{uniformdelanoise} will be our basis of showing (\ref{uniformboundsassumpt1}). The task is to control $$2\left \langle \mathcal{P}_n\left(-\mathcal{P}\mathcal{L} -\nu A +\frac{1}{2}\sum_{i=1}^\infty \mathcal{P}B_i^2\right) f_n , f_n      \right \rangle_2 + \sum_{i=1}^\infty \norm{\mathcal{P}_n\mathcal{P}B_if_n}_2^2$$ 
    which we rewrite as
    \begin{equation} \label{the rewrite}-2\inner{\mathcal{P}_n\mathcal{P}\mathcal{L}_{f_n}f_n}{f_n}_2 - 2\nu\inner{\mathcal{P}_nAf_n}{f_n}_2 + \sum_{i=1}^\infty\left(\inner{\mathcal{P}_n\mathcal{P}B_i^2f_n}{f_n}_2 + \norm{\mathcal{P}_n\mathcal{P}B_if_n}_2^2 \right).\end{equation}
    Recalling the assumption (\ref{bounds on xi_i}) and Lemmas \ref{uniformdelanonlinear}, \ref{uniformdelanoise}, we have that for any $\varepsilon > 0$,
    \begin{align*}(\ref{the rewrite}) &\leq - 2\nu\inner{\mathcal{P}_nAf_n}{f_n}_2 + c(\varepsilon)\norm{f_n}^4_2 + \varepsilon \norm{f_n}_{3}^2 + \sum_{i=1}^\infty\left(c(\varepsilon)\norm{\xi_i}^2_{W^{3,\infty}}\norm{f_n}_2^2 + \varepsilon \norm{\xi_i}_{W^{3,\infty}}^2\norm{f_n}_{3}^2 \right)\\
    &= - 2\nu\inner{Af_n}{f_n}_2 + \left[ c(\varepsilon)\norm{f_n}_2^4 + c(\varepsilon)\sum_{i=1}^\infty \norm{\xi_i}_{W^{3,\infty}}^2\norm{f_n}_2^2\right] + \varepsilon\left[1 + \sum_{i=1}^\infty \norm{\xi_i}_{W^{3,\infty}}^2\right]\norm{f_n}_3^2\\
    &\leq - 2\nu\inner{A^{1/2}Af_n}{A^{3/2}f_n}_2 + c(\varepsilon)\left[1 + \norm{f_n}_2^2\right]\norm{f_n}_2^2 + \varepsilon\left[1 + \sum_{i=1}^\infty \norm{\xi_i}_{W^{3,\infty}}^2\right]\norm{f_n}_3^2\\
    &= - 2\nu\norm{f_n}_3^2 + c(\varepsilon)\left[1 + \norm{f_n}_2^2\right]\norm{f_n}_2^2 + \varepsilon\left[1 + \sum_{i=1}^\infty \norm{\xi_i}_{W^{3,\infty}}^2\right]\norm{f_n}_3^2
    \end{align*}
    where we have embedded the $\sum_{i=1}^\infty \norm{\xi_i}_{W^{3,\infty}}^2$ into the constant $c(\varepsilon)$. Therefore by choosing $$\varepsilon:= \frac{\nu}{1 + \sum_{i=1}^\infty \norm{\xi_i}_{W^{3,\infty}}^2}$$ 
    then (\ref{uniformboundsassumpt1}) is satisfied for $\kappa:= \nu$. Moving on to (\ref{uniformboundsassumpt2}), we are interested in the term $$\sum_{i=1}^\infty \inner{\mathcal{P}_n\mathcal{P}B_if_n}{f_n}_2^2.$$ Using Lemmas \ref{APequalsA} and \ref{PBiequalsPBiP} once more, observe that
    \begin{align*}
        \inner{\mathcal{P}_n\mathcal{P}B_if_n}{f_n}_2^2 &= \inner{\mathcal{P}B_if_n}{f_n}_2^2\\
        &= \inner{A\mathcal{P}B_if_n}{Af_n}^2\\
        &= \inner{A B_if_n}{Af_n}^2\\
        &=  \inner{\mathcal{P}[\Delta, B_i]f_n + \mathcal{P}B_i \Delta f_n}{Af_n}^2\\
        &\leq 2\inner{\mathcal{P}[\Delta, B_i]f_n}{Af_n}^2 + 2\inner{\mathcal{P}B_i \Delta f_n}{Af_n}^2\\
        &= 2\inner{\mathcal{P}[\Delta, B_i]f_n}{Af_n}^2 + 2\inner{\mathcal{P}B_i A f_n}{Af_n}^2\\
        &= 2\inner{[\Delta, B_i]f_n}{Af_n}^2 + 2\inner{B_i A f_n}{Af_n}^2.
    \end{align*}
    The first of these terms is dealt with through a simple Cauchy-Schwartz, as
    \begin{align*}
        \inner{[\Delta, B_i]f_n}{Af_n}^2 &\leq \norm{[\Delta, B_i]f_n}^2\norm{Af_n}^2\\
        &\leq c\norm{\xi_i}^2_{W^{3,\infty}}\norm{f_n}^2_{W^{2,2}}\norm{f_n}^2_2\\
        &\leq c\norm{\xi_i}^2_{W^{3,\infty}}\norm{f_n}^4_{2}
    \end{align*}
    using Proposition \ref{boundoncommutator}, and the second comes directly from (\ref{finalboundinderivativeproof}) as
    \begin{align*}
        \inner{B_i A f_n}{Af_n}^2
        \leq c\norm{\xi_i}^2_{W^{1,\infty}}\norm{Af_n}^4
        \leq c\norm{\xi_i}^2_{W^{3,\infty}}\norm{f_n}_2^4.
    \end{align*}
    Summing up the two terms and over all $i$ gives that
    $$\sum_{i=1}^\infty \inner{\mathcal{P}_n\mathcal{P}B_if_n}{f_n}_2^2 \leq \left(c\sum_{i=1}^\infty\norm{\xi_i}^2_{W^{3,\infty}} \right)\norm{f_n}_2^4 $$
    which justifies (\ref{uniformboundsassumpt2}) and Assumption \ref{assumptions for uniform bounds2}.
\end{proof}

Towards Assumption \ref{therealcauchy assumptions} we again state an intermediary lemma. 

\begin{lemma} \label{thecauchynonlinear}
    For any $\varepsilon > 0$, we have that
    \begin{align*}\left\vert\inner{\mathcal{P}\mathcal{L}_{f}f - \mathcal{P}\mathcal{L}_{g}g}{f-g}_1\right\vert \leq c(\varepsilon)\left(\norm{g}^4_1 + \norm{f}_2^2\right) \norm{f-g}_1^2  + \varepsilon \norm{f-g}_2^2 
    \end{align*}
\end{lemma}

\begin{proof}
    Observe that
    \begin{align*}
         \inner{\mathcal{P}\mathcal{L}_{f}f - \mathcal{P}\mathcal{L}_{g}g}{f-g}_1
        &= \inner{\mathcal{P}\mathcal{L}_{f}f - \mathcal{L}_{g}g}{A(f-g)}\\
        &= \inner{\mathcal{L}_{f}f - \mathcal{L}_{g}g}{A(f-g)}\\
        &= \inner{\mathcal{L}_{f}f - \mathcal{L}_{f}g + \mathcal{L}_{f}g - \mathcal{L}_{g}g}{A(f-g)}\\
        &= \inner{\mathcal{L}_{f - g}f +  \mathcal{L}_{g}(f - g)}{A(f-g)}
    \end{align*}
    and so it is sufficient to control the terms
    \begin{align}
        &\left\vert\inner{\mathcal{L}_{f - g}f}{A(f-g)}\right\vert \label{cauchynonlinear1},\\
          &\left\vert\inner{\mathcal{L}_{g}(f - g)}{A(f-g)}\right\vert \label{cauchynonlinear2}
    \end{align}
    We bound each item individually, using (\ref{nonlinear eq1}):
    \begin{align*}
        (\ref{cauchynonlinear1}) &\leq \norm{\mathcal{L}_{f - g}f}\norm{A(f-g)}\\
        &\leq c\norm{f-g}_1\norm{f}_1^{1/2}\norm{f}_2^{1/2}\norm{f-g}_2\\
        &\leq c\norm{f-g}_1\norm{f}_2\norm{f-g}_2\\
        &\leq c(\varepsilon)\norm{f}_2^2\norm{f-g}_1^2 + \frac{\varepsilon}{2} \norm{f-g}_2^2
    \end{align*}
    and
    \begin{align*}
        (\ref{cauchynonlinear2}) &\leq \norm{\mathcal{L}_{g}(f - g)}\norm{A(f-g)}\\
        &\leq c\norm{g}_1\norm{f - g}_1^{1/2}\norm{f - g}_2^{1/2}\norm{f - g}_2\\
        &= c\norm{g}_1\norm{f - g}_1^{1/2}\norm{f - g}_2^{3/2}\\
        &\leq c(\varepsilon)\norm{g}^4_1\norm{f - g}_1^{2} + \frac{\varepsilon}{2}\norm{f - g}_2^{2}
    \end{align*}
    using Young's Inequality with conjugate exponents $4$ and $4/3$.

\end{proof}

\begin{proof}[Assumption \ref{therealcauchy assumptions}:]
    For (\ref{therealcauchy1}) we must bound the term
    \begin{align} 
        \nonumber 2\left \langle\left(-\mathcal{P}\mathcal{L} -\nu A +\frac{1}{2}\sum_{i=1}^\infty \mathcal{P}B_i^2\right) f - \left(-\mathcal{P}\mathcal{L} -\nu A +\frac{1}{2}\sum_{i=1}^\infty \mathcal{P}B_i^2\right) g, f-g  \right \rangle_1  + \sum_{i=1}^\infty \norm{\mathcal{P}B_if - \mathcal{P}B_ig}_1^2
    \end{align}
    which we simply rewrite as
    \begin{align*}
        -2\inner{\mathcal{P}\mathcal{L}_ff - \mathcal{P}\mathcal{L}_gg}{f-g}_1 -2\nu \inner{A(f-g)}{f-g}_1 + \sum_{i=1}^\infty \left(\inner{\mathcal{P}B_i^2(f-g)}{f-g}_1 + \norm{\mathcal{P}B_i(f-g)}_1^2 \right)
    \end{align*}
    and inspect the distinct items individually. Firstly from Lemma \ref{thecauchynonlinear} we have that for any $\varepsilon >0$,
    \begin{equation}\label{mr1}-2\inner{\mathcal{P}\mathcal{L}_ff - \mathcal{P}\mathcal{L}_gg}{f-g}_1 \leq c(\varepsilon)\left(\norm{g}^4_1 + \norm{f}_2^2\right) \norm{f-g}_1^2  + \varepsilon \norm{f-g}_2^2.\end{equation}
    Similarly to the justification of Assumption \ref{assumptions for uniform bounds2} we also see that
    \begin{equation}\label{mr2}-2\nu \inner{A(f-g)}{f-g}_1 = -2\nu \norm{f-g}_2^2. \end{equation}
    Shifting focus to the final term, note that in Lemma \ref{preliminary bound for cauchy} we in fact showed that
    $$\inner{\mathcal{P}B_i^2f_n}{f_n}_1 + \inner{\mathcal{P}B_if_n}{\mathcal{P}B_if_n}_1 \leq  c(\varepsilon)\norm{\xi_i}_{W^{3,\infty}}^2\norm{f_n}_1^2 + \varepsilon \norm{\xi_i}_{W^{3,\infty}}^2\norm{f_n}_2^2$$ and scanning the proof we see that all arguments hold for arbitrary $f_n \in W^{3,2}_{\sigma}(\T;\R^N)$ so we can deduce directly the bound
    \begin{equation}\label{mr3}\inner{\mathcal{P}B_i^2(f-g)}{f-g}_1 + \norm{\mathcal{P}B_i(f-g)}_1^2 \leq  c(\varepsilon)\norm{\xi_i}_{W^{3,\infty}}^2\norm{f-g}_1^2 + \varepsilon \norm{\xi_i}_{W^{3,\infty}}^2\norm{f-g}_2^2.\end{equation}
    Summing over (\ref{mr1}), (\ref{mr2}) and all $i$ in (\ref{mr3}), we deduce a bound by
    $$-2\nu \norm{f-g}_2^2 + c(\varepsilon)\left[\norm{g}^4_1 + \norm{f}_2^2 + \sum_{i=1}^\infty\norm{\xi_i}_{W^{3,\infty}}^2 \right]\norm{f-g}_1^2 + \varepsilon\left[1 + \sum_{i=1}^\infty\norm{\xi_i}_{W^{3,\infty}}^2 \right]\norm{f-g}_2^2$$
    so again a choice of \begin{equation}\label{familiar choice}\varepsilon:= \frac{\nu}{1 + \sum_{i=1}^\infty \norm{\xi_i}_{W^{3,\infty}}^2}\end{equation}
    ensures (\ref{therealcauchy1}) is satisfied for $\kappa:= \nu$. Moving on to (\ref{therealcauchy2}), we are interested in the term \begin{equation} \label{of interest1} \sum_{i=1}^\infty \inner{\mathcal{P}B_i(f-g)}{f-g}_1^2, \end{equation}
    noting that
       \begin{align*}
        \inner{\mathcal{P}B_i(f-g)}{f-g}_1^2 &= \inner{A\mathcal{P}B_i(f-g)}{f-g}^2\\
        &= \inner{A B_i(f-g)}{f-g}^2\\
        &= \inner{\Delta  B_i(f-g)}{f-g}^2\\
        &= \left\langle \sum_{k=1}^N \partial_k^2 B_i(f-g),f-g \right \rangle^2\\
        &= \left(\sum_{k=1}^N\inner{\partial_k B_i(f-g)}{\partial_k(f - g)} \right)^2\\
        &\leq N\sum_{k=1}^N\inner{\partial_k B_i(f-g)}{\partial_k(f - g)}^2
    \end{align*}
    using Lemma \ref{APequalsA}. Combining the ideas of Lemmas \ref{biggestcorcor} and \ref{boundsonL_xi}, we have $$\partial_kB_i(f-g) = B_{\partial_k\xi_i}(f-g) + B_i\partial_k(f-g)$$ so 
    $$\inner{\partial_k B_i(f-g)}{\partial_k(f - g)}^2 \leq 2\inner{B_{\partial_k\xi_i}(f-g)}{\partial_k(f - g)}^2+ 2\inner{B_i\partial_k(f-g)}{\partial_k(f-g)}^2.$$
    Now from (\ref{boundsonB_i}),
    \begin{align*}
        \inner{B_{\partial_k\xi_i}(f-g)}{\partial_k(f - g)}^2 &\leq \norm{B_{\partial_k\xi_i}(f-g)}^2\norm{\partial_k(f - g)}^2\\
        &\leq c\norm{\partial_k\xi_i}^2_{W^{1,\infty}}\norm{f-g}^2_{W^{1,2}}\norm{\partial_k(f - g)}^2\\
        &\leq c\norm{\xi_i}^2_{W^{3,\infty}}\norm{f-g}^2_1\norm{f - g}^2_1\\
        &= c\norm{\xi_i}^2_{W^{3,\infty}}\norm{f-g}^4_1
    \end{align*}
    and from Corollary \ref{cancellationproperty},
\begin{align*}
    \inner{B_i\partial_k(f-g)}{\partial_k(f-g)}^2 &= \inner{\mathcal{T}_{\xi_i}\partial_k(f-g)}{\partial_k(f-g)}^2\\
    &\leq \norm{\mathcal{T}_{\xi_i}\partial_k(f-g)}^2\norm{\partial_k(f-g)}^2\\
    &\leq c\norm{\xi_i}^2_{W^{1,\infty}}\norm{\partial_k(f-g)}^2\norm{\partial_k(f-g)}^2\\
    &\leq c\norm{\xi_i}^2_{W^{3,\infty}}\norm{f-g}^4_1.
\end{align*}
    By summing both terms, over all $k =1, \dots, N$ and $i \in \N$, we have shown that
    \begin{equation}\label{have shown that}\sum_{i=1}^\infty\inner{\mathcal{P}B_i(f-g)}{f-g}_1^2 \leq \left(c\sum_{i=1}^\infty \norm{\xi_i}^2_{W^{3,\infty}}\right)\norm{f-g}^4_1\end{equation}
    demonstrating (\ref{therealcauchy2}) and hence Assumption \ref{therealcauchy assumptions}.
\end{proof}

\begin{proof}[Assumption \ref{assumption for prob in V}:]
    For (\ref{probability first}) we must bound the term    \begin{align} 
        \nonumber 2\left \langle\left(-\mathcal{P}\mathcal{L} -\nu A +\frac{1}{2}\sum_{i=1}^\infty \mathcal{P}B_i^2\right) f, f  \right \rangle_1  + \sum_{i=1}^\infty \norm{\mathcal{P}B_if}_1^2
    \end{align}
    which we simply rewrite as
    \begin{align}
        -2\inner{\mathcal{P}\mathcal{L}_ff}{f}_1 -2\nu \inner{Af}{f}_1 + \sum_{i=1}^\infty \left(\inner{\mathcal{P}B_i^2f}{f}_1 + \norm{\mathcal{P}B_if}_1^2 \right). \label{align no star}
    \end{align}
    The nonlinear term can be controlled precisely as done for (\ref{cauchynonlinear2}) to deduce that for any $\varepsilon > 0$, 
    $$\left\vert\inner{\mathcal{P}\mathcal{L}_ff}{f}_1\right\vert \leq c(\varepsilon)\norm{f}_1^6 + \varepsilon \norm{f}_2^2.$$
    Meanwhile across (\ref{mr2}) and (\ref{mr3}) we have that
    \begin{align*}-2\nu \inner{Af}{f}_1 &+ \sum_{i=1}^\infty \left(\inner{\mathcal{P}B_i^2f}{f}_1 + \norm{\mathcal{P}B_if}_1^2 \right)\\ & \qquad \qquad \leq -2\nu\norm{f}_2^2 + c(\varepsilon)\left[\sum_{i=1}^N\norm{\xi_i}_{W^{3,\infty}}^2\right]\norm{f}_1^2 + \varepsilon\left[\sum_{i=1}^\infty \norm{\xi_i}_{W^{3,\infty}}^2\right]\norm{f}_2^2
    \end{align*}
    so with the familiar choice of $\varepsilon$ (\ref{familiar choice}) we see that
    $$(\ref{align no star}) \leq  c\left(1 + \norm{f}_1^6 \right) - \nu\norm{f}_2^2$$
    which is more than sufficient to show (\ref{probability first}). (\ref{probability second}) follows immediately from $(\ref{have shown that})$, concluding the proof.
\end{proof}

\begin{proof}[Assumption \ref{finally the last assumption}:]
    For any $\eta \in W^{2,2}_{\sigma}(\T;\R^N)$ we must bound the term \begin{align} 
        \nonumber \left \langle\left(-\mathcal{P}\mathcal{L} -\nu A +\frac{1}{2}\sum_{i=1}^\infty \mathcal{P}B_i^2\right) f - \left(-\mathcal{P}\mathcal{L} -\nu A +\frac{1}{2}\sum_{i=1}^\infty \mathcal{P}B_i^2\right) g, \eta  \right \rangle_1
    \end{align}
    which we simply rewrite as
    \begin{align*}
        -2\inner{\mathcal{P}\mathcal{L}_ff - \mathcal{P}\mathcal{L}_gg}{\eta}_1 -2\nu \inner{A(f-g)}{\eta}_1 + \sum_{i=1}^\infty \inner{\mathcal{P}B_i^2(f-g)}{\eta}_1
    \end{align*}
    and further by 
      \begin{align*}
        -2\inner{\mathcal{P}\mathcal{L}_ff - \mathcal{P}\mathcal{L}_gg}{A\eta} -2\nu \inner{A(f-g)}{A\eta} + \sum_{i=1}^\infty \inner{\mathcal{P}B_i^2(f-g)}{A\eta}.
    \end{align*}
    Through Cauchy-Schwartz this is controlled by 
      \begin{align*}
       \norm{\eta}_2\left( 2\norm{\mathcal{P}\mathcal{L}_ff - \mathcal{P}\mathcal{L}_gg} + 2\nu \norm{A(f-g)} + \sum_{i=1}^\infty \norm{\mathcal{P}B_i^2(f-g)}\right).
    \end{align*}
    so our problem is reduced to bounding the bracketed terms. The linear terms are trivial when recalling (\ref{boundsonB_i}), and for the nonlinear term we revert back to (\ref{first begin align}) to see that
    \begin{align*}
        \norm{\mathcal{P}\mathcal{L}_ff - \mathcal{P}\mathcal{L}_gg} &\leq \norm{\mathcal{L}_{f-g}f}+\norm{\mathcal{L}_g(f-g)}\\
        &\leq c\left(\norm{f-g}_1\norm{f}_2 +  \norm{g}_1\norm{f-g}_2\right)\\
        &\leq c\left(\norm{f}_2+\norm{g}_1\right)\norm{f-g}_2
    \end{align*}
    comfortably justifying the assumption.
\end{proof}

Moving on now to the setting of Appendix V, \ref{appendix v}, the new space $X$ will again be $L^2_{\sigma}(\T;\R^N)$ as laid out in Subsection \ref{subsection ito}. We choose the bilinear form $\inner{\cdot}{\cdot}_{X \times H}$ to be given by \begin{equation}\label{given bilinear form}\inner{f}{g}_{X \times H}:= \inner{f}{Ag}\end{equation} noting that the property (\ref{bilinear form}) follows from Proposition \ref{prop for moving stokes}. Noting also that the system $(a_k)$ given in Proposition \ref{basis of stokes} is an orthogonal basis of $W^{1,2}_{\sigma}(\T;\R^N)$, and that the operators were shown to be measurable into the relevant spaces in Subsection \ref{subsection ito}, we are in the setting of Appendix V.  We now proceed to justify the assumptions of this appendix.

\begin{proof}[Assumption \ref{gg3}:]
    This follows identically to Assumption \ref{new assumption 1}, referring again to Subsection \ref{subsection ito} and Lemma \ref{basic nonlinear bound}.
\end{proof}

The process to justify Assumption \ref{uniqueness for H valued} is the same as for \ref{therealcauchy assumptions}, and we introduce the corresponding Lemma to \ref{thecauchynonlinear}.

\begin{lemma} \label{lemmalemma}
    For any $\varepsilon > 0$, we have that \begin{align*}\left\vert\inner{\mathcal{P}\mathcal{L}_{f}f - \mathcal{P}\mathcal{L}_{g}g}{f-g}\right\vert \leq c(\varepsilon)\norm{f}_2^2 \norm{f-g}^2  + \varepsilon \norm{f-g}_1^2 .
    \end{align*}
\end{lemma}

\begin{proof}
    As in Lemma \ref{thecauchynonlinear}, we use the inequality
    $$\left\vert\inner{\mathcal{P}\mathcal{L}_{f}f - \mathcal{P}\mathcal{L}_{g}g}{f-g}\right\vert \leq \left\vert\inner{\mathcal{L}_{f-g}f}{f-g}\right\vert + \left\vert\inner{\mathcal{L}_{g}(f-g)}{f-g}\right\vert.$$
    For the first term, appealing to (\ref{first begin align}), observe that
    \begin{align*}
        \left\vert\inner{\mathcal{L}_{f-g}f}{f-g}\right\vert &\leq \norm{\mathcal{L}_{f-g}f}\norm{f-g}\\
        &\leq \norm{f-g}_1\norm{f}_2\norm{f-g}\\
        &\leq c(\varepsilon)\norm{f}_2^2 \norm{f-g}^2  + \varepsilon \norm{f-g}_1^2.
    \end{align*}
    In fact the second term is null due to Corollary \ref{cancellationproperty}, which concludes the proof.
\end{proof}

\begin{proof}[Assumption \ref{uniqueness for H valued}:]
    Continuing to consider the distinct terms, we have that
    $$-2\nu \inner{A(f-g)}{f-g} = -2\nu\norm{f-g}_1^2$$ and
    \begin{align*}\inner{\mathcal{P}B_i^2(f-g)}{f-g} + \norm{\mathcal{P}B_i(f-g)}^2 & \leq \inner{B_i^2(f-g)}{f-g} + \norm{B_i(f-g)}^2\\&\leq  c\norm{\xi_i}_{W^{2,\infty}}^2\norm{f-g}^2.\end{align*}
    from (\ref{combinedterminenergyinequality}). With these components in place, the proof of (\ref{therealcauchy1*}) then follows identically to that of (\ref{therealcauchy1}). (\ref{therealcauchy2*}) is a direct consequence of (\ref{finalboundinderivativeproof}), concluding the justification.
\end{proof}

 \begin{proof}[Assumption \ref{assumption for probability in H}:]
    This stronger Assumption was in fact already verified in the address of Assumption \ref{assumption for prob in V}.
\end{proof}

 \newpage 

\section{Analysis of the Vorticity Equation} \label{section vorticity form}

In order to address the well-posedness problem of the SALT Navier-Stokes Equation on the bounded domain, we now pose it in vorticity form. The analysis conducted in Subsection \ref{subsection velocity existence} was done with reference to the properties derived across Subsections \ref{functional framework subsection} and \ref{subsection the salt operator}, applicable to the bounded domain as well as the torus. The issue in studying the velocity form is that our operators do not map into the correct spaces in order to use these properties: in particular, the Leray Projector does not preserve the zero trace property and so the operators do not map into the necessary $W^{k,2}_{\sigma}(\mathscr{O};\R^N)$ spaces. The motivation behind the vorticity form is to circumvent the necessity of Leray Projection.\\

Throughout this section we shall specifically work in the case $N=3$, understanding that the results similarly hold for $N=2$. We do this to work with an explicit form of the curl as seen below. Similarly our attentions shall be decidedly on the bounded domain $\mathscr{O}$, though the results carry over seamlessly to the torus $\T$. For this section we impose new constraints on the $\xi_i$, which are such that for each $i \in \N$, $\xi_i \in W^{1,2}_{\sigma}(\mathscr{O};\R^3) \cap W^{2,2}_{0}(\mathscr{O};\R^3) \cap W^{3,\infty}(\mathscr{O};\R^3)$ and they collectively satisfy
\begin{equation} \label{new xi i bound}
    \sum_{i=1}^\infty \norm{\xi_i}_{W^{3,\infty}}^2 < \infty.
\end{equation}

\subsection{Deriving the Equation} \label{subsection deriving the equation}

The vorticity form of the equation is derived through taking the curl of the velocity form, where the curl operator is defined for $f \in W^{1,2}(\mathscr{O};\R^3)$ by $$\textnormal{curl}f :=
\begin{pmatrix}
\partial_2 f^3 - \partial_3 f^2\\
\partial_3f^1 - \partial_1f^3\\
\partial_1f^2 - \partial_2f^1
\end{pmatrix}.
$$
We introduce the operator $\mathscr{L}$ defined on sufficiently regular functions $f,g: \mathscr{O} \rightarrow \R^3$ by  \begin{equation} \label{new vorticity operator}
    \mathscr{L}_fg := \mathcal{L}_fg - \mathcal{L}_gf.
\end{equation}
In [\cite{Robinson}] Subesction 12.1 it is shown that, with notation $\phi:= \textnormal{curl}f$, $$\textnormal{curl}\left(\mathcal{L}_ff\right) = \mathscr{L}_f \phi$$ where it is also observed that the curl of elements of $W^{1,2}(\mathscr{O};\R^3) \cap L^{2,\perp}_{\sigma}(\mathscr{O};\R^3)$ is null. It is clear that the Laplacian commutes with the curl operation, and so in looking to take the curl of equation (\ref{number2equation}) it only remains to consider the SALT Operator $B$.

\begin{proposition}
   We have that \begin{equation} \label{vorticity prop}
        \textnormal{curl}(B_if) = \mathscr{L}_{\xi_i}\phi
    \end{equation}
    where once more $\phi := \textnormal{curl}f.$
\end{proposition}

\begin{proof}
    We shall show only that the identity (\ref{vorticity prop}) holds in its first component, with the others following similarly. To this end we calculate the first component of the left hand side of (\ref{vorticity prop}):
    \begin{align*}
        [\textnormal{curl}(B_if)]^1 &= \partial_2[B_if]^3 - \partial_3[B_if]^2 \\
        &= \partial_2\left(\sum_{j=1}^3\xi_i^j\partial_jf^3 + f^j\partial_3\xi_i^j\right) - \partial_3\left(\sum_{j=1}^3\xi_i^j\partial_jf^2 + f^j\partial_2\xi_i^j\right)\\
        &= \sum_{j=1}^3\left(\partial_2\xi_i^j\partial_jf^3 + \xi_i^j\partial_2\partial_jf^3 + \partial_2f^j\partial_3\xi_i^j  + f^j\partial_2\partial_3\xi_i^j\right)\\ & \qquad \qquad - \sum_{j=1}^3\left( \partial_3\xi_i^j\partial_jf^2 + \xi_i^j\partial_3\partial_jf^2 + \partial_3f^j\partial_2\xi_i^j  + f^j\partial_3\partial_2\xi_i^j\right)\\
        &= \sum_{j=1}^3\left(\partial_2\xi_i^j\partial_jf^3 + \xi_i^j\partial_2\partial_jf^3 + \partial_2f^j\partial_3\xi_i^j - \partial_3\xi_i^j\partial_jf^2 - \xi_i^j\partial_3\partial_jf^2 - \partial_3f^j\partial_2\xi_i^j  \right)\\
        &= \sum_{j=1}^3\xi_i^j\partial_j(\partial_2f^3 - \partial_3f^2) + \sum_{j=1}^3\left(\partial_2\xi_i^j\partial_jf^3 + \partial_2f^j\partial_3\xi_i^j - \partial_3\xi_i^j\partial_jf^2 - \partial_3f^j\partial_2\xi_i^j  \right)\\
        &= [\mathcal{L}_{\xi_i}\phi]^1 + \sum_{j=1}^3\left(\partial_2\xi_i^j\partial_jf^3 + \partial_2f^j\partial_3\xi_i^j - \partial_3\xi_i^j\partial_jf^2 - \partial_3f^j\partial_2\xi_i^j  \right).
    \end{align*}
    Therefore it remains to show that \begin{equation} \label{vort id to show}
        \sum_{j=1}^3\left(\partial_2\xi_i^j\partial_jf^3 + \partial_2f^j\partial_3\xi_i^j - \partial_3\xi_i^j\partial_jf^2 - \partial_3f^j\partial_2\xi_i^j  \right) = -[\mathcal{L}_{\phi}\xi_i]^1.
    \end{equation}
    We expand the sum in (\ref{vort id to show}) to 
    \begin{align*}
        &\left(\partial_2\xi_i^1\partial_1f^3 + \partial_2f^1\partial_3\xi_i^1 - \partial_3\xi_i^1\partial_1f^2 - \partial_3f^1\partial_2\xi_i^1\right) + \left(\partial_2\xi_i^2\partial_2f^3 + \partial_2f^2\partial_3\xi_i^2 - \partial_3\xi_i^2\partial_2f^2 - \partial_3f^2\partial_2\xi_i^2 \right)\\ &+ \left(\partial_2\xi_i^3\partial_3f^3 + \partial_2f^3\partial_3\xi_i^3 - \partial_3\xi_i^3\partial_3f^2 - \partial_3f^3\partial_2\xi_i^3 \right)
    \end{align*}
    achieving some immediate cancellation in the second two brackets to
        \begin{align*}
        \left(\partial_2\xi_i^1\partial_1f^3 + \partial_2f^1\partial_3\xi_i^1 - \partial_3\xi_i^1\partial_1f^2 - \partial_3f^1\partial_2\xi_i^1\right) + \left(\partial_2\xi_i^2\partial_2f^3 - \partial_3f^2\partial_2\xi_i^2 \right)+ \left( \partial_2f^3\partial_3\xi_i^3 - \partial_3\xi_i^3\partial_3f^2 \right).
    \end{align*}
    We now simply rewrite the above by combining like terms, into
    \begin{equation} \nonumber
        \partial_2\xi_i^1(\partial_1f^3 - \partial_3f^1) + \partial_3\xi_i^1(\partial_2f^1 - \partial_1f^2) + (\partial_2\xi_i^2 + \partial_3\xi_i^3)(\partial_2f^3 - \partial_3f^2)
    \end{equation}
    or more succinctly as
    $$-\partial_2\xi_i^1\phi^2 - \partial_3\xi_i^1\phi^3 + (\partial_2\xi_i^2 + \partial_3\xi_i^3)\phi^1$$
    to which we add and subtract $\partial_1\xi_i^1\phi^1$ to arrive at $$ -\sum_{j=1}^3 \phi^j\partial_j\xi_i^1 +\sum_{j=1}^3\left(\partial_j\xi_i^j \right)\phi^1.$$
    The first term is precisely $-[\mathcal{L}_{\phi}\xi_i]^1$ as we wished to show, appreciating that the second term is zero given the divergence free condition on $\xi_i$. 
    
\end{proof}
From this point forwards we adopt familiar notation of $\mathscr{L}_i := \mathscr{L}_{\xi_i}$. Writing the Stratonovich integral of (\ref{number2equation}) in its component form over the basis vectors of $\mathfrak{U}$, and introducing the notation $w:= \textnormal{curl}u$, at a heuristic level we can take the curl of (\ref{number2equation}) to obtain \begin{equation}\label{vorticity strat} w_t - w_0 + \int_0^t\mathscr{L}_{u_s}w_s\,ds - \nu\int_0^t \Delta w_s\, ds + \sum_{i=1}^\infty \int_0^t \mathscr{L}_iw \circ dW^i_s = 0.\end{equation}
Having already braved the rigorous It\^{o} conversion of (\ref{projected strato}) we shall make the conversion without explicit reference to the conditions of Appendix III (\ref{appendix iii}), for reasons elucidated at the start of Subsection \ref{subsection vorticity existence}. So at least again then at the heuristic level, the It\^{o} form is 
\begin{equation}\nonumber w_t = w_0 - \int_0^t\mathscr{L}_{u_s}w_s\,ds + \nu\int_0^t \Delta w_s\, ds + \frac{1}{2}\int_0^t\sum_{i=1}^\infty \mathscr{L}_i^2w_sds- \sum_{i=1}^\infty \int_0^t \mathscr{L}_iw \, dW^i_s\end{equation}
which can again be projected to the equation
\begin{equation}\label{vorticity ito} w_t = w_0 - \int_0^t\mathcal{P}\mathscr{L}_{u_s}w_s\,ds - \nu\int_0^t A w_s\, ds + \frac{1}{2}\int_0^t\sum_{i=1}^\infty\mathcal{P} \mathscr{L}_i^2w_sds- \sum_{i=1}^\infty \int_0^t \mathcal{P}\mathscr{L}_iw \, dW^i_s.\end{equation}
Having motivated this section by an avoidance of the Leray Projection this may seem counter intuitive, however we shall shortly show that the projection is not felt in the noise (where it becomes problematic in velocity form). The goal is to deduce the existence of a unique maximal solution of (\ref{vorticity ito}), a task which requires some clarification. Having reached (\ref{vorticity ito}) from the velocity form, we now look to solve the equation for vorticity which demands a representation of the velocity $u$ in terms of the vorticity $w$. For this we quote Theorem 1 of [\cite{Biot}] (or in fact, a slightly relaxed version).

\begin{theorem} \label{Biot theorem}
There exists a mapping $K: \mathscr{O} \times \mathscr{O} \rightarrow \R^3$ such that for every $\phi \in W^{1,2}_{\sigma}(\mathscr{O};\R^3) \cap W^{k,p}(\mathscr{O};\R^3)$ where $k \in \N \cup \{0\}$, $1 < p <\infty$, the function $f:\mathscr{O} \rightarrow \R^3$ defined $\lambda-a.e.$ by \begin{equation}\label{defining the function}f(x) = \int_{\mathscr{O}}K(x,y)\phi(y)dy\end{equation} is such that:
\begin{enumerate}
    \item $f \in L^{2}_{\sigma}(\mathscr{O};\R^3) \cap W^{k+1,p}(\mathscr{O};\R^3)$;
    \item $\textnormal{curl}f = \phi$;
    \item \label{regularity biot} There exists a constant $C$ independent of $\phi$ (but dependent on $k,p$) such that $$\norm{f}_{W^{k+1,p}} \leq C\norm{\phi}_{W^{k,p}}.$$
\end{enumerate}
\end{theorem}
It should immediately be noted that such a $K$ is not claimed to be unique, and in [\cite{Biot}] is explicitly shown to be non-unique, however it does allow us to identify \textit{a} velocity from a given vorticity satisfying the divergence-free and boundary conditions (\ref{boundary condition}). From this point forwards we fix a specific $K$ from the class of admissable integral kernels postulatied in Theorem \ref{Biot theorem}. We thus understand the nonlinear term as a mapping $$\phi \mapsto \mathscr{L}_f \phi$$ where $f$ is defined as in (\ref{defining the function}). This mapping shall at times be simply denoted by $\mathscr{L}$. The equation (\ref{vorticity ito}) is thus closed in $w$.

\subsection{Existence, Uniqueness and Maximality} \label{subsection vorticity existence}

We now state and prove the existence, uniqueness and maximality results for (\ref{vorticity ito}). We recall that to solve the velocity form (\ref{projected Ito}) we used the extended criterion of Appendix V (\ref{appendix v}), requiring the space $V:=W^{3,2}_{\sigma}(\T;\R^N)$ to prove Theorem \ref{theorem2NS}. This arose naturally in first showing Theorem \ref{existence for NS Strat}, where we considered solutions explicitly in terms of the original Stratonovich form. For (\ref{vorticity ito}), however, solutions can be obtained for the natural choice of $w_0 \in W^{1,2}_{\sigma}(\mathscr{O};\R^3)$ (so the equation satisfies its identity in $L^2(\mathscr{O};\R^3)$) with an application only of Theorem \ref{theorem1} in Appendix IV (\ref{appendix iv}). We thus do not take the detour of considering a fourth Hilbert Space to rigorously define solutions of the Stratonovich form (\ref{vorticity strat}), although this can be done similarly.

\begin{definition} \label{definitionofirregularsolutionNSw}
A pair $(w,\tau)$ where $\tau$ is a $\mathbb{P}-a.s.$ positive stopping time and $w$ is a process such that for $\mathbb{P}-a.e.$ $\omega$, $w_{\cdot}(\omega) \in C\left([0,T];W^{1,2}_{\sigma}\left(\mathscr{O};\R^3\right)\right)$ and $w_{\cdot}(\omega)\mathbbm{1}_{\cdot \leq \tau(\omega)} \in L^2\left([0,T];W^{2,2}_{\sigma}\left(\mathscr{O};\R^3\right)\right)$ for all $T>0$ with $w_{\cdot}\mathbbm{1}_{\cdot \leq \tau}$ progressively measurable in $W^{2,2}_{\sigma}\left(\mathscr{O};\R^3\right)$, is said to be a local strong solution of the equation (\ref{vorticity ito}) if the identity
\begin{equation} \nonumber
     w_t = w_0 - \int_0^{t\wedge \tau}\mathcal{P}\mathscr{L}_{u_s}w_s\,ds - \nu \int_0^{t \wedge \tau} A w_s\, ds + \frac{1}{2}\int_0^{t \wedge \tau}\sum_{i=1}^\infty\mathcal{P}\mathscr{L}^2_iw_s\,ds - \sum_{i=1}^\infty\int_0^{t\wedge \tau}\mathcal{P}\mathscr{L}_iw_s \, dW^i_s 
\end{equation}
holds $\mathbb{P}-a.s.$ in $L^2_{\sigma}\left(\mathscr{O};\R^N\right)$ for all $t \geq 0$.
\end{definition}

\begin{definition} \label{H valued maximal definition NSw}
A pair $(w,\Theta)$ such that there exists a sequence of stopping times $(\theta_j)$ which are $\mathbb{P}-a.s.$ monotone increasing and convergent to $\Theta$, whereby $(w_{\cdot \wedge \theta_j},\theta_j)$ is a local strong solution of the equation (\ref{vorticity ito}) for each $j$, is said to be a maximal strong solution of the equation (\ref{vorticity ito}) if for any other pair $(\eta,\Gamma)$ with this property then $\Theta \leq \Gamma$ $\mathbb{P}-a.s.$ implies $\Theta = \Gamma$ $\mathbb{P}-a.s.$.
\end{definition}

\begin{definition} \label{definition unique NSw}
A maximal strong solution $(w,\Theta)$ of the equation (\ref{vorticity ito}) is said to be unique if for any other such solution $(\eta,\Gamma)$, then $\Theta = \Gamma$ $\mathbb{P}-a.s.$ and for all $t \in [0,\Theta)$, \begin{equation} \nonumber\mathbb{P}\left(\left\{\omega \in \Omega: w_{t}(\omega) =  \eta_{t}(\omega)  \right\} \right) = 1. \end{equation}
\end{definition}

\begin{theorem} \label{theorem2NSw}
For any given $\mathcal{F}_0-$ measurable $w_0:\Omega \rightarrow W^{1,2}_{\sigma}(\mathscr{O};\R^N)$, there exists a unique maximal strong solution $(w,\Theta)$ of the equation (\ref{vorticity ito}). Moreover at $\mathbb{P}-a.e.$ $\omega$ for which $\Theta(\omega)<\infty$, we have that \begin{equation}\nonumber \sup_{r \in [0,\Theta(\omega))}\norm{w_r(\omega)}_{1}^2 + \int_0^{\Theta(\omega)}\norm{w_r(\omega)}_2^2dr = \infty.\end{equation}
\end{theorem}

As discussed the idea is to apply Theorem \ref{theorem1}, which we look to do for the spaces $$V:= W^{2,2}_{\sigma}(\mathscr{O};\R^3), \qquad H:= W^{1,2}_{\sigma}(\mathscr{O};\R^3), \qquad U:= L^{2}_{\sigma}(\mathscr{O};\R^3).$$ The density relations are clear as $C^{\infty}_{0,\sigma}(\mathscr{O};\R^3) \subset W^{2,2}_{\sigma}(\mathscr{O};\R^3)$ is dense in both $W^{1,2}_{\sigma}(\mathscr{O};\R^3)$ and $L^{2}_{\sigma}(\mathscr{O};\R^3)$. The bilinear form (\ref{bilinear formog}) is simply again (\ref{given bilinear form}). Now we shift attentions to checking that the operators are measurable into the correct spaces. We note that $\mathscr{L}$ has improved regularity properties over $\mathcal{L}$ given item \ref{regularity biot} of Theorem \ref{Biot theorem}, so retains the continuity observed in Lemma \ref{continuityofnonlinear} with measurability following. There is no change to the Stokes Operator from Section \ref{section velocity}. As for $\mathcal{P}\mathscr{L}_i$, we in fact first show that for $\mathscr{L}_i \in C\left(W^{2,2}_{\sigma}(\mathscr{O};\R^3);W^{1,2}_{\sigma}(\mathscr{O};\R^3)\right)$ (and hence is invariant under $\mathcal{P}$). This consists of three parts: showing that it is continuous as a mapping into $W^{1,2}(\mathscr{O};\R^3)$, showing the divergence free property and then the zero trace property. In fact with the appropriate regularity, it follows identically to Corollary \ref{corollary for B_i adjoint} that we again have \begin{equation} \label{curly l bounded result}\norm{\mathscr{L}_i\phi}_{W^{k,2}}^2 \leq c\norm{\xi_i}_{W^{k+1,\infty}}^2\norm{\phi}_{W^{k+1,2}}^2\end{equation} which addresses the continuity. The fact that $\mathscr{L}_i \phi$ is divergence free comes immediately from the relation $\mathscr{L}_i \phi = \textnormal{curl}\left(B_if\right)$ and the well established fact the divergence of a curl is zero. For the zero trace property it is sufficient to show the existence of a sequence of compactly supported $\eta_n \in W^{1,2}(\mathscr{O};\R^3)$ which converge to $\mathscr{L}_i\phi$ in $W^{1,2}(\mathscr{O};\R^3)$. By definition of the property that $\xi_i \in W^{2,2}_0(\mathscr{O}\R^3)$ there is a sequence $(\gamma_n)$, $\gamma_n \in C^{\infty}_0(\mathscr{O};\R^3)$ such that $\gamma_n \rightarrow \xi_i$ in $W^{2,2}(\mathscr{O}\R^3)$. Evidently $\eta_n:=\mathscr{L}_{\gamma_n}\phi$ is again compactly supported, and observe that
\begin{align*}
    \norm{\mathscr{L}_{\gamma_n}\phi - \mathscr{L}_i\phi }_{W^{1,2}} = \norm{\mathscr{L}_{\gamma_n-\xi_i}\phi }_{W^{1,2}} \leq c\norm{\gamma_n-\xi_i}_{W^{2,2}}\norm{\phi}_{W^{2,2}}
\end{align*}
from (\ref{third begin align}), which converges to zero as required to justify the zero trace property. 

\begin{remark}
As $\mathcal{P}\mathscr{L}_i$ is equal to $\mathscr{L}_i$ on $W^{2,2}_{\sigma}(\mathscr{O};\R^3)$, then $\left(\mathcal{P}\mathscr{L}_i\right)^2$ is equal to $\mathcal{P}\mathscr{L}_i^2$ on this space too, justifying the consistency between taking the Leray Projector and then converting from Stratonovich to It\^{o} and vice versa as seen for (\ref{projected Ito}). 
\end{remark}

The fact that $\mathcal{P}\mathscr{L}_i^2 \in C\left(W^{2,2}_{\sigma}(\mathscr{O};\R^3); L^2_{\sigma}(\mathscr{O};\R^3)\right)$ again follows from the linearity, (\ref{curly l bounded result}) and Proposition \ref{continuityofP}. We now proceed to justify the assumptions of Appendix IV, \ref{appendix iv}.

\begin{proof}[Assumption \ref{assumption fin dim spaces}:]  We use the system $(a_k)$ of eigenfunctions of the Stokes Oeprator given in Proposition \ref{basis of stokes}, satisfying (\ref{projectionsboundedonH}) and (\ref{mu2}) from Lemmas \ref{projectionsboundedlemma} and \ref{one over n projection lemma} respectively.
\end{proof}

\begin{proof}[Assumption \ref{new assumption 1}:] Items (\ref{111}), (\ref{222}) follow identically to the justification of (\ref{wellposedinX}), (\ref{222*}) for (\ref{projected Ito}) given the increased regularity of $\mathscr{L}$ over $\mathcal{L}$ and the corresponding boundedness of the noise term (\ref{curly l bounded result}). With the linearity of $\mathscr{L}_i$ then (\ref{333}) follows trivially from (\ref{curly l bounded result}).
\end{proof}

In the following $c$ will represent a generic constant changing from line to line, $c(\varepsilon)$ will be a generic constant dependent on a fixed $\varepsilon$, $\phi$ and $\psi$ will be arbitrary elements of $W^{2,2}_{\sigma}(\mathscr{O};\R^3)$ and $\phi_n \in \textnormal{span}\{a_1, \cdots, a_n\}$. 
Towards Assumption \ref{assumptions for uniform bounds2}, we appreciate some further properties of $\mathscr{L}_i$ which were previously shown for $B_i$. Of course the mapping $\mathscr{L}_i$, given by $\phi \mapsto \mathcal{L}_i\phi - \mathcal{L}_{\phi}\xi_i$, is the sum of $\mathcal{L}_i$ and a bounded linear operator on any $W^{k,2}(\mathscr{O};\R^3)$ as was the case for $B_i$ (where the bounded linear operator was $\mathcal{T}_{\xi_i}$). Due to this structure, the results of Corollary \ref{corollary for B_i adjoint}, Proposition \ref{prop for conservation B_i} and Proposition \ref{boundoncommutator} continue to hold for $\mathscr{L}_i$. Deduced from this is the following lemma.

\begin{lemma} \label{preliminary bound for cauchy*}
    For any $\varepsilon > 0$ we have the bound $$\inner{\mathcal{P}_n\mathcal{P}\mathscr{L}_i^2\phi_n}{\phi_n}_1 + \inner{\mathcal{P}_n\mathcal{P}\mathscr{L}_i\phi_n}{\mathcal{P}_n\mathcal{P}\mathscr{L}_i\phi_n}_1 \leq  c(\varepsilon)\norm{\xi_i}_{W^{3,\infty}}^2\norm{\phi_n}_1^2 + \varepsilon \norm{\xi_i}_{W^{3,\infty}}^2\norm{\phi_n}_2^2.$$
\end{lemma}

\begin{proof}
    This now follows precisely as for Lemma \ref{preliminary bound for cauchy}.
\end{proof}

\begin{lemma} \label{kelele}
    For any $\varepsilon > 0$ we have the bound $$\left\vert \inner{\mathcal{P}_n\mathcal{P}\mathscr{L}_{f_n}\phi_n}{\phi_n}_1\right\vert \leq c(\varepsilon)\norm{\phi_n}_1^4 + \varepsilon\norm{\phi_n}_2^2.$$
\end{lemma}

\begin{proof}
    Writing $f_n$ corresponding to $\phi_n$ as in Theorem \ref{Biot theorem}, note that
    $$\left\vert \inner{\mathcal{P}_n\mathcal{P}\mathscr{L}_{f_n}\phi_n}{\phi_n}_1\right\vert = \left\vert \inner{\mathcal{P}_n\mathcal{P}\mathscr{L}_{f_n}\phi_n}{A \phi_n}\right\vert \leq c(\varepsilon)\norm{\mathscr{L}_{f_n}\phi_n}^2 + \varepsilon\norm{\phi_n}_2^2$$
    and 
    $$\norm{\mathscr{L}_{f_n}\phi_n}^2 \leq 2\left(\norm{\mathcal{L}_{f_n}\phi_n}^2 + \norm{\mathcal{L}_{\phi_n}f_n}^2\right)$$
    so we look to control these two terms. Indeed,
    \begin{align*}
        \norm{\mathcal{L}_{f_n}\phi_n}^2 &\leq c\sum_{j=1}^3\sum_{k=1}^3\norm{f_n^j\partial_j\phi_n^k}_{L^2(\mathscr{O};\R)}^2\\
        &\leq c\sum_{j=1}^3\sum_{k=1}^3\norm{f_n^j}_{L^\infty(\mathscr{O};\R)}^2\norm{\partial_j\phi_n^k}_{L^2(\mathscr{O};\R)}^2\\
        &\leq c\sum_{j=1}^3\sum_{k=1}^3\norm{f_n^j}_{W^{2,2}(\mathscr{O};\R)}^2\norm{\partial_j\phi_n^k}_{L^2(\mathscr{O};\R)}^2\\
        &\leq c\norm{f_n}_{W^{2,2}}^2\norm{\phi_n}_{W^{1,2}}^2\\
        &\leq c\norm{\phi_n}_{W^{1,2}}^4\\
        &\leq c\norm{\phi_n}_{1}^4
    \end{align*}
    using the Sobolev Embedding of $W^{2,2}(\mathscr{O};\R) \xhookrightarrow{} L^\infty(\mathscr{O};\R)$ and item (\ref{regularity biot}) of Theorem \ref{Biot theorem}. Likewise observe that 
    \begin{align*}
        \norm{\mathcal{L}_{\phi_n}f_n}^2 &\leq c\sum_{j=1}^3\sum_{k=1}^3\norm{\phi_n^j\partial_jf_n^k}_{L^2(\mathscr{O};\R)}^2\\
        &\leq c\sum_{j=1}^3\sum_{k=1}^3\norm{\phi_n^j}_{L^4(\mathscr{O};\R)}^2\norm{\partial_jf_n^k}_{L^4(\mathscr{O};\R)}^2\\
        &\leq c\sum_{j=1}^3\sum_{k=1}^3\norm{\phi_n^j}_{W^{1,2}(\mathscr{O};\R)}^2\norm{\partial_jf_n^k}_{W^{1,2}(\mathscr{O};\R)}^2\\
        &\leq c\norm{\phi_n}_{W^{1,2}}^2\norm{f_n}_{W^{2,2}}^2\\
        &\leq c\norm{\phi_n}_{W^{1,2}}^4\\
        &\leq c\norm{\phi_n}_{1}^4.
    \end{align*}
    Summing these terms completes the proof. 
\end{proof}

\begin{proof}[Assumption \ref{assumptions for uniform bounds2}:] (\ref{uniformboundsassumpt1}) now follows from Lemmas \ref{preliminary bound for cauchy*} and \ref{kelele} in the same manner as for the velocity equation. (\ref{uniformboundsassumpt2}) is shown exactly as (\ref{probability second}) was for the velocity equation, using that the $\mathcal{P}_n$ are orthogonal projections in $W^{1,2}_{\sigma}(\mathscr{O};\R^3)$.
\end{proof}

In the following $g$ is defined by $$g(x) = \int_{\mathscr{O}}K(x,y)\psi(y)dy $$ as in Theorem \ref{Biot theorem}. The subsequent lemma is in analogy with Lemma \ref{lemmalemma}. 

\begin{lemma} \label{lemmalemmalemma}
    For any $\varepsilon > 0$, we have that
    $$\left\vert \inner{\mathcal{P}\mathscr{L}_f\phi - \mathcal{P}\mathscr{L}_{g}\psi}{\phi - \psi} \right\vert \leq c(\varepsilon)\left[\norm{\phi}_1^2 +\norm{\psi}_1^2\right]\norm{\phi-\psi}^2 + \varepsilon \norm{\phi-\psi}^2_1$$
\end{lemma}
 \begin{proof}
     We write out the left hand side of the above in full, :
     \begin{align*}
         \left\vert \inner{\mathcal{P}\mathscr{L}_f\phi - \mathcal{P}\mathscr{L}_{g}\psi}{\phi - \psi} \right\vert &=  \left\vert \inner{\mathcal{L}_f\phi - \mathcal{L}_{\phi}f  -  \mathcal{L}_{g}\psi + \mathcal{L}_{\psi}g}{\phi - \psi} \right\vert\\
         &=  \left\vert \inner{\mathcal{L}_{f-g}\phi +  \mathcal{L}_{g}(\phi - \psi)  -  \mathcal{L}_{\phi - \psi}f - \mathcal{L}_{\psi}(f-g)}{\phi - \psi} \right\vert
     \end{align*}
     from which we shall split up the terms and control them individually. Firstly,
     \begin{align*}\left\vert \inner{\mathcal{L}_{f-g}\phi}{\phi - \psi} \right\vert &\leq \norm{\mathcal{L}_{f-g}\phi}\norm{\phi - \psi}\\ &\leq c\norm{f-g}_2\norm{\phi}_1\norm{\phi-\psi}\\ &\leq c\norm{\phi-\psi}_1\norm{\phi}_1\norm{\phi-\psi}\\ &\leq c(\varepsilon)\norm{\phi}_1^2\norm{\phi-\psi}^2 + \frac{\varepsilon}{3}\norm{\phi-\psi}_1^2
     \end{align*}
     using (\ref{first begin align}) and that $[f-g](x) = \int_{\mathscr{O}}K(x,y)[\phi-\psi](y)dy$ is the solution specified in Theorem \ref{Biot theorem} for $\phi-\psi$. Even more directly we have that $$\inner{\mathcal{L}_{g}(\phi - \psi)}{\phi - \psi} = 0$$ owing to (\ref{cancellationproperty}), and for the final two terms the bounds
     $$\left\vert \inner{\mathcal{L}_{\phi - \psi}f}{\phi - \psi} \right\vert \leq c\norm{\phi-\psi}_1\norm{f}_2\norm{\phi-\psi} \leq c(\varepsilon)\norm{\phi}_1^2\norm{\phi-\psi}^2 + \frac{\varepsilon}{3}\norm{\phi-\psi}_1^2 $$
     and 
     $$ \left\vert \inner{\mathcal{L}_{\psi}(f-g)}{\phi - \psi} \right\vert \leq c\norm{\psi}_1\norm{f-g}_2\norm{\phi-\psi} \leq c(\varepsilon)\norm{\psi}_1^2\norm{\phi-\psi}^2 + \frac{\varepsilon}{3}\norm{\phi-\psi}_1^2.$$
     Summing these terms concludes the proof.
 \end{proof}

\begin{proof}[Assumption \ref{therealcauchy assumptions}:]
The justification now comes together exactly as in the proof for Assumption \ref{uniqueness for H valued} in the velocity case, noting again that (\ref{combinedterminenergyinequality}) holds for $\mathscr{L}_i$ as well, and using Lemma \ref{lemmalemmalemma}. 
\end{proof}

\begin{proof}[Assumption \ref{assumption for prob in V}:]
    There is very little to demonstrate here, as the linear terms follow from Assumption \ref{therealcauchy assumptions} so we just briefly address the nonlinear term. Through the same process as in Lemma \ref{lemmalemmalemma}, we have that 
    \begin{align*}
         \left\vert \inner{\mathcal{P}\mathscr{L}_f\phi}{\phi} \right\vert &\leq   \left\vert \inner{\mathcal{L}_f\phi}{\phi} + \inner{\mathcal{L}_{\phi}f}{\phi} \right\vert\\
         &\leq c\left[\norm{f}_2\norm{\phi}_1 + \norm{\phi}_1\norm{f}_2 \right]\norm{\phi}\\
         &\leq c\norm{\phi}\norm{\phi}_1^2
     \end{align*}
     where the rest simply follows as in Assumption \ref{therealcauchy assumptions}. 
\end{proof}

\begin{proof}[Assumption \ref{finally the last assumption}:]
    We consider the different operators in turn, starting with the nonlinear term and using that
  \begin{align*}
         \left\vert \inner{\mathcal{P}\mathscr{L}_f\phi - \mathcal{P}\mathscr{L}_{g}\psi}{\eta} \right\vert &=  \left\vert \inner{\mathcal{L}_f\phi - \mathcal{L}_{\phi}f  -  \mathcal{L}_{g}\psi + \mathcal{L}_{\psi}g}{\eta} \right\vert\\
         &=  \left\vert \inner{\mathcal{L}_{f-g}\phi +  \mathcal{L}_{g}(\phi - \psi)  -  \mathcal{L}_{\phi - \psi}f - \mathcal{L}_{\psi}(f-g)}{\eta} \right\vert
     \end{align*}
where exactly as in Lemma \ref{lemmalemmalemma} we have that
$$\norm{\mathcal{L}_{f-g}\phi +  \mathcal{L}_{g}(\phi - \psi)  -  \mathcal{L}_{\phi - \psi}f - \mathcal{L}_{\psi}(f-g) } \leq c\left[\norm{\phi}_1 +\norm{\psi}_1\right]\norm{\phi-\psi}_1$$
so in particular
\begin{equation} \label{kele}
    \left\vert \inner{\mathcal{P}\mathscr{L}_f\phi - \mathcal{P}\mathscr{L}_{g}\psi}{\eta} \right\vert \leq \norm{\eta}\left(c\left[\norm{\phi}_1 +\norm{\psi}_1\right]\norm{\phi-\psi}_1 \right).
\end{equation}
For the Stokes Operator we simply apply Proposition $\ref{prop for moving stokes}$ to see that
\begin{equation} \label{kelele*}
\abs{\inner{A\phi -A\psi}{\eta}} = \abs{\inner{A(\phi -\psi)}{\eta}} = \inner{\phi - \psi}{\eta}_1 \leq \norm{\phi-\psi}_1\norm{\eta}_1
\end{equation}
and for the $\mathscr{L}_i^2$ term we use Corollary \ref{corollary for B_i adjoint} to observe that
\begin{equation} \label{kelelele}
    \abs{\inner{\mathscr{L}_i^2\phi -\mathscr{L}_i^2\psi}{\eta}} = \abs{\inner{\mathscr{L}_i^2(\phi -\psi)}{\eta}} = \abs{\inner{\mathscr{L}_i(\phi -\psi)}{\mathscr{L}_i^*\eta}} \leq c\norm{\xi_i}^2_{W^{1,\infty}}\norm{\phi-\psi}_1\norm{\eta}_1.
\end{equation}
Combining (\ref{kele}),(\ref{kelele*}) and (\ref{kelelele}) gives the result.

\end{proof}

\section{Appendices} \label{section appendices}
\subsection{Appendix I: Proofs from Subsection \ref{functional framework subsection} } \label{Appendix I}

Before proving the promised results, we state the Sobolev Embeddings which are vital for our analysis. These can be found in [\cite{adams}] Section 4.

\begin{theorem} \label{SobolevEmbeds}
    Letting $k \in \{0\} \cup \N$ be arbitrary and $'\xhookrightarrow{}'$ denote 'embeds into', we have the following results:
    
    \begin{itemize}
        \item For $mp > N$, or $m=N$ and $p =1$, \begin{equation} \label{embedintocontinuous} W^{m + k, p}(\mathcal{O};\R) \xhookrightarrow{} C^k_b(\mathcal{O};\R)\end{equation} and further if $p \leq r \leq \infty$, then \begin{equation} \label{embed} W^{m + k,p}(\mathcal{O};\R) \xhookrightarrow{} W^{k,r}(\mathcal{O};\R)\end{equation} and in particular \begin{equation} \label{embed2} W^{m,p}(\mathcal{O};\R) \xhookrightarrow{} L^r(\mathcal{O};\R).\end{equation}
        
        \item For $mp = N$, if $p \leq r < \infty$ then (\ref{embed}) and (\ref{embed2}) again hold.
        
        \item For $mp < N$, if $p \leq r \leq Np/(N -mp)$ then (\ref{embed}) and (\ref{embed2}) again hold.
        
    \end{itemize}
    
The embedding is a continuous linear operator, where the boundedness constant is dependent only on the choices of $m,p,N,r,k$.
\end{theorem}

\begin{corollary}
    We have the embeddings
    \begin{align*}
        W^{2,2}(\mathcal{O};\R) &\xhookrightarrow{} L^{\infty}(\mathcal{O};\R)\\
        W^{1,2}(\mathcal{O};\R) &\xhookrightarrow{} L^6(\mathcal{O};\R).
    \end{align*}
\end{corollary}

\begin{proof}
Immediate from (\ref{embed2}), from the first and third bullet points respectively.
\end{proof}

We also state the following classical result.

\begin{theorem}[Gagliardo-Nirenberg Inequality] \label{gagliardonirenberginequality}
    Let $p,q,\alpha \in \R$, $m \in \N$ be such that $p > q \geq 1$, $m > N(\frac{1}{2} - \frac{1}{p})$ and $\frac{1}{p} = \frac{\alpha}{q} + (1-\alpha)(\frac{1}{2} - \frac{m}{N})$. Then there exists a constant $c$ (dependent on the given parameters) such that for any $f \in L^p(\T;\R) \cap W^{m,2}(\T;\R)$, we have \begin{equation}\label{gag torus}\norm{f}_{L^p(\T;\R)} \leq c\norm{f}^{\alpha}_{L^q(\T;\R)}\norm{f}^{1-\alpha}_{W^{m,2}(\T;\R)}.\end{equation} In the case $f \in L^p(\mathscr{O};\R) \cap W^{m,2}(\mathscr{O};\R)$ we have instead \begin{equation}\label{gag bounded domain}\norm{f}_{L^p(\mathscr{O};\R)} \leq c\left(\norm{f}^{\alpha}_{L^q(\mathscr{O};\R)}\norm{f}^{1-\alpha}_{W^{m,2}(\mathscr{O};\R)} + \norm{f}_{L^2(\mathscr{O};\R)} \right).\end{equation}
\end{theorem}

\begin{proof}
See [\cite{Gagliardo}] pp.125-126. 
\end{proof}

\begin{proof}[Proof of \ref{continuityofnonlinear}:]
    We shall directly jump to the continuity of the mapping before showing it is in fact well defined, as the latter argument will be contained in the former. We proceed by sequential continuity, so we take an arbitrary sequence $(f_n)$ convergent to some $f$ in $W^{m+1,2}(\mathcal{O};\R^N)$ and deduce the convergence of $(\mathcal{L}_{f_n}f_n)$ to $\mathcal{L}_ff$. We have for any $n$
    \begin{align}
        \nonumber \norm{\mathcal{L}_{f_n}f_n - \mathcal{L}_ff}_{W^{m,2}(\mathcal{O};\R^N)} &= \norm{\mathcal{L}_{f_n}f_n - \mathcal{L}_{f_n}f + \mathcal{L}_{f_n}f - \mathcal{L}_{f}f}_{W^{m,2}(\mathcal{O};\R^N)}\\ \nonumber
        &\leq \norm{\mathcal{L}_{f_n}f_n - \mathcal{L}_{f_n}f}_{W^{m,2}(\mathcal{O};\R^N)} + \norm{\mathcal{L}_{f_n}f - \mathcal{L}_{f}f}_{W^{m,2}(\mathcal{O};\R^N)}\\ \label{labelalign2}
        &= \bigg \vert \bigg \vert \sum_{j=1}^Nf_n^j (\partial_jf_n-\partial_jf) \bigg \vert \bigg \vert_{W^{m,2}(\mathcal{O};\R^N)} + \bigg \vert \bigg \vert \sum_{j=1}^N (f_n^j-f^j)\partial_j f \bigg \vert \bigg \vert_{W^{m,2}(\mathcal{O};\R^N)}
    \end{align}
The norms are defined by the $W^{m,2}(\mathcal{O};\R)$ norms of the component mappings, so we consider the $i^\textnormal{th}$ component. Looking at the first term, we have \begin{align*} \label{yoyoyo} \bigg \vert \bigg \vert \sum_{j=1}^Nf_n^j (\partial_jf_n^i-\partial_jf^i) \bigg \vert \bigg \vert_{W^{m,2}(\mathcal{O};\R)}^2 &= \sum_{0 \leq \abs{\alpha} \leq m} \bigg \vert \bigg \vert \sum_{j=1}^N D^\alpha\big( f_n^j(\partial_jf_n^i-\partial_jf^i)\big)\bigg \vert \bigg \vert_{L^2(\mathcal{O};\R)}^2\end{align*}
which we further split up by considering each $\alpha$ in the sum, introducing the notation for a multi-index $\alpha'$ whereby $\alpha' \leq \alpha$ we mean $\alpha'_l \leq \alpha_l$ all $1 \leq l \leq N$ ($\alpha'_l \in \N \cup \{0\})$. Further by $\alpha - \alpha'$ we mean the multi-index with components $(\alpha - \alpha')_l = \alpha_l - \alpha'_l$. Then by the Leibniz Rule, for each $j$,
\begin{equation} \label{byleibniz}
    D^\alpha\big( f_n^j(\partial_jf_n^i-\partial_jf^i)\big) = \sum_{\alpha' \leq \alpha}D^{\alpha - \alpha'}f_n^j\big(D^{\alpha'}\partial_jf_n^i-D^{\alpha'}\partial_jf^i\big)
\end{equation}
so
\begin{align*} \norm{(\ref{byleibniz})}^2_{L^2(\mathcal{O};\R)} &\leq 2^{\abs{\alpha}}\sum_{\alpha' \leq \alpha}\norm{D^{\alpha - \alpha'}f_n^j\big(D^{\alpha'}\partial_jf_n^i-D^{\alpha'}\partial_jf^i\big)}_{L^2(\mathcal{O};\R)}^2\\
&\leq 2^m\sum_{\alpha' \leq \alpha}\norm{D^{\alpha - \alpha'}f_n^j\big(D^{\alpha'}\partial_jf_n^i-D^{\alpha'}\partial_jf^i\big)}_{L^2(\mathcal{O};\R)}^2\\
&= 2^m\bigg(\sum_{\alpha' < \alpha}\norm{D^{\alpha - \alpha'}f_n^j\big(D^{\alpha'}\partial_jf_n^i-D^{\alpha'}\partial_jf^i\big)}_{L^2(\mathcal{O};\R)}^2 + \norm{f_n^j\big(D^\alpha \partial_jf_n^i - D^\alpha \partial_jf^i\big)}_{L^2(\mathcal{O};\R)}^2\bigg)
\\ &\leq 2^m\bigg(\sum_{\alpha' < \alpha}\norm{ D^{\alpha-\alpha'}f_n^j}^2_{L^4(\mathcal{O};\R)}\norm{D^{\alpha'}\partial_jf_n^i-D^{\alpha'}\partial_jf^i}^2_{L^4(\mathcal{O};\R)} \\& \qquad \qquad \qquad \qquad \qquad \qquad  + \norm{f_n^j}^2_{L^\infty(\mathcal{O};\R)}\norm{(D^{\alpha}\partial_jf_n^i-D^{\alpha}\partial_jf^i)}^2_{L^2(\mathcal{O};\R)}\bigg)
\end{align*} 
having applied H\"{o}lder type inequalities and invoking two different Sobolev embeddings; the first is the embedding of $W^{1,2}(\mathcal{O};\R)$ into $L^4(\mathcal{O};\R)$, noting of course that the operators $D^{\alpha - \alpha'}$ and $D^{\alpha'}\partial_j$ are of order no greater than $m$ so the terms are indeed in $W^{1,2}(\mathcal{O};\R)$. The second is the embedding of $W^{2,2}(\mathcal{O};\R)$ into $L^{\infty}(\mathcal{O};\R)$ (and hence the embedding of $W^{m+1,2}(\mathcal{O};\R)$ into $L^{\infty}(\mathcal{O};\R)$). Returning now to (\ref{labelalign2}),
\begin{align*}
    \bigg \vert \bigg \vert \sum_{j=1}^Nf_n^j (\partial_jf_n-\partial_jf) \bigg \vert \bigg \vert_{W^{m,2}(\mathcal{O};\R^N)}^2 &= \sum_{i=1}^N\sum_{ \abs{\alpha} \leq m} \bigg \vert \bigg \vert \sum_{j=1}^N D^\alpha\big( f_n^j(\partial_jf_n^i-\partial_jf^i)\big)\bigg \vert \bigg \vert_{L^2(\mathcal{O};\R)}^2\\
    &\leq N \sum_{i=1}^N\sum_{\abs{\alpha} \leq m}\sum_{j=1}^N\bigg \vert \bigg \vert D^\alpha\big( f_n^j(\partial_jf_n^i-\partial_jf^i)\big)\bigg \vert \bigg \vert_{L^2(\mathcal{O};\R)}^2
\end{align*}
at which point we use the bound derived on (\ref{byleibniz}) to deduce further that this is 
\begin{align*}&\leq 2^mN\sum_{i=1}^N\sum_{ \abs{\alpha} \leq m}\sum_{j=1}^N\bigg(\sum_{\alpha' < \alpha}\norm{ D^{\alpha-\alpha'}f_n^j}^2_{L^4(\mathcal{O};\R)}\norm{D^{\alpha'}\partial_jf_n^i-D^{\alpha'}\partial_jf^i}^2_{L^4(\mathcal{O};\R)}\\ & \qquad \qquad \qquad \qquad \qquad \qquad \qquad \qquad \qquad \qquad  + \norm{f_n^j}^2_{L^\infty(\mathcal{O};\R)}\norm{(D^{\alpha}\partial_jf_n^i-D^{\alpha}\partial_jf^i)}^2_{L^2(\mathcal{O};\R)}\bigg).\end{align*}
Now the quoted embeddings are continuous, so we can deduce convergence of $(D^{\alpha'} \partial_jf_n^i)_n$ to $D^{\alpha'}\partial_jf^i$ in $L^4(\mathcal{O};\R)$ from the convergence in $W^{1,2}(\mathcal{O};\R)$, which is itself inherited from the assumed convergence of $(f_n)$ to $f$ in $W^{m+1,2}(\mathcal{O};\R^N)$. A similar argument affords us control of $\norm{ D^{\alpha-\alpha'}f_n^j}^2_{L^4(\mathcal{O};\R)}$, as the sequence $(D^{\alpha- \alpha'}f_n^j)_n$ is convergent in $W^{1,2}(\mathcal{O};\R)$ and hence $L^4(\mathcal{O};\R)$. Thus we can bound this term uniformly, say by $\norm{D^{\alpha-\alpha'}f^j}^2_{L^4(\mathcal{O};\R)} + 1$ for sufficiently large $n$. We implore the same idea to bound $\norm{f^j_n}^2_{L^{\infty}(\mathcal{O};\R)}$, noting also that $(D^{\alpha}\partial_jf_n^i)_n$ converges to $D^{\alpha}\partial_jf^i$ in $L^2(\mathcal{O};\R)$ from the assumed $W^{m+1,2}(\mathcal{O};\R^N)$ convergence. Thus all terms in the finite sum above converge to $0$ as $n \rightarrow \infty$, which deals with the first term in (\ref{labelalign2}). In fact the second term is treated near identically so we conclude the proof here, noting of course that these H\"{o}lder type inequalities are what justifies the mapping to be well defined.

\end{proof}

\begin{proof}[Proof of \ref{basic nonlinear bound}:]
We use the same ideas as seen in the proof of Lemma \ref{continuityofnonlinear}. For (\ref{first begin align})  starting with the term $\mathcal{L}_fg$:
\begin{align*}
    \norm{\mathcal{L}_fg} &\leq \sum_{j=1}^N\norm{f^j\partial_jg}\\
    &\leq \sum_{j=1}^N\sum_{l=1}^N\norm{f^j\partial_jg^l}_{L^2(\mathcal{O};\R)}\\
    &\leq \sum_{j=1}^N\sum_{l=1}^N\norm{f^j}_{L^{\infty}(\mathcal{O};\R)}\norm{\partial_jg^l}_{L^2(\mathcal{O};\R)}\\
    &\leq c\sum_{j=1}^N\sum_{l=1}^N\norm{f^j}_{W^{2,2}(\mathcal{O};\R)}\norm{\partial_jg^l}_{L^2(\mathcal{O};\R)}\\
    &\leq c\norm{f}_{W^{2,2}}\norm{g}_{W^{1,2}}
\end{align*}
\end{proof}
using that $\left(\sum_{l=1}^N\abs{a_l}^2\right)^{1/2} \leq \sum_{l=1}^N\abs{a_l}$ and the Sobolev Embedding $W^{2,2}(\mathcal{O};\R) \xhookrightarrow{} L^{\infty}(\mathcal{O};\R)$. Similarly comes the bound
\begin{align*}
    \norm{\mathcal{L}_gf} & \leq \sum_{j=1}^N\sum_{l=1}^N\norm{g^j\partial_jf^l}_{L^2(\mathcal{O};\R)}\\
    & \leq \sum_{j=1}^N\sum_{l=1}^N\norm{g^j}_{L^{4}(\mathcal{O};\R)}\norm{\partial_jf^l}_{L^4(\mathcal{O};\R)}\\
    & \leq c\sum_{j=1}^N\sum_{l=1}^N\norm{g^j}_{W^{1,2}(\mathcal{O};\R)}\norm{\partial_jf^l}_{W^{1,2}(\mathcal{O};\R)}\\
    & \leq c\sum_{j=1}^N\sum_{l=1}^N\norm{g^j}_{W^{1,2}(\mathcal{O};\R)}\norm{f^l}_{W^{2,2}(\mathcal{O};\R)}\\
    &\leq c\norm{g}_{W^{1,2}}\norm{f}_{W^{2,2}}
\end{align*}
applying the embedding of $W^{1,2}(\mathcal{O};\R) \xhookrightarrow{} L^4(\mathcal{O};\R)$. This justifies (\ref{first begin align}). As for (\ref{second begin align}),
$$  \norm{\mathcal{L}_gf}_{W^{1,2}}^2 =  \norm{\mathcal{L}_gf}^2 +  \sum_{k=1}^N\norm{\partial_k\mathcal{L}_{g}f}^2$$
where $$\norm{\mathcal{L}_gf}^2 \leq c\norm{g}_{W^{1,2}}^2\norm{f}_{W^{2,2}}^2 \leq c\norm{g}_{W^{1,2}}^2\norm{f}_{W^{3,2}}^2$$ having applied (\ref{first begin align}), and
\begin{align*}
   \sum_{k=1}^N\norm{\partial_k\mathcal{L}_{g}f}^2
        &\leq c\sum_{k=1}^N\sum_{j=1}^N\left(\norm{\partial_kg^j\partial_jf}^2 + \norm{g^j\partial_k\partial_jf}^2\right)\\
        &= c\sum_{k=1}^N\sum_{j=1}^N\sum_{l=1}^N\left(\norm{\partial_kg^j\partial_jf^l}_{L^2(\mathcal{O};\R)}^2 + \norm{g^j\partial_k\partial_jf^l}_{L^2(\mathcal{O};\R)}^2\right)\\
        &\leq c\sum_{k=1}^N\sum_{j=1}^N\sum_{l=1}^N\left(\norm{\partial_kg^j}_{L^{2}(\mathcal{O};\R)}^2\norm{\partial_jf^l}_{L^{\infty}(\mathcal{O};\R)}^2 + \norm{g^j}_{L^4(\mathcal{O};\R)}^2\norm{\partial_k\partial_jf^l}_{L^4(\mathcal{O};\R)}^2\right)\\
        &\leq c\sum_{k=1}^N\sum_{j=1}^N\sum_{l=1}^N\left(\norm{\partial_kg^j}_{L^2(\mathcal{O};\R)}^2\norm{\partial_jf^l}_{W^{2,2}(\mathcal{O};\R)}^2 + \norm{g^j}^2_{W^{1,2}(\mathcal{O};\R)}\norm{\partial_k\partial_jf^l}_{W^{1,2}(\mathcal{O};\R)}^2\right)\\
        &\leq c\norm{g}_{W^{1,2}}^2\norm{f}_{W^{3,2}}^2.
\end{align*}
Combining these justifies (\ref{second begin align}). Indeed (\ref{third begin align}) is attained by the same procedure, just instead passing to the bounds $$ \norm{\partial_kg^j\partial_jf^l}_{L^2(\mathcal{O};\R)}^2 + \norm{g^j\partial_k\partial_jf^l}_{L^2(\mathcal{O};\R)}^2
        \leq \norm{\partial_kg^j}_{L^{4}(\mathcal{O};\R)}^2\norm{\partial_jf^l}_{L^{4}(\mathcal{O};\R)}^2 + \norm{g^j}_{L^{\infty}(\mathcal{O};\R)}^2\norm{\partial_k\partial_jf^l}_{L^2(\mathcal{O};\R)}^2$$
and proceeding in the same manner.

\begin{proof}[Proof of \ref{biggglemma}:]
We prove this result in the case of the bounded domain $\mathscr{O}$, noting that the Torus follows in the same way just without the detour to compactly supported functions. The first task is to show that these inner products are well defined, which immediately prompts clarification as to how the $L^2(\mathscr{O};\R^N)$ inner product is being used here. Formally this should be understood as the duality bracket between $L^{6/5}(\mathscr{O};\R^N)$ and $L^6(\mathscr{O};\R^N)$, noting that the reciprocal of these components sum to $1$ from which the duality follows. Of course, this duality is prescribed by the $L^2(\mathscr{O};\R^N)$ inner product. Without loss of generality we treat the left hand side of (\ref{wloglhs}) for arbitrarily chosen $\phi, f, g$, and consider the $l^{\textnormal{th}}$ component. We have that 
\begin{align*}
    \norm{(\mathcal{L}_{\phi}f)^l}_{L^{6/5}(\mathscr{O};\R)} &\leq \sum_{j=1}^N \norm{\phi^j\partial_jf^l}_{L^{6/5}(\mathscr{O};\R)}\\
    &\leq \sum_{j=1}^N \norm{\phi^j}_{L^3(\mathscr{O};\R)}\norm{\partial_jf^l}_{L^2(\mathscr{O};\R)}
\end{align*}
having applied a general H\"{o}lder Inequality, noting that $$\frac{5}{6} = \frac{1}{3} + \frac{1}{2}.$$ Indeed this expression is finite from the embedding $W^{1,2}(\mathscr{O};\R) \xhookrightarrow{} L^3(\mathscr{O};\R)$. Subsequently $$\norm{\mathcal{L}_{\phi}f}_{L^{6/5}(\mathscr{O};\R^N)} = \bigg(\sum_{l=1}^N\norm{(\mathcal{L}_{\phi}f)^l}_{L^{6/5}(\mathscr{O};\R)}^{6/5}\bigg)^{5/6} < \infty$$ as required. Showing that $g \in L^6(\mathscr{O};\R^N)$ is simply immediate from the stronger embedding $W^{1,2}(\mathscr{O};\R) \xhookrightarrow{} L^6(\mathscr{O};\R)$. To show the desired equality (\ref{wloglhs}) consider a sequence $(\phi_n)$ in $C^\infty_{0,\sigma}(\mathscr{O};\R^N)$ convergent to $\phi$ in $W^{1,2}(\mathscr{O};\R^N)$, which certainly exists from Lemma \ref{W12sigmacharacter}. Then 
    \begin{align*}
       \inner{\mathcal{L}_{\phi_n}f}{g} &= \Big\langle\sum_{j=1}^N\phi^j_n\partial_jf,g\Big\rangle\\
       &= \sum_{l=1}^N\sum_{j=1}^N\inner{\phi^j_n\partial_jf^l}{g^l}_{L^2(\mathscr{O};\R)}\\
       &= -\sum_{l=1}^N\sum_{j=1}^N\Big(\inner{f^l}{\partial_j\phi^j_ng^l}_{L^2(\mathscr{O};\R)} + \inner{f^l}{\phi^j_n\partial_jg^l}_{L^2(\mathscr{O};\R)}\Big)
    \end{align*}
having applied integration by parts and calling upon the compact support of $\phi_n$. Now note that upon summation over $j$, the first inner product is nullified thanks to the divergence free property of $\phi_n$. We can now collapse the second inner product back: $$-\sum_{l=1}^N\sum_{j=1}^N\inner{f^l}{\phi^j_n\partial_jg^l}_{L^2(\mathscr{O};\R)} = -\sum_{j=1}^N\inner{f}{\phi^j_n\partial_jg} = -\inner{f}{\mathcal{L}_{\phi_n}g}.$$ It only remains to show that the equality holds in the limit, so let's again treat the LHS of (\ref{wloglhs}) for the $l^{\textnormal{th}}$ component. Using the H\"{o}lder Inequality which defines the duality bracket,
\begin{align*}
    \abs{\inner{(\mathcal{L}_{\phi}f -\mathcal{L}_{\phi_n}f)^l}{g^l}_{L^2(\mathscr{O};\R)}} &\leq \norm{(\mathcal{L}_{\phi}f -\mathcal{L}_{\phi_n}f)^l}_{L^{6/5}(\mathscr{O};\R)}\norm{g^l}_{L^6(\mathscr{O};\R)}\\ &\leq \bigg(\sum_{j=1}^N \norm{(\phi-\phi_n)^j}_{L^3(\mathscr{O};\R)}\norm{\partial_jf^l}_{L^2(\mathscr{O};\R)}\bigg)\norm{g^l}_{L^6(\mathscr{O};\R)}
\end{align*}
demonstrating the required convergence in the limit from the continuity of the embedding $W^{1,2}(\mathscr{O};\R)$ into $L^3(\mathscr{O};\R)$. Thus we have the convergence for each component in the $L^2(\mathscr{O};\R^N)$ inner product, from which we conclude the result.

\end{proof}

\begin{proof}[Proof of \ref{inequalityforcauchynonlinear}:]
We prove this in the more involved case of the bounded domain $\mathscr{O}$, calling upon Theorem \ref{gagliardonirenberginequality} and using (\ref{gag bounded domain}) in this case instead of the simpler (\ref{gag torus}). For such $f$ and $g$ observe that
\begin{align*}
    \norm{\mathcal{L}_fg} &= \left\Vert \sum_{j=1}^N f^j\partial_jg \right\Vert\\
    &\leq\sum_{j=1}^N  \left\Vert  f^j\partial_jg \right\Vert\\
    &\leq\sum_{j=1}^N \sum_{l=1}^N \left\Vert  f^j\partial_jg^l \right\Vert_{L^2(\mathscr{O};\R)}\\
    &\leq\sum_{j=1}^N \sum_{l=1}^N \norm{ f^j}_{L^6(\mathscr{O};\R)}\norm{\partial_jg^l}_{L^3(\mathscr{O};\R)}
\end{align*}
having just used a H\'{o}lder Inequality. We now apply Theorem \ref{gagliardonirenberginequality} for the values $p=3,q=2,m=1$ and $\alpha = \frac{1}{2}$ to bound this further by $$\sum_{j=1}^N \sum_{l=1}^N c\norm{ f^j}_{L^6(\mathscr{O};\R)}\left(\norm{\partial_jg^l}^{1/2}_{L^2(\mathscr{O};\R)}\norm{\partial_jg^l}^{1/2}_{W^{1,2}(\mathscr{O};\R)} + \norm{\partial_jg^l}_{L^2(\mathscr{O};\R)}\right).$$ We now use the Sobolev Embedding $W^{1,2}(\mathscr{O};\R) \xhookrightarrow{} L^6(\mathscr{O};\R)$ from the third bullet point of Theorem \ref{SobolevEmbeds} and some coarse bounds to control this further by
\begin{align*}
    &\sum_{j=1}^N \sum_{l=1}^N c\norm{ f^j}_{W^{1,2}(\mathscr{O};\R)}\left(\norm{\partial_jg^l}^{1/2}_{L^2(\mathscr{O};\R)}\norm{\partial_jg^l}^{1/2}_{W^{1,2}(\mathscr{O};\R)} + \norm{\partial_jg^l}_{L^2(\mathscr{O};\R)}\right)\\ &\leq \sum_{j=1}^N \sum_{l=1}^N c\norm{ f}_{W^{1,2}}\left(\norm{g}^{1/2}_{W^{1,2}}\norm{g}^{1/2}_{W^{2,2}} + \norm{g}_{W^{1,2}}\right)\\
    &\leq c\norm{ f}_{W^{1,2}}\left(\norm{g}^{1/2}_{W^{1,2}}\norm{g}^{1/2}_{W^{2,2}} + \norm{g}_{W^{1,2}}\right)\\
    &\leq c\norm{f}_{1}\left(\norm{g}_{1}^{1/2}\norm{g}_{2}^{1/2} + \norm{g}_1\right)
\end{align*}
using the norm equivalences from Proposition \ref{A2}.

\end{proof}

\newpage
\subsection{Appendix II: Proofs from Subsection \ref{subsection the salt operator}} \label{appendix ii}

In the following proofs we will track the constant $c$ stated in the results, though we make no attempt to optimise the tracked constant that we give. $\eta_k$ will represent the number of distinct multi-indices $\alpha$ such that $\abs{\alpha} \leq k$.

\begin{proof}[Proof of \ref{biggestcorcor}:]
    We fix any such $k$, and consider differential operators $D^\alpha$ for $\abs{\alpha}\leq k$. From the definition of the operator $\mathcal{T}_{\xi_i}$ and the Leibniz Rule, it is clear that
    $$D^\alpha \mathcal{T}_{\xi_i}f = \sum_{\alpha' \leq \alpha}\mathcal{T}_{D^{\alpha - \alpha'}\xi_i}D^{\alpha'}f.$$
    Now
    \begin{align*}
    \norm{D^\alpha \mathcal{T}_{\xi_i}f}^2 &\leq 2^{\abs{\alpha}}\sum_{\alpha' \leq \alpha}\norm{\mathcal{T}_{D^{\alpha - \alpha'} \xi_i}D^{\alpha'}f}^2\\
    &\leq 2^{k}\sum_{\alpha' \leq \alpha}\norm{\mathcal{T}_{D^{\alpha - \alpha'} \xi_i}D^{\alpha'}f}^2\\
    &= 2^{k}\sum_{\alpha' \leq \alpha}\Big \vert \Big \vert \sum_{j=1}^N D^{\alpha'}f^j\big(\nabla D^{\alpha - \alpha'} \xi_i^j\big) \Big \vert \Big \vert^2\\
    &\leq 2^{k}\sum_{\alpha' \leq \alpha}\norm{D^{\alpha - \alpha'}\xi_i}^2_{W^{1,\infty}}\Big \vert \Big \vert \sum_{j=1}^ND^{\alpha'}f^j \Big \vert \Big \vert^2_{L^2(\mathcal{O};\R)}\\
    &\leq 2^{k}N\norm{\xi_i}^2_{W^{k+1,\infty}}\sum_{\alpha' \leq \alpha}\sum_{j=1}^N\norm{D^{\alpha'}f^j}_{L^2(\mathcal{O};\R)}^2\\
    &= 2^{k}N\norm{\xi_i}^2_{W^{k+1,\infty}}\sum_{\alpha' \leq \alpha}\norm{D^{\alpha'}f}^2\\
    &\leq 2^{k}N\norm{\xi_i}^2_{W^{k+1,\infty}}\norm{f}^2_{W^{k,2}}
\end{align*}
    so ultimately
    \begin{align*}
        \norm{\mathcal{T}_{\xi_i}f}_{W^{k,2}}^2 &= \sum_{\abs{\alpha} \leq k}\norm{D^\alpha \mathcal{T}_{\xi_i}f}^2\\
        &\leq 2^{k}N\eta_k \norm{\xi_i}^2_{W^{k+1,\infty}}\norm{f}^2_{W^{k,2}}.
    \end{align*}
    
\end{proof}

\begin{proof}[Proof of \ref{boundsonL_xi}, (\ref{L_ibound}):]
        We fix any such $k$, and similarly to \ref{biggestcorcor} it is clear that  $$D^\alpha \mathcal{L}_{\xi_i}f = \sum_{\alpha' \leq \alpha}\mathcal{L}_{D^{\alpha - \alpha'}\xi_i}D^{\alpha'}f$$ again for any $\abs{\alpha} \leq k$. Looking at each term in this sum,
    \begin{align*}
    \norm{\mathcal{L}_{D^{\alpha - \alpha'}\xi_i}D^{\alpha'}f}^2 &= \Big \vert \Big\vert\sum_{j=1}^ND^{\alpha - \alpha'}\xi_i^j \partial_j D^{\alpha'}f\Big\vert\Big\vert^2\\
    &\leq N\sum_{j=1}^N\sum_{l=1}^N\norm{D^{\alpha - \alpha'}\xi_i^j \partial_jD^{\alpha'}f^l}^2_{L^2(\mathcal{O};\R)}\\
    &\leq N\norm{D^{\alpha - \alpha'}\xi_i}^2_{L^{\infty}}\sum_{j=1}^N\sum_{l=1}^N\norm{ \partial_jD^{\alpha'}f^l}^2_{L^2(\mathcal{O};\R)}\\
    &\leq N\norm{\xi_i}^2_{W^{k,\infty}}\sum_{j=1}^N\norm{ \partial_jD^{\alpha'}f}^2\\
    &\leq N\norm{\xi_i}^2_{W^{k,\infty}}\norm{f}^2_{W^{k+1,2}}
\end{align*}
    moreover
    \begin{align*}
        \norm{\mathcal{L}_{\xi_i}f}_{W^{k,2}}^2 &= \sum_{\abs{\alpha}\leq k}\norm{D^\alpha \mathcal{L}_{\xi_i}f}^2\\
        &= \sum_{\abs{\alpha}\leq k} \bigg\vert\bigg\vert \sum_{\alpha' \leq \alpha}\mathcal{L}_{D^{\alpha - \alpha'}\xi_i}D^{\alpha'}f\bigg\vert\bigg\vert^2\\
        &\leq \sum_{\abs{\alpha}\leq k}\sum_{\alpha' \leq \alpha}\norm{\mathcal{L}_{D^{\alpha - \alpha'}\xi_i}D^{\alpha'}f}^2\\
        &\leq 2^kN\eta_k\norm{\xi_i}^2_{W^{k,\infty}}\norm{f}^2_{W^{k+1,2}}.
    \end{align*}

\end{proof}
    
    In the following proofs we generalise the definition of $B$ to a mapping $$B_fg:=\mathcal{L}_{f}g + \mathcal{T}_fg.$$ In particular, $B_{i}:=B_{\xi_i}$.

\begin{proof}[Proof of \ref{prop for conservation B_i}, (\ref{combinedterminenergyinequality}):]
We consider terms \begin{equation} \label{commandv}\inner{B_i^2f}{f}_{W^{k,2}} +  \norm{B_if}_{W^{k,2}}^2\end{equation}
and each derivative in the sum for the inner product: that is, we are looking at \begin{equation}\label{thatiswearelookingat}\inner{D^\alpha B_i^2f}{D^\alpha f} +  \inner{D^\alpha B_if}{D^\alpha B_if}\end{equation} where the only starting place here is to simplify these derivatives. Combining the arguments in \ref{biggestcorcor} and \ref{boundsonL_xi} then evidently \begin{align}\nonumber D^\alpha B_{\xi_i}f &= \label{thenevidently}\sum_{\alpha' \leq \alpha}B_{D^{\alpha-\alpha'} \xi_i}D^{\alpha'}f\\
&= \sum_{\alpha' < \alpha}B_{D^{\alpha-\alpha'} \xi_i}D^{\alpha'}f + B_{\xi_i}D^{\alpha}f.\end{align}
Plugging this result in, we also see that
\begin{align*}D^\alpha B^2_{\xi_i}f &= D^\alpha B_{\xi_i}\big(B_{\xi_i}f\big)\\
&= \sum_{\alpha' < \alpha}B_{D^{\alpha-\alpha'} \xi_i}D^{\alpha'}B_{\xi_i}f + B_{ \xi_i}D^{\alpha}B_{\xi_i}f
\end{align*}
which will use in our analysis of (\ref{thatiswearelookingat}), reducing the expression to $$\Big\langle \sum_{\alpha' < \alpha}B_{D^{\alpha-\alpha'} \xi_i}D^{\alpha'}B_{\xi_i}f + B_{\xi_i}D^\alpha B_{\xi_i}f,D^\alpha f\Big\rangle +  \inner{D^\alpha B_{\xi_i}f}{D^\alpha B_{\xi_i}f}$$ which we further break up in terms of the $L^2(\mathcal{O};\R^N)$ adjoint $B_{\xi_i}^*$: \begin{equation}\label{equitino}\Big\langle \sum_{\alpha' < \alpha} B_{D^{\alpha-\alpha'} \xi_i}D^{\alpha'}B_{\xi_i}f,D^\alpha f\Big\rangle + \inner{D^\alpha B_{\xi_i}f}{B^*_{\xi_i} D^\alpha f} +  \inner{D^\alpha B_{\xi_i}f}{D^\alpha B_{\xi_i}f}\end{equation} requiring that $D^\alpha f \in W^{1,2}$, which is satisfied by the assumption $f \in W^{k+2,2}$. By combining the second and third inner products and using (\ref{thenevidently}), this becomes $$\Big\langle \sum_{\alpha' < \alpha} B_{D^{\alpha-\alpha'} \xi_i}D^{\alpha'}B_{\xi_i}f,D^\alpha f\Big\rangle + \Big\langle D^\alpha B_{\xi_i}f,B^*_{\xi_i} D^\alpha f + \sum_{\alpha' < \alpha}B_{D^{\alpha-\alpha'} \xi_i}D^{\alpha'}f + B_{\xi_i}D^\alpha f\Big\rangle$$
which we look to simplify by combining $B_{\xi_i}^*$ and $B_{\xi_i}$, noting that $$B_i^*+B_i= \mathcal{L}_{\xi_i}^* + \mathcal{T}_i^* + \mathcal{L}_{\xi_i} + \mathcal{T}_i = \mathcal{T}_i^* + \mathcal{T}_i.$$ Indeed this arrives us at the expression
$$\Big\langle \sum_{\alpha' < \alpha}B_{D^{\alpha-\alpha'} \xi_i}D^{\alpha'}B_{\xi_i}f,D^\alpha f\Big\rangle + \Big\langle D^\alpha B_{\xi_i}f,\big(\mathcal{T}_{\xi_i} + \mathcal{T}_{\xi_i}^*\big) D^\alpha f + \sum_{\alpha' < \alpha}B_{D^{\alpha-\alpha'} \xi_i}D^{\alpha'}f\Big\rangle.$$ As we are looking to achieve control with respect to the $W^{k,2}(\mathcal{O};\R^N)$ norm of $f$, then it is the terms with differential operators of order greater than $k$ that concern us. Of course this was the motivating factor behind combining $B_{\xi_i}$ and its adjoint, nullifying the additional derivative coming from $\mathcal{L}_{\xi_i}$. There are more higher order terms to go though, and the strategy will be to write these in terms of commutators with a differential operator of controllable order. This will involve considering different aspects of our sum in tandem, which will be helped by calling (\ref{thenevidently}) into action once more to reduce our expression again to  $$\Big\langle \sum_{\alpha' < \alpha}B_{D^{\alpha-\alpha'} \xi_i}D^{\alpha'}B_{\xi_i}f,D^\alpha f\Big\rangle + \Big\langle \sum_{\beta < \alpha} B_{D^{\alpha-\beta}\xi_i}D^\beta f + B_{\xi_i}D^{\alpha}f,\big(\mathcal{T}_{\xi_i} + \mathcal{T}_{\xi_i}^*\big) D^\alpha f + \sum_{\alpha' < \alpha}B_{D^{\alpha-\alpha'} \xi_i}D^{\alpha'}f\Big\rangle.$$ Ultimately the terms in the summand are split up into
\begin{align}
    \label{one}& \inner{B_{\xi_i}D^\alpha f}{\big(\mathcal{T}_{\xi_i}+\mathcal{T}^*_{\xi_i}\big) D^\alpha f}&\\
    \label{two}+ &\Big\langle \sum_{\beta < \alpha} B_{D^{\alpha-\beta} \xi_i}D^{\beta}f ,\big(\mathcal{T}_{\xi_i} + \mathcal{T}_{\xi_i}^*\big) D^\alpha f + \sum_{\alpha' < \alpha} B_{D^{\alpha-\alpha'} \xi_i}D^{\alpha'}f\Big\rangle\\
    \label{three}+ & \sum_{\alpha' < \alpha}\bigg(\inner{B_{D^{\alpha-\alpha'} \xi_i}D^{\alpha'}B_{\xi_i}f}{D^\alpha f} + \inner{B_{\xi_i}D^\alpha f}{B_{D^{\alpha-\alpha'} \xi_i}D^{\alpha'}f}\bigg)
\end{align}
with the intention of controlling each one individually. Firstly for a treatment of (\ref{one}),
    \begin{align*}(\ref{one}) &=\inner{(\mathcal{L}_{\xi_i} + \mathcal{T}_{\xi_i})D^{\alpha}f}{(\mathcal{T}_{\xi_i}^* + \mathcal{T}_{\xi_i})D^{\alpha}f}\\ &= \inner{\mathcal{L}_{\xi_i}D^{\alpha}f}{\mathcal{T}_{\xi_i}^*D^{\alpha}f} + \inner{\mathcal{L}_{\xi_i}D^{\alpha}f}{\mathcal{T}_{\xi_i}D^{\alpha}f} + \inner{\mathcal{T}_{\xi_i}D^{\alpha}f}{\mathcal{T}_{\xi_i}^*D^{\alpha}f} + \inner{\mathcal{T}_{\xi_i}D^{\alpha}f}{\mathcal{T}_{\xi_i}D^{\alpha}f}\\
    &= \Big(\inner{\mathcal{T}_{\xi_i}\mathcal{L}_{\xi_i}D^{\alpha}f}{D^{\alpha}f} + \inner{\mathcal{L}_{\xi_i}D^{\alpha}f}{\mathcal{T}_{\xi_i}D^{\alpha}f}\Big) + \Big(\inner{\mathcal{T}_{\xi_i}^2D^{\alpha}f}{D^{\alpha}f} + \inner{\mathcal{T}_{\xi_i}D^{\alpha}f}{\mathcal{T}_{\xi_i}D^{\alpha}f}\Big).
    \end{align*}
    We now bound the brackets in terms of $\norm{D^{\alpha}f}^2$ separately, starting with the latter one as
    \begin{align*}\inner{\mathcal{T}_{\xi_i}^2D^{\alpha}f}{D^{\alpha}f} &\leq \norm{\mathcal{T}_{\xi_i}^2D^{\alpha}f}\norm{D^{\alpha}f}\\ &\leq \sqrt{N}\norm{\xi_i}_{W^{1,\infty}}\norm{\mathcal{T}_{\xi_i}D^{\alpha}f}\norm{D^{\alpha}f}\\ &\leq N\norm{\xi_i}^2_{W^{1,\infty}}\norm{D^{\alpha}f}^2
    \end{align*}
    from \ref{biggestcorcor} for the result in $L^2(\mathcal{O};\R^N)$, and similarly $$\inner{\mathcal{T}_{\xi_i}D^{\alpha}f}{\mathcal{T}_{\xi_i}D^{\alpha}f} \leq \norm{\mathcal{T}_{\xi_i}D^{\alpha}f}\norm{\mathcal{T}_{\xi_i}D^{\alpha}f} \leq N\norm{\xi_i}^2_{W^{1,\infty}}\norm{D^{\alpha}f}^2.$$ Now for the first bracket, we add and subtract a term to have an expression through the commutator of the operators:
    \begin{align*}
    &\inner{\mathcal{T}_{\xi_i}\mathcal{L}_{\xi_i}D^{\alpha}f}{D^{\alpha}f} + \inner{\mathcal{L}_{\xi_i}D^{\alpha}f}{\mathcal{T}_{\xi_i}D^{\alpha}f}\\ = & \inner{(\mathcal{T}_{\xi_i}\mathcal{L}_{\xi_i} - \mathcal{L}_{\xi_i}\mathcal{T}_{\xi_i}) D^{\alpha}f}{D^{\alpha}f} + \inner{\mathcal{L}_{\xi_i}\mathcal{T}_{\xi_i}D^{\alpha}f}{D^{\alpha}f} + \inner{\mathcal{L}_{\xi_i}D^{\alpha}f}{\mathcal{T}_{\xi_i}D^{\alpha}f}\\
    = & \inner{(\mathcal{T}_{\xi_i}\mathcal{L}_{\xi_i} - \mathcal{L}_{\xi_i}\mathcal{T}_{\xi_i}) D^{\alpha}f}{D^{\alpha}f} + \inner{\mathcal{T}_{\xi_i}D^{\alpha}f}{\mathcal{L}_{\xi_i}^*D^{\alpha}f} + \inner{\mathcal{L}_{\xi_i}D^{\alpha}f}{\mathcal{T}_{\xi_i}D^{\alpha}f}\\
    = & \inner{(\mathcal{T}_{\xi_i}\mathcal{L}_{\xi_i} - \mathcal{L}_{\xi_i}\mathcal{T}_{\xi_i}) D^{\alpha}f}{D^{\alpha}f}.
    \end{align*}
    The commutator term is given explicitly through
    \begin{align*}
        \mathcal{T}_{\xi_i}\mathcal{L}_{\xi_i}D^{\alpha}f &= \mathcal{T}_{\xi_i}\Big(\sum_{j=1}^N\xi_i^j\partial_jD^{\alpha}f\Big)\\
        &= \sum_{k=1}^N\Big(\sum_{j=1}^N\xi_i^j\partial_jD^{\alpha}f\Big)^k\nabla \xi_i^k\\
        &= \sum_{k=1}^N\sum_{j=1}^N\xi_i^j\partial_jD^{\alpha}f^k\nabla \xi_i^k
    \end{align*}
    and
    \begin{align*}
        \mathcal{L}_{\xi_i}\mathcal{T}_{\xi_i}D^{\alpha}f &= \mathcal{L}_{\xi_i}\Big(\sum_{k=1}^ND^{\alpha}f^k\nabla \xi_i^k\Big)\\
        &= \sum_{j=1}^N\xi_i^j\partial_j\Big(\sum_{k=1}^ND^{\alpha}f^k\nabla \xi_i^k\Big)\\
        &= \sum_{j=1}^N\sum_{k=1}^N\xi_i^j\partial_j \big(D^{\alpha}f^k\nabla \xi_i^k\big)\\
        &= \sum_{j=1}^N\sum_{k=1}^N\Big(\xi_i^j\partial_jD^{\alpha}f^k \nabla \xi_i^k + \xi_i^jD^{\alpha}f^k\partial_j\nabla\xi_i^k\Big)
    \end{align*}
    such that $$(\mathcal{T}_{\xi_i}\mathcal{L}_{\xi_i} - \mathcal{L}_{\xi_i}\mathcal{T}_{\xi_i}) D^{\alpha}f = \sum_{j=1}^N\sum_{k=1}^N\xi_i^jD^{\alpha}f^k\partial_j\nabla\xi_i^k$$
    therefore
    \begin{align*}
        \norm{(\mathcal{T}_{\xi_i}\mathcal{L}_{\xi_i} - \mathcal{L}_{\xi_i}\mathcal{T}_{\xi_i}) D^{\alpha}f}^2 &\leq N^2\sum_{j=1}^N\sum_{k=1}^N\norm{\xi_i^jD^{\alpha}f^k\partial_j\nabla\xi_i^k}^2\\
        &= N^2\sum_{j=1}^N\sum_{k=1}^N\sum_{l=1}^N\norm{\xi_i^jD^{\alpha}f^k\partial_j\partial_l\xi_i^k}^2_{L^2(\mathcal{O};\R)}\\
        &\leq N^2\sum_{j=1}^N\sum_{k=1}^N\sum_{l=1}^N\norm{\xi_i^j\partial_j\partial_l\xi_i^k}^2_{L^\infty(\mathcal{O};\R)}\norm{D^{\alpha}f^k}^2_{L^2(\mathcal{O};\R)}\\
        &\leq N^2\norm{\xi_i}^4_{W^{2,\infty}}\sum_{j=1}^N\sum_{k=1}^N\sum_{l=1}^N\norm{D^{\alpha}f^k}^2_{L^2(\mathcal{O};\R)}\\
        &= N^4\norm{\xi_i}^4_{W^{2,\infty}}\norm{D^{\alpha}f}^2
    \end{align*}
    and
    $$\inner{(\mathcal{T}_{\xi_i}\mathcal{L}_{\xi_i} - \mathcal{L}_{\xi_i}\mathcal{T}_{\xi_i}) D^{\alpha}f}{D^{\alpha}f} \leq \norm{(\mathcal{T}_{\xi_i}\mathcal{L}_{\xi_i} - \mathcal{L}_{\xi_i}\mathcal{T}_{\xi_i}) D^{\alpha}f}\norm{D^{\alpha}f} \leq N^2\norm{\xi_i}^2_{W^{2,\infty}}\norm{D^{\alpha}f}^2.$$
    Combining these inequalities we determine the bound
\begin{align*}(\ref{one}) &\leq \big(N^2\norm{\xi_i}^2_{W^{2,\infty}} + 2N\norm{\xi_i}^2_{W^{2,\infty}}\big)\norm{D^{\alpha}f}^2\\
&\leq 3N^2\norm{\xi_i}^2_{W^{2,\infty}}\norm{D^\alpha f}^2.
\end{align*}
As for (\ref{two}) we look to use Cauchy-Schwartz and bound each item in the inner product. Indeed straight from (\ref{boundsonB_i}) in the $L^2(\mathcal{O};\R^N)$ setting, by simply replacing $\xi_i$ with $D^{\alpha-\beta}\xi_i$, we have that 
\begin{align*}\norm{B_{D^{\alpha-\beta}\xi_i}D^{\beta}f}^2 &\leq 4N\norm{D^{\alpha-\beta}\xi_i}^2_{W^{1,\infty}}\norm{D^\beta f}^2_{W^{1,2}}\\
&\leq 4N\norm{\xi_i}^2_{W^{k+1,\infty}}\norm{ f}^2_{W^{k,2}}\end{align*}
Moreover,
\begin{align*}
    \bigg\vert\bigg\vert \sum_{\beta < \alpha}B_{D^{\alpha-\beta}\xi_i}D^{\beta}f\bigg\vert\bigg\vert^2 &\leq 2^{k-1}\sum_{\beta < \alpha}\norm{B_{D^{\alpha-\beta}\xi_i}D^{\beta}f}^2\\
    &\leq 2^{k+1}N\norm{\xi_i}^2_{W^{k+1,\infty}}\norm{f}^2_{W^{k,2}}.
\end{align*}
In addition to this, we have again the result from \ref{biggestcorcor} in the $L^2(\mathcal{O};\R^N)$ setting that
$$\norm{\big(\mathcal{T}_{\xi_i}+\mathcal{T}^*_{\xi_i}\big) D^\alpha f} \leq \norm{\mathcal{T}_{\xi_i} D^\alpha f} + \norm{\mathcal{T}^*_{\xi_i} D^\alpha f} \leq 2\sqrt{N}\norm{\xi_i}_{W^{k+1,\infty}}\norm{D^\alpha f}$$
using the equivalence in operator norm of the adjoint. Together this amounts to
\begin{align*}
    (\ref{two}) &\leq \bigg\vert\bigg\vert  \sum_{\beta < \alpha}B_{D^{\alpha-\beta}\xi_i}D^{\beta}f\bigg\vert \bigg\vert \cdot \bigg\vert\bigg\vert \big(\mathcal{T}_{\xi_i}+\mathcal{T}^*_{\xi_i}\big) D^\alpha f +  \sum_{\alpha' < \alpha} B_{D^{\alpha-\alpha'}\xi_i}D^{\alpha'}f\bigg\vert\bigg\vert\\ &\leq 
    \bigg\vert\bigg\vert  \sum_{\beta < \alpha}B_{D^{\alpha-\beta}\xi_i}D^{\beta}f\bigg\vert\bigg\vert \Bigg(\norm{\big(\mathcal{T}_{\xi_i}+\mathcal{T}^*_{\xi_i}\big) D^\alpha f} + \bigg\vert\bigg\vert \sum_{\alpha' < \alpha} B_{D^{\alpha-\alpha'}\xi_i}D^{\alpha'}f\bigg\vert\bigg\vert\Bigg)\\
    &\leq \sqrt{2^{k+1}N}\norm{\xi_i}_{W^{k+1,\infty}}\norm{f}_{W^{k,2}}\Big(2\sqrt{N}\norm{\xi_i}_{W^{k+1,\infty}}\norm{D^\alpha f} + \sqrt{2^{k+1}N}\norm{\xi_i}_{W^{k+1,\infty}}\norm{f}_{W^{k,2}}\Big)\\
    &\leq \sqrt{2^{k+2}N}\norm{\xi_i}_{W^{k+1,\infty}}\norm{f}_{W^{k,2}}\Big(\sqrt{2^{k+2}N}\norm{\xi_i}_{W^{k+1,\infty}}\norm{D^\alpha f} + \sqrt{2^{k+2}N}\norm{\xi_i}_{W^{k+1,\infty}}\norm{f}_{W^{k,2}}\Big)\\
    &\leq 2^{k+3}N\norm{\xi_i}_{W^{k+1,\infty}}^2\norm{f}_{W^{k,2}}^2.\\
\end{align*}
Let's now turn our attentions to (\ref{three}), which for each $\alpha'$ in the sum we rewrite as \begin{equation} \label{rewriteofprime}\inner{D^\alpha f}{B_{D^{\alpha-\alpha'}\xi_i}D^{\alpha'}B_{\xi_i}f + B^*_{\xi_i}B_{D^{\alpha-\alpha'} \xi_i}D^{\alpha'}f}\end{equation} and employing (\ref{thenevidently}) again we see this becomes \begin{align*}& \Big\langle D^\alpha f,B_{D^{\alpha-\alpha'}\xi_i}\bigg(\sum_{\beta < \alpha'}B_{D^{\alpha'-\beta}\xi_i}D^\beta f + B_{\xi_i}D^{\alpha'}f\bigg)+ B^*_{\xi_i}B_{D^{\alpha-\alpha'} \xi_i}D^{\alpha'}f\Big\rangle\\
= & \Big\langle D^\alpha f,\sum_{\beta < \alpha'}B_{D^{\alpha-\alpha'}\xi_i}B_{D^{\alpha'-\beta}\xi_i}D^\beta f\Big\rangle + \inner{D^\alpha f}{B_{D^{\alpha-\alpha'}\xi_i}B_{\xi_i}D^{\alpha'}f + B^*_{\xi_i}B_{D^{\alpha-\alpha'} \xi_i}D^{\alpha'}f}.  \end{align*}
We have split up these terms to make our approach clearer, as the two right hand sides of the inner products will be considered separately. For the first inner product, two applications of (\ref{boundsonB_i}) give that
\begin{align*}
    \norm{B_{D^{\alpha-\alpha'}\xi_i}B_{D^{\alpha'-\beta}\xi_i}D^\beta f}^2 &\leq 4N\norm{D^{\alpha-\alpha'}\xi_i}^2_{W^{1,\infty}}\norm{B_{D^{\alpha'-\beta}\xi_i}D^\beta f}^2_{W^{1,2}}\\
    &\leq 4N\norm{D^{\alpha-\alpha'}\xi_i}^2_{W^{1,\infty}}\Big(8N(N+1)\norm{D^{\alpha'-\beta}\xi_i}^2_{W^{2,\infty}}\norm{D^\beta f}^2_{W^{2,2}}\Big)\\
    &\leq 32N^2(N+1)\norm{\xi_i}_{W^{k+1,\infty}}^4\norm{f}^2_{W^{k,2}}
\end{align*}
noting that the $N+1$ comes from the number of distinct multi-indices of order up to $1$. Moreover,
\begin{align*}
    \bigg\vert\bigg\vert \sum_{\beta < \alpha'}B_{D^{\alpha-\alpha'}\xi_i}B_{D^{\alpha'-\beta}\xi_i}D^\beta f\bigg\vert\bigg\vert^2 &\leq 2^{k-2}\norm{B_{D^{\alpha-\alpha'}\xi_i}B_{D^{\alpha'-\beta}\xi_i}D^\beta f}^2\\
    &\leq 2^{k+3}N^2(N+1)\norm{\xi_i}_{W^{k+1,\infty}}^4\norm{f}^2_{W^{k,2}}.
\end{align*}
As for the second inner product, we rewrite the right side as $$B_{D^{\alpha-\alpha'} \xi_i}\big((\mathcal{L}_{\xi_i} + \mathcal{T}_{\xi_i})D^{\alpha'}f\big) + (\mathcal{L}^*_{\xi_i} + \mathcal{T}^*_{\xi_i})B_{D^{\alpha-\alpha'} \xi_i}D^{\alpha'}f$$ and further \begin{equation}\label{theonewiththecommutator}\big(B_{D^{\alpha -\alpha'} \xi_i}\mathcal{L}_{\xi_i} - \mathcal{L}_{\xi_i}B_{D^{\alpha-\alpha'} \xi_i}\big)D^{\alpha'}f + B_{D^{\alpha-\alpha'} \xi_i}\mathcal{T}_{\xi_i}D^{\alpha'}f + \mathcal{T}_{\xi_i}^*B_{D^{\alpha-\alpha'} \xi_i}D^{\alpha'}f.\end{equation}
Starting with the latter two terms, we use the familiar (\ref{boundsonB_i}) and \ref{biggestcorcor} for the bound
\begin{align*}
    \norm{B_{D^{\alpha-\alpha'} \xi_i}\mathcal{T}_{\xi_i}D^{\alpha'}f}^2 &\leq 4N\norm{D^{\alpha-\alpha'} \xi_i}^2_{W^{1,\infty}}\norm{\mathcal{T}_{\xi_i}D^{\alpha'}f}^2_{W^{1,2}}\\
    &\leq 4N\norm{D^{\alpha-\alpha'} \xi_i}^2_{W^{1,\infty}} \Big(2N(N+1)\norm{\xi_i}^2_{W^2,\infty}\norm{D^{\alpha'}f}^2_{W^{1,2}}\Big)\\
    &\leq 8N^2(N+1)\norm{\xi_i}^4_{W^{k+1,\infty}}\norm{f}^2_{W^{k,2}}
\end{align*}
and likewise
\begin{align*}
\norm{\mathcal{T}_{\xi_i}^*B_{D^{\alpha-\alpha'} \xi_i}D^{\alpha'}f}^2 &\leq N\norm{\xi_i}^2_{W^{1,\infty}}\norm{B_{D^{\alpha-\alpha'} \xi_i}D^{\alpha'}f}^2\\
&\leq N\norm{\xi_i}^2_{W^{1,\infty}}\Big(4N\norm{D^{\alpha-\alpha'} \xi_i}^2_{W^{1,\infty}}\norm{D^{\alpha'}f}^2_{W^{1,2}}\Big)\\
&\leq 4N^2\norm{\xi_i}^4_{W^{k+1,\infty}}\norm{f}^2_{W^{k,2}}
\end{align*}
Now we show explicitly that the commutator
\begin{equation} \label{theactualcommutator}
    \big(B_{D^{\alpha -\alpha'} \xi_i}\mathcal{L}_{\xi_i} - \mathcal{L}_{\xi_i}B_{D^{\alpha-\alpha'} \xi_i}\big)D^{\alpha'}f
\end{equation}
given in (\ref{theonewiththecommutator}) is of first order (so of $k^{\textnormal{th}}$ order when composed with $D^{\alpha'}$), through the expressions
\begin{align*}
    B_{D^{\alpha-\alpha'} \xi_i}\mathcal{L}_{\xi_i}D^{\alpha'}f &= \sum_{j=1}^N\bigg(D^{\alpha-\alpha'}\xi_i^j \partial_j\Big(\sum_{k=1}^N\xi_i^k\partial_k D^{\alpha'}f\Big) + \Big(\sum_{k=1}^N\xi_i^k\partial_k D^{\alpha'}f\Big)^j\nabla D^{\alpha-\alpha'} \xi_i^j\bigg)\\
    &= \sum_{j=1}^N\sum_{k=1}^N\bigg(D^{\alpha-\alpha'}\xi_i^j\partial_j\xi_i^k\partial_kD^{\alpha'}f + D^{\alpha-\alpha'}\xi_i^j\xi_i^k\partial_j\partial_kD^{\alpha'}f + \xi_i^k\partial_k D^{\alpha'}f^j \nabla D^{\alpha-\alpha'} \xi_i^j\bigg) 
\end{align*}
and
\begin{align*}
    \mathcal{L}_{\xi_i}B_{D^{\alpha-\alpha'} \xi_i}D^{\alpha'}f &= \sum_{k=1}^N\xi_i^k\partial_k\bigg(\sum_{j=1}^ND^{\alpha-\alpha'} \xi_i^j\partial_jD^{\alpha'}f + D^{\alpha'}f^j \nabla D^{\alpha-\alpha'} \xi_i^j\bigg)\\
    &= \sum_{j=1}^N\sum_{k=1}^N\bigg(\xi_i^k\partial_kD^{\alpha-\alpha'} \xi_i^j \partial_jD^{\alpha'}f + \xi_i^kD^{\alpha-\alpha'}\xi_i^j\partial_k\partial_jD^{\alpha'}f \\& \qquad \qquad \qquad \qquad \qquad + \xi_i^k\partial_kD^{\alpha'}f^j \nabla D^{\alpha-\alpha'} \xi_i^j + \xi_i^kD^{\alpha'}f^j\partial_k \nabla D^{\alpha-\alpha'} \xi_i^j\bigg)
\end{align*}
such that
\begin{equation} \nonumber (\ref{theactualcommutator}) = \sum_{j=1}^N\sum_{k=1}^N\bigg(D^{\alpha-\alpha'}\xi_i^j\partial_j\xi_i^k\partial_kD^{\alpha'}f - \xi_i^k\partial_kD^{\alpha-\alpha'} \xi_i^j \partial_jD^{\alpha'}f - \xi_i^kD^{\alpha'}f^j\partial_k \nabla D^{\alpha-\alpha'} \xi_i^j\bigg)\end{equation}
and in particular
\begin{align*}
    \norm{(\ref{theactualcommutator})}^2 &\leq 3N^2\sum_{j=1}^N\sum_{k=1}^N\bigg(\norm{D^{\alpha-\alpha'}\xi_i^j\partial_j\xi_i^k\partial_kD^{\alpha'}f}^2 + \norm{\xi_i^k\partial_kD^{\alpha-\alpha'} \xi_i^j \partial_jD^{\alpha'}f}^2 + \norm{\xi_i^kD^{\alpha'}f^j\partial_k \nabla D^{\alpha-\alpha'} \xi_i^j}^2\bigg)\\
    &= 3N^2\sum_{j=1}^N\sum_{k=1}^N\sum_{l=1}^N\bigg(\norm{D^{\alpha-\alpha'}\xi_i^j\partial_j\xi_i^k\partial_kD^{\alpha'}f^l}^2_{L^2(\mathcal{O};\R)} \\ & \qquad \qquad \qquad  \qquad \qquad + \norm{\xi_i^k\partial_kD^{\alpha-\alpha'} \xi_i^j \partial_jD^{\alpha'}f^l}^2_{L^2(\mathcal{O};\R)} + \norm{\xi_i^kD^{\alpha'}f^j\partial_k \partial_l D^{\alpha-\alpha'} \xi_i^j}^2_{L^2(\mathcal{O};\R)}\bigg)\\
    &\leq 3N^2\norm{\xi_i}^4_{W^{k+2,\infty}}\sum_{j=1}^N\sum_{k=1}^N\sum_{l=1}^N\bigg(\norm{\partial_k D^{\alpha'}f^l}^2_{L^2(\mathcal{O};\R)} + \norm{\partial_j D^{\alpha'}f^l}^2_{L^2(\mathcal{O};\R)} + \norm{D^{\alpha'}f^l}^2_{L^2(\mathcal{O};\R)}\bigg)\\
    &\leq 3N^2\norm{\xi_i}^4_{W^{k+2,\infty}}\bigg(\sum_{j=1}^N\norm{f}^2_{W^{k,2}} + \sum_{k=1}^N\norm{f}^2_{W^{k,2}} + \sum_{j=1}^N\sum_{k=1}^N\norm{D^{\alpha'}f}^2\bigg)\\
    &\leq 9N^4\norm{\xi_i}^4_{W^{k+2,\infty}}\norm{f}^2_{W^{k,2}}.\\
\end{align*}
Finally now we can piece these four inequalities together to produce a bound on (\ref{rewriteofprime}):
\begin{align*}
    (\ref{rewriteofprime}) &\leq \norm{D^\alpha f}\bigg\vert\bigg\vert\sum_{\beta < \alpha'}B_{D^{\alpha-\alpha'}\xi_i}B_{D^{\alpha'-\beta}\xi_i}D^\beta f + (\ref{theonewiththecommutator})\bigg\vert\bigg\vert\\
    &\leq \norm{D^\alpha f}\bigg(\bigg\vert\bigg\vert\sum_{\beta < \alpha'}B_{D^{\alpha-\alpha'}\xi_i}B_{D^{\alpha'-\beta}\xi_i}D^\beta f\bigg\vert\bigg\vert + \norm{B_{D^{\alpha-\alpha'} \xi_i}\mathcal{T}_{\xi_i}D^{\alpha'}f} + \norm{\mathcal{T}_{\xi_i}^*B_{D^{\alpha-\alpha'} \xi_i}D^{\alpha'}f} + \norm{(\ref{theactualcommutator})}\bigg)\\
    &\leq \norm{D^\alpha f}\bigg(\sqrt{2^{k+3}(N+1)}N\norm{\xi_i}_{W^{k+1,\infty}}^2\norm{f}_{W^{k,2}} + \sqrt{8(N+1)}N\norm{\xi_i}^2_{W^{k+1,\infty}}\norm{f}_{W^{k,2}}\\& \qquad \qquad \qquad \qquad \qquad \qquad \qquad \qquad \qquad + 2N\norm{\xi_i}^2_{W^{k+1,\infty}}\norm{f}_{W^{k,2}} + 3N^2\norm{\xi_i}^2_{W^{k+2,\infty}}\norm{f}_{W^{k,2}}\bigg)\\
    &\leq \norm{D^\alpha f} \Big((2^{\frac{k+8}{2}}+3)N^2\norm{\xi_i}_{W^{k+2,\infty}}^2\norm{f}_{W^{k,2}}\Big)\\
    &\leq (2^{\frac{k+8}{2}}+3)N^2\norm{\xi_i}_{W^{k+2,\infty}}^2\norm{f}_{W^{k,2}}^2
\end{align*}
using that for $N \geq 2$ we have $\sqrt{N+1} \leq N$, and subsequently of (\ref{three}):
\begin{align*}
    (\ref{three}) &= \sum_{\alpha' < \alpha}(\ref{rewriteofprime})\\
    &\leq \sum_{\alpha' < \alpha}(2^{\frac{k+8}{2}}+3)N^2\norm{\xi_i}_{W^{k+2,\infty}}^2\norm{f}_{W^{k,2}}^2\\
    &\leq 2^k(2^{\frac{k+8}{2}}+3)N^2\norm{\xi_i}_{W^{k+2,\infty}}^2\norm{f}_{W^{k,2}}^2.
\end{align*}
We can now conclude the proof by observing that
\begin{align*}
    (\ref{commandv}) &= \sum_{\abs{\alpha}\leq k} (\ref{thatiswearelookingat})\\
    &= \sum_{\abs{\alpha}\leq k} (\ref{one}) + (\ref{two}) + (\ref{three})\\
    &\leq \sum_{\abs{\alpha}\leq k} \bigg(3N^2\norm{\xi_i}^2_{W^{2,\infty}}\norm{D^\alpha f}^2 + 2^{k+3}N\norm{\xi_i}_{W^{k+1,\infty}}^2\norm{f}_{W^{k,2}}^2\\ & \qquad \qquad\qquad\qquad\qquad\qquad\qquad\qquad + 2^k(2^{\frac{k+8}{2}}+3)N^2\norm{\xi_i}_{W^{k+2,\infty}}^2\norm{f}_{W^{k,2}}^2\bigg)\\
    &\leq \sum_{\abs{\alpha}\leq k} \bigg(3N^2\norm{\xi_i}^2_{W^{k+2,\infty}}\norm{f}_{W^{k,2}}^2 + 2^{k+3}N\norm{\xi_i}_{W^{k+2,\infty}}^2\norm{f}_{W^{k,2}}^2\\ & \qquad \qquad\qquad\qquad\qquad\qquad\qquad\qquad + (2^{2k+4}+2^k\cdot 3)N^2\norm{\xi_i}_{W^{k+2,\infty}}^2\norm{f}_{W^{k,2}}^2\bigg)\\
    &\leq \sum_{\abs{\alpha}\leq k} \bigg(3\cdot2^{2k+4}N^2\norm{\xi_i}^2_{W^{k+2,\infty}}\norm{f}_{W^{k,2}}^2 + 2^{2k+4}N^2\norm{\xi_i}_{W^{k+2,\infty}}^2\norm{f}_{W^{k,2}}^2\\ & \qquad \qquad\qquad\qquad\qquad\qquad\qquad\qquad + (2^{2k+4}+2^{2k+4}\cdot 3)N^2\norm{\xi_i}_{W^{k+2,\infty}}^2\norm{f}_{W^{k,2}}^2\bigg)\\
    &= \sum_{\abs{\alpha}\leq k} 2^{2k+7}N^2\norm{\xi_i}_{W^{k+2,\infty}}^2\norm{f}_{W^{k,2}}^2\\
    &= 2^{2k+7}N^2\eta_k\norm{\xi_i}_{W^{k+2,\infty}}^2\norm{f}_{W^{k,2}}^2.
\end{align*}
\end{proof}

\begin{proof}[Proof of \ref{prop for conservation B_i}, (\ref{finalboundinderivativeproof}):]
    Using (\ref{thenevidently}) once more, we see that for each $\alpha$ in the sum for the inner product,
    \begin{align*}
        \abs{\inner{D^\alpha B_i f}{D^\alpha f}} &= \Big\vert\Big \langle \sum_{\alpha' < \alpha}B_{D^{\alpha-\alpha'}\xi_i}D^{\alpha'}f + B_{\xi_i}D^\alpha f, D^\alpha f \Big\rangle\Big\vert\\
        &= \Big\vert\Big \langle \sum_{\alpha' < \alpha}B_{D^{\alpha-\alpha'}\xi_i}D^{\alpha'}f , D^\alpha f \Big\rangle + \inner{B_{\xi_i}D^\alpha f}{D^\alpha f}\Big\vert\\
        &\leq \Big\vert\Big \langle \sum_{\alpha' < \alpha}B_{D^{\alpha-\alpha'}\xi_i}D^{\alpha'}f , D^\alpha f \Big\rangle\Big\vert + \abs{\inner{\mathcal{T}_{\xi_i}D^\alpha f}{D^\alpha f}}
    \end{align*}
    using the cancellation from \ref{biggglemma} to dispose of the order $k+1$ term. In our treatment of (\ref{two}) in \ref{combinedterminenergyinequality}, we showed the bound
    \begin{align*}
    \bigg\vert\bigg\vert \sum_{\beta < \alpha}B_{D^{\alpha-\beta}\xi_i}D^{\beta}f\bigg\vert\bigg\vert
    &\leq \sqrt{2^{k+1}N}\norm{\xi_i}_{W^{k+1,\infty}}\norm{f}_{W^{k,2}}
\end{align*}
and therefore
\begin{align*}\Big\vert\Big \langle \sum_{\alpha' < \alpha}B_{D^{\alpha-\alpha'}\xi_i}D^{\alpha'}f , D^\alpha f \Big\rangle\Big\vert &\leq \sqrt{2^{k+1}N}\norm{\xi_i}_{W^{k+1,\infty}}\norm{f}_{W^{k,2}}\norm{D^\alpha f}\\
&\leq \sqrt{2^{k+1}N}\norm{\xi_i}_{W^{k+1,\infty}}\norm{f}_{W^{k,2}}^2\end{align*}
whilst in the treatment of (\ref{one}) we noted $$\inner{\mathcal{T}_{\xi_i}D^{\alpha}f}{\mathcal{T}_{\xi_i}D^{\alpha}f}\leq N\norm{\xi_i}^2_{W^{1,\infty}}\norm{D^{\alpha}f}^2 \leq N\norm{\xi_i}^2_{W^{k+1,\infty}}\norm{f}_{W^{k,2}}^2.$$ Combining these terms, we loosen to the bound $$\abs{\inner{D^\alpha B_i f}{f}} \leq (\norm{\xi_i}_{W^{k+1,\infty}}+ \sqrt{2^{k+1}})N\norm{\xi_i}_{W^{k+1,\infty}}\norm{f}^2_{W^{k,2}}$$
and see therefore that $$\abs{\inner{B_if}{f}_{W^{k,2}}} \leq  (\norm{\xi_i}_{W^{k+1,\infty}}+ \sqrt{2^{k+1}})N\eta_k\norm{\xi_i}_{W^{k+1,\infty}}\norm{f}^2_{W^{k,2}}$$ which readily gives the result.
\end{proof}

\begin{proof}[Proof of \ref{boundoncommutator}]
    We fix any such $f$ and first show that \begin{equation}\label{identity of commutator}[\Delta, B_i]f = \sum_{k=1}^N\sum_{j=1}^N\left(\partial_k^2\xi_i^j\partial_jf + 2 \partial_k\xi_i^j\partial_k\partial_jf + 2\partial_kf^j \partial_k\nabla \xi_i^j + f^j \partial_k^2\nabla \xi_i^j \right).\end{equation}
        Indeed
    \begin{align*}
        \Delta B_i f &= \sum_{k=1}^N\partial_k^2\left( \sum_{j=1}^N\left(
    \xi_i^j\partial_jf + f^j \nabla \xi_i^j\right)\right)\\
        &= \sum_{k=1}^N\partial_k\left(\sum_{j=1}^N\left(\partial_k\xi_i^j\partial_jf + \xi_i^j\partial_k\partial_jf + \partial_kf^j \nabla \xi_i^j + f^j \partial_k \nabla \xi_i^j\right)\right)\\
        &= \sum_{k=1}^N\sum_{j=1}^N\left(\partial_k^2\xi_i^j\partial_jf + 2 \partial_k\xi_i^j\partial_k\partial_jf + \xi_i^j\partial_k^2\partial_jf + \partial_k^2f^j \nabla \xi_i^j + 2\partial_kf^j \partial_k\nabla \xi_i^j + f^j \partial_k^2\nabla \xi_i^j\right)
    \end{align*}
    
    and 
    
    \begin{align*}
        B_i\Delta f &= \sum_{j=1}^N\left(\xi_i^j\partial_j\left(\sum_{k=1}^N\partial_k^2f\right) + \left(\sum_{k=1}^N\partial_k^2f\right)^j\nabla \xi_i^j\right)\\
        &= \sum_{k=1}^N\sum_{j=1}^N\left(\xi_i^j\partial_k^2\partial_jf + \partial_k^2f^j\nabla \xi_i^j \right)
    \end{align*}
    
    therefore
    
    \begin{align*}
        [\Delta, B_i]f &= \Delta B_i f - B_i \Delta f\\
        &= \sum_{k=1}^N\sum_{j=1}^N\left(\partial_k^2\xi_i^j\partial_jf + 2 \partial_k\xi_i^j\partial_k\partial_jf + 2\partial_kf^j \partial_k\nabla \xi_i^j + f^j \partial_k^2\nabla \xi_i^j\right) 
    \end{align*}
    justifying (\ref{identity of commutator}). The result then follows from a direct proof:
         \begin{align*}
         \left \Vert [\Delta, B_i] f\right \Vert ^2 &= \left \Vert \sum_{k=1}^N\sum_{j=1}^N\left(\partial_k^2\xi_i^j\partial_jf + 2 \partial_k\xi_i^j\partial_k\partial_jf + 2\partial_kf^j \partial_k\nabla \xi_i^j + f^j \partial_k^2\nabla \xi_i^j\right) \right \Vert^2\\
         &\leq N^2 \sum_{k=1}^N\sum_{j=1}^N \left \Vert\partial_k^2\xi_i^j\partial_jf + 2 \partial_k\xi_i^j\partial_k\partial_jf + 2\partial_kf^j \partial_k\nabla \xi_i^j + f^j \partial_k^2\nabla \xi_i^j \right \Vert ^2\\
         &\leq N^2 \sum_{k=1}^N\sum_{j=1}^N\sum_{l=1}^N \left \Vert\partial_k^2\xi_i^j\partial_jf^l + 2 \partial_k\xi_i^j\partial_k\partial_jf^l + 2\partial_kf^j \partial_k \partial_l \xi_i^j + f^j \partial_k^2\partial_l\xi_i^j \right \Vert_{L^2(\mathcal{O};\R)} ^2\\
         &\leq 4N^2\sum_{k=1}^N\sum_{j=1}^N\sum_{l=1}^N \Big(\norm{\partial_k^2\xi_i^j\partial_jf^l}_{L^2(\mathcal{O};\R)}^2 + \norm{2 \partial_k\xi_i^j\partial_k\partial_jf^l}_{L^2(\mathcal{O};\R)}^2\\& \qquad \qquad \qquad \qquad \qquad \qquad \qquad \qquad  + \norm{2\partial_kf^j \partial_k \partial_l \xi_i^j}_{L^2(\mathcal{O};\R)}^2 + \norm{f^j \partial_k^2\partial_l\xi_i^j}_{L^2(\mathcal{O};\R)}^2\Big)\\
         &\leq 16N^2\norm{\xi_i}_{W^{3,\infty}}^2\sum_{k=1}^N\sum_{j=1}^N\sum_{l=1}^N\Big(\norm{\partial_jf^l}_{L^2(\mathcal{O};\R)}^2 + \norm{\partial_k\partial_jf^l}_{L^2(\mathcal{O};\R)}^2\\& \qquad \qquad \qquad \qquad \qquad \qquad \qquad \qquad  + \norm{\partial_kf^j}_{L^2(\mathcal{O};\R)}^2 + \norm{f^j}_{L^2(\mathcal{O};\R)}^2\Big)\\
         &\leq 64N^4\norm{\xi_i}_{W^{3,\infty}}^2\norm{f}^2_{W^{2,2}}.
     \end{align*}
\end{proof}

\subsection{Appendix III: A Conversion from Stratonovich to It\^{o}} \label{appendix iii}

This theory is taken from [\cite{goodair2022stochastic}] Subsections 2.2 and 2.3, and is provided here for simplicity to apply in Subsection \ref{subsection ito}. We work with a quartet of embedded Hilbert Spaces $$V \hookrightarrow H \hookrightarrow U \hookrightarrow X$$ where the embedding is assumed as a continuous linear injection. We start from an SPDE \begin{equation} \label{stratoSPDE}
    \sy_t = \sy_0 + \int_0^t \mathcal{Q}\sy_sds + \int_0^t\mathcal{G}\sy_s \circ d\mathcal{W}_s.
\end{equation}
where the operators $\mathcal{Q}$, $\mathcal{G}$ satisfy the following conditions, with the general operator $\tilde{K}:H \rightarrow \R$ defined by $$\tilde{K}(\phi):= c\left(1 + \norm{\phi}_U^p + \norm{\phi}_H^q\right)$$ for any constants $c,p,q$ independent of $\phi$.
 \begin{assumption} \label{Qassumpt}
$\mathcal{Q}: V \rightarrow U$ is measurable and for any $\phi \in V$, $$\norm{\mathcal{Q}\phi}_U \leq \tilde{K}(\phi)[1 + \norm{\phi}_V^2].$$
 \end{assumption}
 
\begin{assumption} \label{Gassumpt}
$\mathcal{G}$ is understood as a measurable operator \begin{align*}
    \mathcal{G}&: V \rightarrow \mathscr{L}^2(\mathfrak{U};H)\\
     \mathcal{G}&: H \rightarrow \mathscr{L}^2(\mathfrak{U};U)\\
       \mathcal{G}&: U \rightarrow \mathscr{L}^2(\mathfrak{U};X)
\end{align*}
defined over $\mathfrak{U}$ by its action on the basis vectors $$\mathcal{G}(\cdot, e_i):= \mathcal{G}_i(\cdot).$$ In addition each $\mathcal{G}_i$ is linear and there exists constants $c_i$ such that for all $\phi \in V$, $\psi \in H$, $\eta \in U$:
\begin{align*}
    \norm{\mathcal{G}_i\phi}_{H} &\leq c_i \norm{\phi}_V\\
    \norm{\mathcal{G}_i\psi}_{U} &\leq c_i \norm{\psi}_H\\
    \norm{\mathcal{G}_i\eta}_{X} &\leq c_i \norm{\eta}_U\\
    \sum_{i=1}^\infty c_i^2 &< \infty.
\end{align*}
\end{assumption}

In this setting, we have the following result ([\cite{goodair2022stochastic}] Theorem 2.3.1 and Corollary 2.3.1.1).

\begin{theorem} \label{theorem for ito strat conversion}
    Suppose that $(\sy,\tau)$ are such that: $\tau$ is a $\mathbbm{P}-a.s.$ positive stopping time and $\sy$ is a process whereby for $\mathbbm{P}-a.e.$ $\omega$, $\sy_{\cdot}(\omega) \in C\left([0,T];H\right)$ and $\sy_{\cdot}(\omega)\mathbbm{1}_{\cdot \leq \tau(\omega)} \in L^2\left([0,T];V\right)$ for all $T>0$ with $\sy_{\cdot}\mathbbm{1}_{\cdot \leq \tau}$ progressively measurable in $V$, and moreover satisfy the identity
\begin{equation} \nonumber
    \sy_{t} = \sy_0 + \int_0^{t\wedge \tau} \left(\mathcal{Q} + \frac{1}{2}\sum_{i=1}^\infty \mathcal{G}_i^2\right)\sy_sds + \int_0^{t \wedge \tau}\mathcal{G}\sy_s d\mathcal{W}_s
\end{equation}
$\mathbbm{P}-a.s.$ in $U$ for all $t \geq 0$. Then the pair $(\sy,\tau)$ satisfies the identity $$\sy_{t} = \sy_0 + \int_0^{t\wedge \tau} \mathcal{Q}\sy_sds + \int_0^{t \wedge \tau}\mathcal{G}\sy_s \circ d\mathcal{W}_s $$
$\mathbbm{P}-a.s.$ in $X$ for all $t \geq 0$.
\end{theorem}

The operator $\frac{1}{2}\sum_{i=1}^\infty \mathcal{G}_i^2$ is understood as a pointwise limit, which is justified in [\cite{goodair2022stochastic}] Subsection 2.3.

\subsection{Appendix IV: Abstract Solution Criterion, Part I} \label{appendix iv}

The result is given in the context of an It\^{o} SPDE 
\begin{equation} \label{thespde}
    \sy_t = \sy_0 + \int_0^t \mathcal{A}(s,\sy_s)ds + \int_0^t\mathcal{G} (s,\sy_s) d\mathcal{W}_s.
\end{equation}
We state the assumptions for a triplet of embedded Hilbert Spaces $$V \hookrightarrow H \hookrightarrow U$$ and ask that there is a continuous bilinear form $\inner{\cdot}{\cdot}_{U \times V}: U \times V \rightarrow \R$ such that for $\phi \in H$ and $\psi \in V$, \begin{equation} \label{bilinear formog}
    \inner{\phi}{\psi}_{U \times V} =  \inner{\phi}{\psi}_{H}.
\end{equation}
The operators $\mathcal{A},\mathcal{G}$ are such that for any $T>0$,
    $\mathcal{A}:[0,T] \times V \rightarrow U,
    \mathcal{G}:[0,T] \times V \rightarrow \mathscr{L}^2(\mathfrak{U};H)$ are measurable. We assume that $V$ is dense in $H$ which is dense in $U$. 

\begin{assumption} \label{assumption fin dim spaces}
There exists a system $(a_n)$ of elements of $V$ which are such that, defining the spaces $V_n:= \textnormal{span}\left\{a_1, \dots, a_n \right\}$ and $\mathcal{P}_n$ as the orthogonal projection to $V_n$ in $U$, then:
\begin{enumerate}
    \item There exists some constant $c$ independent of $n$ such that for all $\phi\in H$,
\begin{equation} \label{projectionsboundedonH}
    \norm{\mathcal{P}_n \phi}_H^2 \leq c\norm{\phi}_H^2.
\end{equation}
\item There exists a real valued sequence $(\mu_n)$ with $\mu_n \rightarrow \infty$ such that for any $\phi \in H$, \begin{align}
     \label{mu2}
    \norm{(I - \mathcal{P}_n)\phi}_U \leq \frac{1}{\mu_n}\norm{\phi}_H
\end{align}
where $I$ represents the identity operator in $U$.
\end{enumerate}
\end{assumption}

 These assumptions are of course supplemented by a series of assumptions on the operators. We shall use general notation $c_t$ to represent a function $c_\cdot:[0,\infty) \rightarrow \R$ bounded on $[0,T]$ for any $T > 0$, evaluated at the time $t$. Moreover we define functions $K$, $\tilde{K}$ relative to some non-negative constants $p,\tilde{p},q,\tilde{q}$. We use a generic notation to define the functions $K: U \rightarrow \R$, $K: U \times U \rightarrow \R$, $\tilde{K}: H \rightarrow \R$ and $\tilde{K}: H \times H \rightarrow \R$ by
\begin{align*}
    K(\phi)&:= 1 + \norm{\phi}_U^{p},\\
    K(\phi,\psi)&:= 1+\norm{\phi}_U^{p} + \norm{\psi}_U^{q},\\
    \tilde{K}(\phi) &:= K(\phi) + \norm{\phi}_H^{\tilde{p}},\\
    \tilde{K}(\phi,\psi) &:= K(\phi,\psi) + \norm{\phi}_H^{\tilde{p}} + \norm{\psi}_H^{\tilde{q}}
\end{align*}
 Distinct use of the function $K$ will depend on different constants but in no meaningful way in our applications, hence no explicit reference to them shall be made. In the case of $\tilde{K}$, when $\tilde{p}, \tilde{q} = 2$ then we shall denote the general $\tilde{K}$ by $\tilde{K}_2$. In this case no further assumptions are made on the $p,q$. That is, $\tilde{K}_2$ has the general representation \begin{equation}\label{Ktilde2}\tilde{K}_2(\phi,\psi) = K(\phi,\psi) + \norm{\phi}_H^2 + \norm{\psi}_H^2\end{equation} and similarly as a function of one variable.\\
 
 We state the subsequent assumptions for arbitrary elements $\phi,\psi \in V$, $\phi^n \in V_n$, $\eta \in H$ and $t \in [0,\infty)$, and a fixed $\kappa > 0$. Understanding $\mathcal{G}$ as an operator $\mathcal{G}: [0,\infty) \times V \times \mathfrak{U} \rightarrow H$, we introduce the notation $\mathcal{G}_i(\cdot,\cdot):= \mathcal{G}(\cdot,\cdot,e_i)$.
 
 
  \begin{assumption} \label{new assumption 1} \begin{align}
     \label{111} \norm{\mathcal{A}(t,\boldsymbol{\phi})}^2_U +\sum_{i=1}^\infty \norm{\mathcal{G}_i(t,\boldsymbol{\phi})}^2_H &\leq c_t K(\boldsymbol{\phi})\left[1 + \norm{\boldsymbol{\phi}}_V^2\right],\\ \label{222}
     \norm{\mathcal{A}(t,\boldsymbol{\phi}) - \mathcal{A}(t,\boldsymbol{\psi})}_U^2 &\leq  c_t\left[K(\phi,\psi) + \norm{\phi}_V^2 + \norm{\psi}_V^2\right]\norm{\phi-\psi}_V^2,\\ \label{333}
    \sum_{i=1}^\infty \norm{\mathcal{G}_i(t,\boldsymbol{\phi}) - \mathcal{G}_i(t,\boldsymbol{\psi})}_U^2 &\leq c_tK(\phi,\psi)\norm{\phi-\psi}_H^2.
 \end{align}
 \end{assumption}

\begin{assumption} \label{assumptions for uniform bounds2}
 \begin{align}
   \label{uniformboundsassumpt1}  2\inner{\mathcal{P}_n\mathcal{A}(t,\boldsymbol{\phi}^n)}{\boldsymbol{\phi}^n}_H + \sum_{i=1}^\infty\norm{\mathcal{P}_n\mathcal{G}_i(t,\boldsymbol{\phi}^n)}_H^2 &\leq c_t\tilde{K}_2(\boldsymbol{\phi}^n)\left[1 + \norm{\boldsymbol{\phi}^n}_H^2\right] - \kappa\norm{\boldsymbol{\phi}^n}_V^2,\\  \label{uniformboundsassumpt2}
    \sum_{i=1}^\infty \inner{\mathcal{P}_n\mathcal{G}_i(t,\boldsymbol{\phi}^n)}{\boldsymbol{\phi}^n}^2_H &\leq c_t\tilde{K}_2(\boldsymbol{\phi}^n)\left[1 + \norm{\boldsymbol{\phi}^n}_H^4\right].
\end{align}
\end{assumption}

\begin{assumption} \label{therealcauchy assumptions}
\begin{align}
  \nonumber 2\inner{\mathcal{A}(t,\boldsymbol{\phi}) - \mathcal{A}(t,\boldsymbol{\psi})}{\boldsymbol{\phi} - \boldsymbol{\psi}}_U &+ \sum_{i=1}^\infty\norm{\mathcal{G}_i(t,\boldsymbol{\phi}) - \mathcal{G}_i(t,\boldsymbol{\psi})}_U^2\\ \label{therealcauchy1} &\leq  c_{t}\tilde{K}_2(\boldsymbol{\phi},\boldsymbol{\psi}) \norm{\boldsymbol{\phi}-\boldsymbol{\psi}}_U^2 - \kappa\norm{\boldsymbol{\phi}-\boldsymbol{\psi}}_H^2,\\ \label{therealcauchy2}
    \sum_{i=1}^\infty \inner{\mathcal{G}_i(t,\boldsymbol{\phi}) - \mathcal{G}_i(t,\boldsymbol{\psi})}{\boldsymbol{\phi}-\boldsymbol{\psi}}^2_U & \leq c_{t} \tilde{K}_2(\boldsymbol{\phi},\boldsymbol{\psi}) \norm{\boldsymbol{\phi}-\boldsymbol{\psi}}_U^4.
\end{align}
\end{assumption}

\begin{assumption} \label{assumption for prob in V}
\begin{align}
   \label{probability first} 2\inner{\mathcal{A}(t,\boldsymbol{\phi})}{\boldsymbol{\phi}}_U + \sum_{i=1}^\infty\norm{\mathcal{G}_i(t,\boldsymbol{\phi})}_U^2 &\leq c_tK(\boldsymbol{\phi})\left[1 +  \norm{\boldsymbol{\phi}}_H^2\right],\\\label{probability second}
    \sum_{i=1}^\infty \inner{\mathcal{G}_i(t,\boldsymbol{\phi})}{\boldsymbol{\phi}}^2_U &\leq c_tK(\boldsymbol{\phi})\left[1 + \norm{\boldsymbol{\phi}}_H^4\right].
\end{align}
\end{assumption}

\begin{assumption} \label{finally the last assumption}
 \begin{equation} \label{lastlast assumption}
    \inner{\mathcal{A}(t,\phi)-\mathcal{A}(t,\psi)}{\eta}_U \leq c_t(1+\norm{\eta}_H)\left[K(\phi,\psi) + \norm{\phi}_V + \norm{\psi}_V\right]\norm{\phi-\psi}_H.
    \end{equation}
\end{assumption}

With these assumptions in place we state the relevant definitions and results, first announced in [\cite{Goodair conference}] and proven in [\cite{Goodair abs}]. Definition \ref{definitionofregularsolution} is stated for an $\mathcal{F}_0-$ measurable $\sy_0:\Omega \rightarrow H$.

\begin{definition} \label{definitionofregularsolution}
A pair $(\sy,\tau)$ where $\tau$ is a $\mathbb{P}-a.s.$ positive stopping time and $\sy$ is a process such that for $\mathbb{P}-a.e.$ $\omega$, $\sy_{\cdot}(\omega) \in C\left([0,T];H\right)$ and $\sy_{\cdot}(\omega)\mathbf{1}_{\cdot \leq \tau(\omega)} \in L^2\left([0,T];V\right)$ for all $T>0$ with $\sy_{\cdot}\mathbf{1}_{\cdot \leq \tau}$ progressively measurable in $V$, is said to be an $H$-valued local strong solution of the equation (\ref{thespde}) if the identity
\begin{equation} \label{identityindefinitionoflocalsolution}
    \sy_{t} = \sy_0 + \int_0^{t\wedge \tau} \mathcal{A}(s,\sy_s)ds + \int_0^{t \wedge \tau}\mathcal{G} (s,\sy_s) d\mathcal{W}_s
\end{equation}
holds $\mathbb{P}-a.s.$ in $U$ for all $t \geq 0$.
\end{definition}

\begin{definition} \label{V valued maximal definition}
A pair $(\sy,\Theta)$ such that there exists a sequence of stopping times $(\theta_j)$ which are $\mathbb{P}-a.s.$ monotone increasing and convergent to $\Theta$, whereby $(\sy_{\cdot \wedge \theta_j},\theta_j)$ is a $V-$valued local strong solution of the equation (\ref{thespde}) for each $j$, is said to be an $H-$valued maximal strong solution of the equation (\ref{thespde}) if for any other pair $(\py,\Gamma)$ with this property then $\Theta \leq \Gamma$ $\mathbb{P}-a.s.$ implies $\Theta = \Gamma$ $\mathbb{P}-a.s.$.
\end{definition}

\begin{definition}
An $H-$valued maximal strong solution $(\sy,\Theta)$ of the equation (\ref{thespde}) is said to be unique if for any other such solution $(\py,\Gamma)$, then $\Theta = \Gamma$ $\mathbb{P}-a.s.$ and for all $t \in [0,\Theta)$, \begin{equation} \nonumber\mathbb{P}\left(\left\{\omega \in \Omega: \sy_{t}(\omega) =  \py_{t}(\omega)  \right\} \right) = 1. \end{equation}
\end{definition}

The theorem below is stated on the condition that the assumptions of this subsection are met.

\begin{theorem} \label{theorem1}
For any given $\mathcal{F}_0-$ measurable $\sy_0:\Omega \rightarrow H$, there exists a unique $H-$valued maximal strong solution $(\sy,\Theta)$ of the equation (\ref{thespde}). Moreover at $\mathbb{P}-a.e.$ $\omega$ for which $\Theta(\omega)<\infty$, we have that \begin{equation}\label{blowupproperty}\sup_{r \in [0,\Theta(\omega))}\norm{\sy_r(\omega)}_H^2 + \int_0^{\Theta(\omega)}\norm{\sy_r(\omega)}_V^2dr = \infty.\end{equation}
\end{theorem}

\subsection{Appendix V: Abstract Solution Criterion, Part II} \label{appendix v}

We extend the framework of Appendix IV, \ref{appendix iv}, introducing now another Hilbert Space $X$ which is such that $U \xhookrightarrow{} X$.  We ask that there is a continuous bilinear form $\inner{\cdot}{\cdot}_{X \times H}: X \times H \rightarrow \R$ such that for $\phi \in U$ and $\psi \in H$, \begin{equation} \label{bilinear form}
    \inner{\phi}{\psi}_{X \times H} =  \inner{\phi}{\psi}_{U}.
\end{equation}
Moreover it is now necessary that the system $(a_n)$ forms an orthogonal basis of $U$. We state the remaining assumptions now for arbitrary elements $\phi,\psi \in H$ and $t \in [0,\infty)$, and continue to use the $c,K,\tilde{K}, \kappa$ notation of Assumption Set 1. We now further assume that for any $T>0$, $\mathcal{A}:[0,T] \times H \rightarrow X$ and  $\mathcal{G}:[0,T] \times H \rightarrow \mathscr{L}^2(\mathfrak{U};U)$ are measurable. 

\begin{assumption} \label{gg3}
\begin{align}
    \label{wellposedinX}
    \norm{\mathcal{A}(t,\boldsymbol{\phi})}^2_X +  \sum_{i=1}^\infty \norm{\mathcal{G}_i(t,\boldsymbol{\phi})}^2_U &\leq c_tK(\boldsymbol{\phi})\left[1 + \norm{\boldsymbol{\phi}}_H^2\right],\\ \label{222*} \norm{\mathcal{A}(t,\boldsymbol{\phi}) - \mathcal{A}(t,\boldsymbol{\psi})}_X &\leq  c_t\left[K(\phi,\psi) + \norm{\phi}_H + \norm{\psi}_H \right]\norm{\phi-\psi}_H
\end{align}

\end{assumption}

\begin{assumption} \label{uniqueness for H valued}
\begin{align}
    2\inner{\mathcal{A}(t,\boldsymbol{\phi}) - \mathcal{A}(t,\boldsymbol{\psi})}{\boldsymbol{\phi} - \boldsymbol{\psi}}_X + \sum_{i=1}^\infty\norm{\mathcal{G}_i(t,\boldsymbol{\phi}) - \mathcal{G}_i(t,\boldsymbol{\psi})}_X^2 &\leq \label{therealcauchy1*} c_{t}\tilde{K}_2(\boldsymbol{\phi},\boldsymbol{\psi}) \norm{\boldsymbol{\phi}-\boldsymbol{\psi}}_X^2,\\
    \sum_{i=1}^\infty \inner{\mathcal{G}_i(t,\boldsymbol{\phi}) - \mathcal{G}_i(t,\boldsymbol{\psi})}{\boldsymbol{\phi}-\boldsymbol{\psi}}^2_X & \leq \label{therealcauchy2*} c_{t} \tilde{K}_2(\boldsymbol{\phi},\boldsymbol{\psi}) \norm{\boldsymbol{\phi}-\boldsymbol{\psi}}_X^4
\end{align}
\end{assumption}

\begin{assumption} \label{assumption for probability in H}
With the stricter requirement that $\phi\in V$ then 
\begin{align}
   \label{probability first H} 2\inner{\mathcal{A}(t,\boldsymbol{\phi})}{\boldsymbol{\phi}}_U + \sum_{i=1}^\infty\norm{\mathcal{G}_i(t,\boldsymbol{\phi})}_U^2 &\leq c_tK(\boldsymbol{\phi}) -  \kappa\norm{\boldsymbol{\phi}}_H^2,\\\label{probability second H}
    \sum_{i=1}^\infty \inner{\mathcal{G}_i(t,\boldsymbol{\phi})}{\boldsymbol{\phi}}^2_U &\leq c_tK(\boldsymbol{\phi}).
\end{align}

\begin{remark}
This is a stronger assumption than Assumption \ref{assumption for prob in V}.
\end{remark}
\end{assumption}

Analagously to Appendix IV, \ref{appendix iv}, we state the relevant definitions and the resulting theorem in this context (again proved in [\cite{Goodair abs}]). Definition \ref{definitionofirregularsolution} is stated for an $\mathcal{F}_0-$ measurable $\sy_0:\Omega \rightarrow U$.

\begin{definition} \label{definitionofirregularsolution}
A pair $(\sy,\tau)$ where $\tau$ is a $\mathbb{P}-a.s.$ positive stopping time and $\sy$ is a process such that for $\mathbb{P}-a.e.$ $\omega$, $\sy_{\cdot}(\omega) \in C\left([0,T];U\right)$ and $\sy_{\cdot}(\omega)\mathbf{1}_{\cdot \leq \tau(\omega)} \in L^2\left([0,T];H\right)$ for all $T>0$ with $\sy_{\cdot}\mathbf{1}_{\cdot \leq \tau}$ progressively measurable in $H$, is said to be a $U$-valued local strong solution of the equation (\ref{thespde}) if the identity
\begin{equation} \label{identityindefinitionoflocalsolutionH}
    \sy_{t} = \sy_0 + \int_0^{t\wedge \tau} \mathcal{A}(s,\sy_s)ds + \int_0^{t \wedge \tau}\mathcal{G} (s,\sy_s) d\mathcal{W}_s
\end{equation}
holds $\mathbb{P}-a.s.$ in $X$ for all $t \geq 0$.
\end{definition}

\begin{definition} \label{H valued maximal definition}
A pair $(\sy,\Theta)$ such that there exists a sequence of stopping times $(\theta_j)$ which are $\mathbb{P}-a.s.$ monotone increasing and convergent to $\Theta$, whereby $(\sy_{\cdot \wedge \theta_j},\theta_j)$ is a $U-$valued local strong solution of the equation (\ref{thespde}) for each $j$, is said to be an $H-$valued maximal strong solution of the equation (\ref{thespde}) if for any other pair $(\py,\Gamma)$ with this property then $\Theta \leq \Gamma$ $\mathbb{P}-a.s.$ implies $\Theta = \Gamma$ $\mathbb{P}-a.s.$.
\end{definition}

\begin{definition} \label{definition unique}
A $U-$valued maximal strong solution $(\sy,\Theta)$ of the equation (\ref{thespde}) is said to be unique if for any other such solution $(\py,\Gamma)$, then $\Theta = \Gamma$ $\mathbb{P}-a.s.$ and for all $t \in [0,\Theta)$, \begin{equation} \nonumber\mathbb{P}\left(\left\{\omega \in \Omega: \sy_{t}(\omega) =  \py_{t}(\omega)  \right\} \right) = 1. \end{equation}
\end{definition}

The theorem below is stated on the condition that the assumptions of this subsection and \ref{appendix iv} are met.

\begin{theorem} \label{theorem2}
For any given $\mathcal{F}_0-$ measurable $\sy_0:\Omega \rightarrow U$, there exists a unique $U-$valued maximal strong solution $(\sy,\Theta)$ of the equation (\ref{thespde}). Moreover at $\mathbb{P}-a.e.$ $\omega$ for which $\Theta(\omega)<\infty$, we have that \begin{equation}\label{blowupproperty2}\sup_{r \in [0,\Theta(\omega))}\norm{\sy_r(\omega)}_U^2 + \int_0^{\Theta(\omega)}\norm{\sy_r(\omega)}_H^2dr = \infty.\end{equation}
\end{theorem}

\end{document}